\documentclass[12pt,a4paper]{amsart}
\usepackage{amsfonts}
\usepackage{amssymb}
\usepackage{amsmath}
\usepackage{amsthm}
\usepackage{cite}
\usepackage{enumitem}
\usepackage{tikz}
\usepackage{subfigure}
\usepackage{color}

\theoremstyle{plain}
\newtheorem{thm}{Theorem}

\newtheorem{lem}{Lemma}
\newtheorem{prop}{Proposition}
\newtheorem{cor}{Corollary}
\newtheorem{bem}{Remark}

\newcommand{\vecII}[2]{
\ensuremath{
\begin{pmatrix}
#1 \\ #2 \\
\end{pmatrix}}}

\newcommand{\matII}[4]{
\ensuremath{ 
\begin{pmatrix}  
#1 & #2 \\
#3 & #4 \\
\end{pmatrix}}}

\providecommand{\sm}{\setminus}
\providecommand{\N}{\mathbb{N}} 
\providecommand{\R}{\mathbb{R}}

\providecommand{\Hh}{\mathbb{H}}
\providecommand{\Z}{\mathbb{Z}} 
\providecommand{\C}{\mathbb{C}}

\providecommand{\skp}[2]{\langle#1,#2\rangle}
\DeclareMathOperator{\cn}{cn}
\DeclareMathOperator{\sn}{sn}
\DeclareMathOperator{\dn}{dn}

\DeclareMathOperator{\sign}{sign}

\DeclareMathOperator{\Arcosh}{Arcosh}

\DeclareMathOperator{\sech}{sech}
\renewcommand{\qed}{\hfill $\Box$}

\setitemize{itemsep=+2pt}
\setenumerate{itemsep=+2pt}

\begin{document}

\allowdisplaybreaks

\title[Willmore surfaces of revolution]{Explicit formulas, symmetry and symmetry breaking for
Willmore surfaces of revolution}

\author{Rainer Mandel}
\address{R. Mandel \hfill\break
Karlsruhe Institute of Technology \hfill\break
Institute for Analysis \hfill\break
Englerstra{\ss}e 2 \hfill\break
D-76131 Karlsruhe, Germany}
 
\email{Rainer.Mandel@kit.edu}
\date{\today}


\begin{abstract} 
  In this paper we prove explicit formulas for all Willmore surfaces of revolution and demonstrate their use
  in the discussion of the associated Dirichlet boundary value problems. It is shown by an
  explicit example that symmetric Dirichlet boundary conditions do in general not entail the symmetry of the
  surface. In addition we prove a symmetry result for a subclass of Willmore surfaces satisfying symmetric
  Dirichlet boundary data. 
\end{abstract}

\maketitle


\section{Introduction}

  In this paper we discuss a class of boundary value problems for immersed but in general nonembedded Willmore
  surfaces of revolution $\Sigma\subset\R^3$. A Willmore surface is, by definition, a critical point of the
  Willmore functional 
  $$
    \mathcal{W}(\Sigma) = \int_\Sigma H^2 \,d\mu
  $$
  where $H=\frac{1}{2}(k_1+k_2)$ is the mean curvature, $k_1,k_2$ are the principal curvatures and where  
  $d\mu$ denotes the associated volume form of the surface. It it known \cite{LanSing_Curves_in_the}
  that Willmore surfaces of revolution are generated by elastic curves in the hyperbolic
  plane $\Hh$ (called elasticae), which, by definition, are critical for the total squared curvature
  functional.
  About 30 years ago these curves were classified by Langer and Singer \cite{LanSing_the_total} so that a
  complete qualitative description of Willmore surfaces of revolution is available. One problematic feature
  of this classification is that a quantitative description of elastic curves turns out to be rather difficult
  since they are characterized by their curvature functions only in terms of the arclength parameter, which a priori is
  unknown. On the other hand, such quantitative information is needed when studying boundary value problems so
  that the above-mentioned classification result has not been of much use so far. An alternative approach was
  followed by Bergner, Dall'Acqua, Eichmann, Grunau and others (see for instance
  \cite{BerDalFro_Symmetric,DalAcDecGru_Classical_solutions,DalAcFroGruSch_Symmetric,Eich_symmetry,Eich_nonuniqueness})
  who considered the corresponding boundary value problems under the additional assumption that the generating
  curve is a graph between its prescribed endpoints. In such a way the unknown piece of a curve is
  determined by a single positive profile function parametrized over its axis of revolution. The Willmore
  equation 
  $$
    -\Delta_\Sigma H = 2H(H^2-K)
  $$ 
  for such graphical surfaces of revolution then simplifies to a rather complicated nonlinear ODE of fourth
  order for the profile function (see for instance (2.5) in \cite{DalAcFroGruSch_Symmetric}) coming with
  appropriate boundary conditions. Using this approach several results for mostly symmetric (i.e. even)
  profile functions satisfying symmetric boundary value problems could be proved whereas the existence of
  nonsymmetric or nongraphical solutions for symmetric or nonsymmetric boundary value problems is almost
  completely open up to now. Our intention is to close this gap. Somewhat surprisingly, in the course of our
  study it became apparent that the above-mentioned characterization of elastic curves by Langer and
  Singer may be used to find \r{rather complicated} explicit formulas for any given elastic curve and thus
  for all Willmore surfaces of revolution. Here, the word explicit refers to the fact that the formulas only involve Jacobi's
  elliptic functions and theta functions, elliptic integrals  etc. In particular, this allows
  us to treat arbitrary boundary conditions by explicit means and to analyze the shape of each solution.
  For instance we find in Theorem~\ref{thm:Symmetry_breaking} that there are symmetric boundary conditions
  for the Willmore equation with nonsymmetric solutions. In order to describe these solutions in a better way
  let us recall the following fundamental result due to Langer and Singer. 

  \begin{thm}[cf. \cite{LanSing_Curves_in_the}, Theorem 3] \label{thm:LanSin}
    Let $\gamma$ be a regular closed curve in $\Hh$ which is critical for the elastic energy functional. Then
    either $\gamma$ is the $m$-fold cover of the ball for some $m$, or $\gamma$ is a member of the family
    of solutions $\gamma_{m,n}$ having the following description: if $m>1$ and $n$ are integers satisfying
    $1<\frac{2m}{n}<\sqrt{2}$ there is (up to congruence) a unique curve $\gamma_{m,n}$, which closes up in
    $n$ periods of its curvature $\kappa=\kappa_0 \cn^2(rs,p)$ while making $m$ orbits about the fixed point
    of the associated rotation field $J$. Further, $\gamma_{m,n}$ oscillates between a pair of invariant
    circles of $J$ with the $n$ maxima of $\kappa$ equally spaced around the outer circle and the $n$ minima
    of $\kappa$ equally spaced around the inner circle. Also, $\gamma_{m,n}$ has $n(m-1)$ points of
    selfintersections.
  \end{thm} 
  
  Notice that $\kappa=\kappa_0 \cn^2(rs,p)$ is a typo in \cite{LanSing_Curves_in_the} since the curvature
  function is basically a $\dn$-function (cf. \eqref{eq:def_kappa_orbitlike}), as we will see later. 
  Following \cite{LanSing_Curves_in_the} (not necessarily closed) elastic curves generated by such
  curvature functions will be called orbitlike.  The proof of the above theorem is in parts provided in the
  paper \cite{LanSing_the_total}. Computations of the winding number and of the number of selfintersections
  are, however, missing, so that our proofs provided later may be of interest. 
  
  \medskip
  
  Our main concern is the study of the Dirichlet problem for Willmore surfaces of
  revolution. In a very general sense the Dirichlet problem  was solved by
  Sch\"atzle \cite{Schaetzle_WillmoreBVP} using variational methods and geometric measure theory. Due to the generality of his approach there is not
  much information about the obtained solutions. Within the class of surfaces of revolution, however, the
  possible shapes of solutions to the Dirichlet problem are, thanks to the classification result mentioned
  above, quite well understood. The surface then takes the form
  \begin{equation} \label{eq:Sigma_gamma}
    \Sigma_\gamma = 
    \Big\{ (\gamma_1(s),\gamma_2(s)\cos(t),\gamma_2(s)\sin(t)) : s\in I, t\in [0,2\pi) \Big\}
  \end{equation} 
  where $I\subset\R$ is some interval und $\gamma$ is an elastic curve in the hyperbolic plane. The problem
  then is to find a Willmore surface with preassigned initial or end positions and angles, i.e. for given
  $(A_1,A_2),(B_1,B_2)$ in the hyperbolic plane and $\phi_A,\phi_B\in\R$ one has to find an elastic curve
  $\gamma$ with (unknown) hyperbolic length $L$ such that the following holds:
  \begin{equation} \label{eq:DirichletBVP}
    \mathcal{W}'(\Sigma_\gamma)=0,\quad  \gamma(0)=(A_1,A_2),\quad \gamma(L)=(B_1,B_2),\quad
    e^{i\phi(0)}=e^{i\phi_A},\quad e^{i\phi(L)}=e^{i\phi_B}. 
  \end{equation}
  Here, we formally write $\mathcal{W}'(\Sigma_\gamma)=0$ in place of the Willmore equation and the
  angle function $\phi$ is given by $\gamma'=|\gamma'|e^{i\phi}$.
  We stress that we do not impose $\gamma$ to be simple so that the generated surface $\Sigma_\gamma$ may not
  be embedded. In particular, in contrast to most of the earlier works on the Dirichlet problem for
  Willmore surfaces of revolution, we do not assume that $\Sigma_\gamma$ has a global parametrization as a
  graph (which in fact only occurs in very special situations). First results concerning non-graphical
  Willmore surfaces of revolution can be found in \cite{Eich_nonuniqueness} (Theorem~8.1) or
  \cite{EiGr_existence} (Theorem~1.1). One example for a result in the graphical setting
  can be found in the papers \cite{DalAcDecGru_Classical_solutions,DalAcFroGruSch_Symmetric}: For
  $A_1=-B_1,A_2=B_2$ and $\phi_A=-\phi_B \in [0,\pi/2)$ the Dirichlet problem has a graphical solution which,
  by Theorem~1.1 in \cite{Eich_nonuniqueness}, may in general fail to be unique.  
  Our contribution to the Dirichlet problem is the following: In a first result we use our explicit formulas
  in order to reduce the Dirichlet problem \eqref{eq:DirichletBVP} to a system of two equations for two unknowns. 
  For the sake of shortness we only consider the Dirichlet problem for orbitlike elasticae but the
  analogous analysis may be done for the wavelike ones as well.  The precise statement of this result
  requires the notation from the following sections so that we only provide it near the end of this paper,
  see~Theorem~\ref{thm:DirProblemI}. In its full generality we can not solve this system, but for particular
  boundary data it is possible to say more about the solution set. For instance we show that for certain
  symmetric boundary data orbitlike solutions must be symmetric themselves. Our result is the following:

  \begin{thm}\label{thm:Symmetry}
    Assume $A_1=-B_1\neq 0, A_2=B_2>0,\phi_A+\phi_B\in 2\pi\Z$ such that
    $$
      \pm\big(\frac{A_2}{A_1}\sin(\phi_A) + \cos(\phi_A)\big) \notin (0,2).
    $$
    Then all positively $(\pm=+)$ respectively negatively $(\pm=-)$ oriented orbitlike solutions
    of~\eqref{eq:DirichletBVP} are symmetric.
  \end{thm}
  
  Here, a solution of \eqref{eq:DirichletBVP} is called symmetric if we have
  $$
    \gamma_1(L/2+s)=-\gamma_1(L/2-s),\quad \gamma_2(L/2+s)=\gamma_1(L/2-s)
    \qquad\text{for all }s\in [0,L/2]
  $$ 
  and it is called nonsymmetric otherwise. Positive / negative orientation means that the
  hyperbolic curvature of the elastica is assumed to be positive / negative at some point (and hence
  everywhere since the $\dn$-function is zero-free). In particular, negatively oriented elasticae are
  symmetric when $A_1=-B_1=-1,A_2=B_2=\alpha>0,\phi_A=\phi_B=0$, which corresponds to the boundary data studied by
  Koeller and Eichmann in Theorem~1.1~\cite{Eich_symmetry}. They were able to prove the symmetry
  of one particular solution of this boundary value problem, namely of the graphical least energy solution.
  Indeed, the energy of this solution is smaller than $4\pi$ according to Theorem~1.1 in
  \cite{DalAcDecGru_Classical_solutions} whereas nonsymmetric solutions (if they exist)  have
  Willmore-energy $>4\pi$ by Theorem~3.9 in \cite{Eich_symmetry}. A link between these results is
  unfortunately missing up to now so that have to leave this issue open. Our next theorem shows that
  the assumption  $A_1=-B_1\neq 0$ from Theorem~\ref{thm:Symmetry} is in fact needed for a symmetry result
  because we can construct a class of nonsymmetric solutions with $A_1=-B_1=0$. 
  
  \begin{thm} \label{thm:Symmetry_breaking}
    There are boundary data $A_1=-B_1,A_2=B_2>0,\phi_A=-\phi_B$ such that \eqref{eq:DirichletBVP} has
    uncountably many nonsymmetric solutions.
  \end{thm}
  
  The proof of Theorem~\ref{thm:Symmetry_breaking} reveals further properties of the constructed solutions. 
  We will see that one can choose $A_1=B_1=0$, $A_2=B_2>0$ and $\phi_A=-\phi_B=0$ and that the conformal
  class of every $\gamma_{m,n}$ $(m,n\in\N)$ described in Theorem~\ref{thm:LanSin} contains uncountably many
  nonsymmetric representatives that solve the Dirichlet problem~\eqref{eq:DirichletBVP}. Examples are
  illustrated in Figure~\ref{Fig:Symmetry_breaking}. 
  
  \medskip
  
  The plan of this paper is the following: In section~\ref{sec:geometry} we briefly review some facts about
  the hyperbolic plane $\Hh$ and curves in it that are parametrized by hyperbolic arclength. Next, in
  section~\ref{sec:hyperbolic_elasticae}, we collect some material about elasticae in $\Hh$
  and their (hyperbolic, geodesic) curvature functions $\kappa$. We show that the hyperbolic distance of an
  elasticae to any given point in the hyperbolic plane satisfies a simple linear ordinary differential
  equation of second order with coefficients depending only on $\kappa,\kappa'$, see 
  Proposition~\ref{prop:ode_Z}. 
  Given that the definition of $\kappa,\kappa'$ involves Jacobi's elliptic function we then continue by
  recalling in section~\ref{sec:Jacobi} some properties of these functions which will be important later on.
  The same will be done for elliptic theta functions, elliptic integrals and others which we will
  use.  In the next two sections we solve the differential equation mentioned
  above  and derive from it the explicit formulas for the elasticae we are interested in. The first of these sections is
  devoted to orbitlike elasticae where the hyperbolic curvature function is built from Jacobi's
  $\dn$-function while the second deals with wavelike elasticae where the curvature is
  built from Jacobi's $\cn$-function. Additionally, we provide new short proofs of the properties mentioned in
  Theorem~\ref{thm:LanSin}. Next, in section~\ref{sec:BVPs} we use the explicit formulas to solve the
  Dirichlet problem for elasticae in $\Hh$ in its full generality and, as an application, we prove our
  symmetry-related results in Corollary~\ref{cor:Symmetry}  and Corollary~\ref{cor:Symmetry_breaking}.
  In the Appendix we provide the proofs of some technical results used earlier in the paper.

\section{Geometry of the hyperbolic plane} \label{sec:geometry}
   
  Willmore surfaces of revolution are known to be generated by elastic curves in the hyperbolic plane
  $\mathbb{H}=\{(x_1,x_2)\in\R^2: x_2>0\}$, see for instance p.532 in \cite{LanSing_Curves_in_the}. The
  metric is given by $\frac{dx_1^2+dx_2^2}{x_2^2}$ 
  and the distance between 
  two points $(x_1,x_2),(y_1,y_2)\in \Hh$ is 
  \begin{equation}\label{eq:metric}
      d_{\Hh}((x_1,x_2),(y_1,y_2)) 
      =  \Arcosh\Big(1+\frac{(x_1-y_1)^2+(x_2-y_2)^2}{2x_2y_2}\Big). 
  \end{equation}
  Regular curves $\gamma$ in $\Hh$ are therefore parametrized by hyperbolic arclength if
  $(\gamma_1')^2+(\gamma_2')^2\equiv \gamma_2^2$ and in fact all curves studied in this paper will be written
  down in this parametrization. The advantage compared to other parametrizations is that the
  hyperbolic curvature function has a nice explicit expression as we will see later. 
  On the other hand, compared to the Euclidean case, it is harder to reconstruct a curve from its curvature
  function and the initial data. While this can be done explicitly in the Euclidean case, it is not clear a
  priori how this can be achieved in the hyperbolic plane. Introducing the angle  function $\phi$ via 
  $\gamma_1' = \gamma_2 \cos(\phi),\gamma_2' = \gamma_2 \sin(\phi)$ (such that $\gamma'=|\gamma'|e^{i\phi}$) 
  one finds that the hyperbolic (or geodesic) curvature function is given by
    \begin{equation} \label{eq:hyperbolic_curvature}
      \kappa
      := - \frac{\gamma_2^2}{\gamma_2'} \Big(\frac{\gamma_1'}{\gamma_2^2}\Big)' 
      = \phi'+\cos(\phi).
    \end{equation}
    Here, we chose this sign convention for $\kappa$ in order to be consistent with (2.3) in
    \cite{DalAcDecGru_Classical_solutions}. 
     So, a curve $\gamma$ may be recovered from its  curvature function $\kappa$ by solving the ODE system
    \begin{align} \label{eq:ODE_solution}
      \gamma_1' = \gamma_2 \cos(\phi),\quad
      \gamma_2' = \gamma_2 \sin(\phi),\quad 
      \phi' + \cos(\phi)  = \kappa. 
    \end{align}
    This ODE system is invariant under M\"obius transformations and in particular under the continuous
    isometries of the hyperbolic plane where only the initial data $\gamma_1(0),\gamma_2(0),\phi(0)$ changes.
    Identifying $\Hh$ with $\{x+iy:y>0\}$ and writing $\gamma=\gamma_1+i\gamma_2$ these isometries take the
    following form:
    \begin{itemize}
      \item[(i)] horizontal translation: $\gamma\mapsto \gamma+a$ for $a\in\R$,
      \item[(ii)] dilation: $\gamma \mapsto b\gamma$ for $b>0$,
      \item[(iii)] hyperbolic rotation: $\gamma\mapsto 
      \tfrac{\cos(\theta/2)\gamma+\sin(\theta/2)}{-\sin(\theta/2)\gamma+\cos(\theta/2)}
      $ for $\theta\in\R$.
    \end{itemize}
    The missing isometries are discrete and come from reflections, i.e. horizontal reflection and inversion
    with respect to the unit sphere. Notice that both cause a sign change in the curvature function.

  \section{Classification of elasticae in the hyperbolic plane}  \label{sec:hyperbolic_elasticae}
  
  In this section we \r{briefly} recall the classification of (free) elasticae in $\Hh$ from
  \cite{LanSing_the_total}. So let $\gamma=(\gamma_1,\gamma_2):\R\to\Hh$ be such a one with hyperbolic
  curvature function $\kappa$. From section 2 in \cite{LanSing_the_total}  we get that $\kappa$
  satisfies
  \begin{equation} \label{eq:ODE_kappa}
      -\kappa''+\kappa =\frac{1}{2}\kappa^3,\qquad\quad
      \mu := -(\kappa')^2+ \kappa^2 - \frac{1}{4}\kappa^4 \in\R. 
  \end{equation}
  Solutions of this ODE are well-known. From Table (2.7)(c) in~\cite{LanSing_the_total} for $G=-1$ we find
  that $\kappa$ is, up to multiplication by $-1$, one of the following functions:
  \begin{align*}
	&(i)\;\;\quad \kappa\equiv \sqrt{2} &&\mu = 1\\
	&(ii)\;\quad \kappa(s) =  \frac{2}{\sqrt{2-k^2}}\dn\Big(\frac{s+a}{\sqrt{2-k^2}},k\Big)\quad(0<k<1)
	&& \mu =\frac{4(1-k^2)}{(2-k^2)^2}, \\ 
	&(iii)\quad \kappa\equiv 0 \quad\text{or}\quad \kappa(s)= 2\sech(s+a) &&\mu= 0, \\
	&(iv)\;\quad \kappa(s)= \frac{2k}{\sqrt{2k^2-1}} \cn\Big(\frac{s+a}{\sqrt{2k^2-1}},k\Big)\quad
	  (\frac{1}{\sqrt 2}<k<1)
	  && \mu = - \frac{4k^2(1-k^2)}{(2k^2-1)^2}.
  \end{align*}
  We will recall the basic properties of $\cn,\dn$ and related functions in the next section. 
  We mention that Jacobi's elliptic functions also appear in the study of elastic knots
  \cite{LanSing_Knotted}, \cite{IvSin_knot_types}. In the following we concentrate on elasticae generated by curvature functions given by the above formulas. 
  Notice that the analysis of an elastica $\gamma=(\gamma_1,\gamma_2)$ with curvature $-\kappa$ may be
  reduced to the analysis of an elastica $\gamma=M\circ (-\gamma_1,\gamma_2)$ having curvature $\kappa$ 
  where $M$ is any of the M\"obius transformations (i),(ii),(iii) from the last section.  Elasticae generated
  by the curvature functions listed above are plotted in Figure~\ref{Fig:plots_of_curves} with the help of
  MAPLE and the explicit formulas for the elasticae that we are going to prove in
  Theorem~\ref{thm:explicit_formula_orbitlike} and Theorem~\ref{thm:explicit_formula_wavelike}.
  These figures motivate the names (i) circular, (ii) orbitlike, (iii)
  geodesic / asymptotically geodesic or (iv) wavelike, respectively.
  
  \begin{figure}[!htb] 
    \centering 
    \subfigure[Orbitlike elastica, $k=0.8$]{
      \includegraphics[scale=.38]{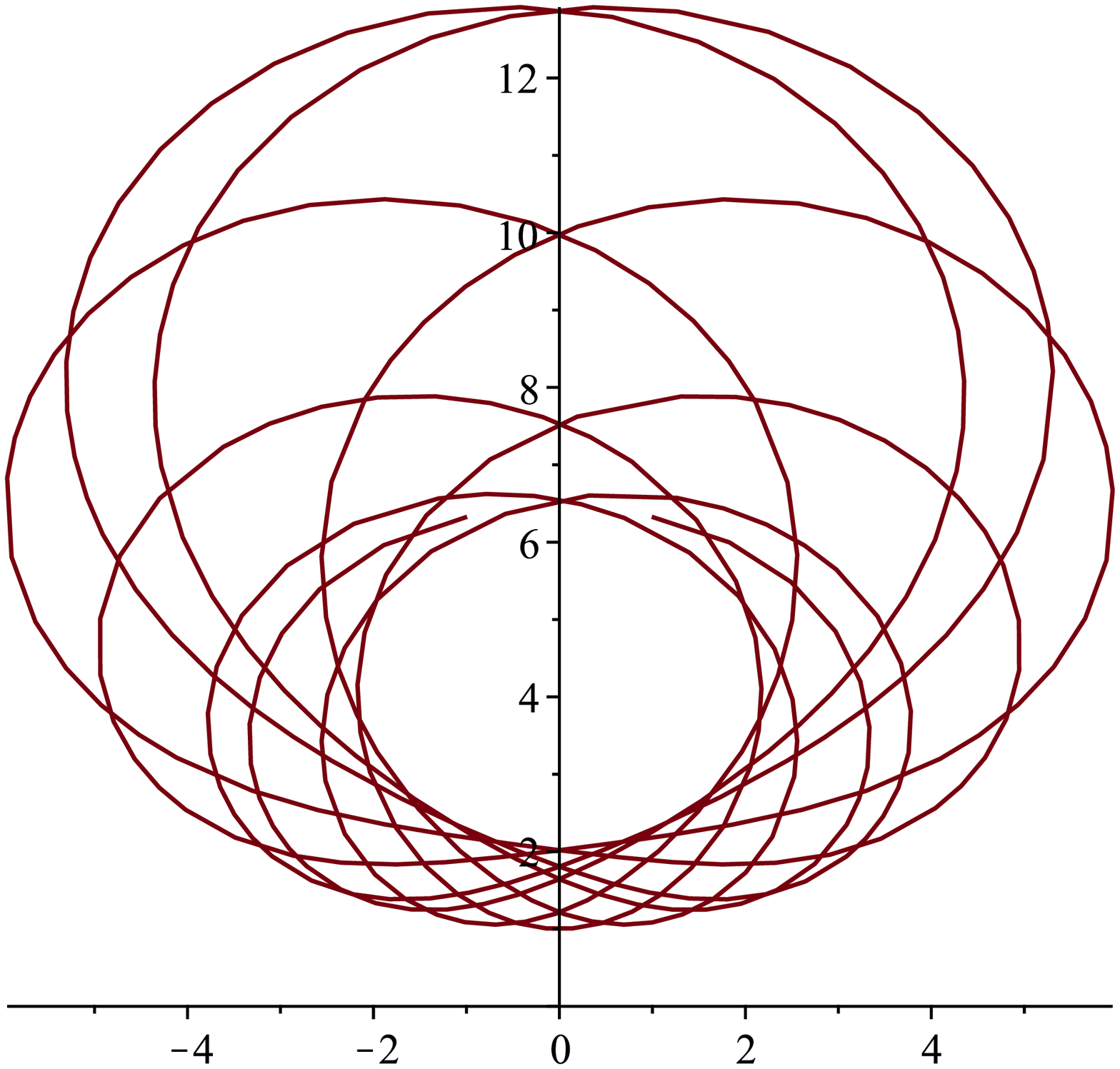} 
    }
    \;
    \subfigure[Wavelike elastica, $k=0.8$]{ 
      \includegraphics[scale=.38]{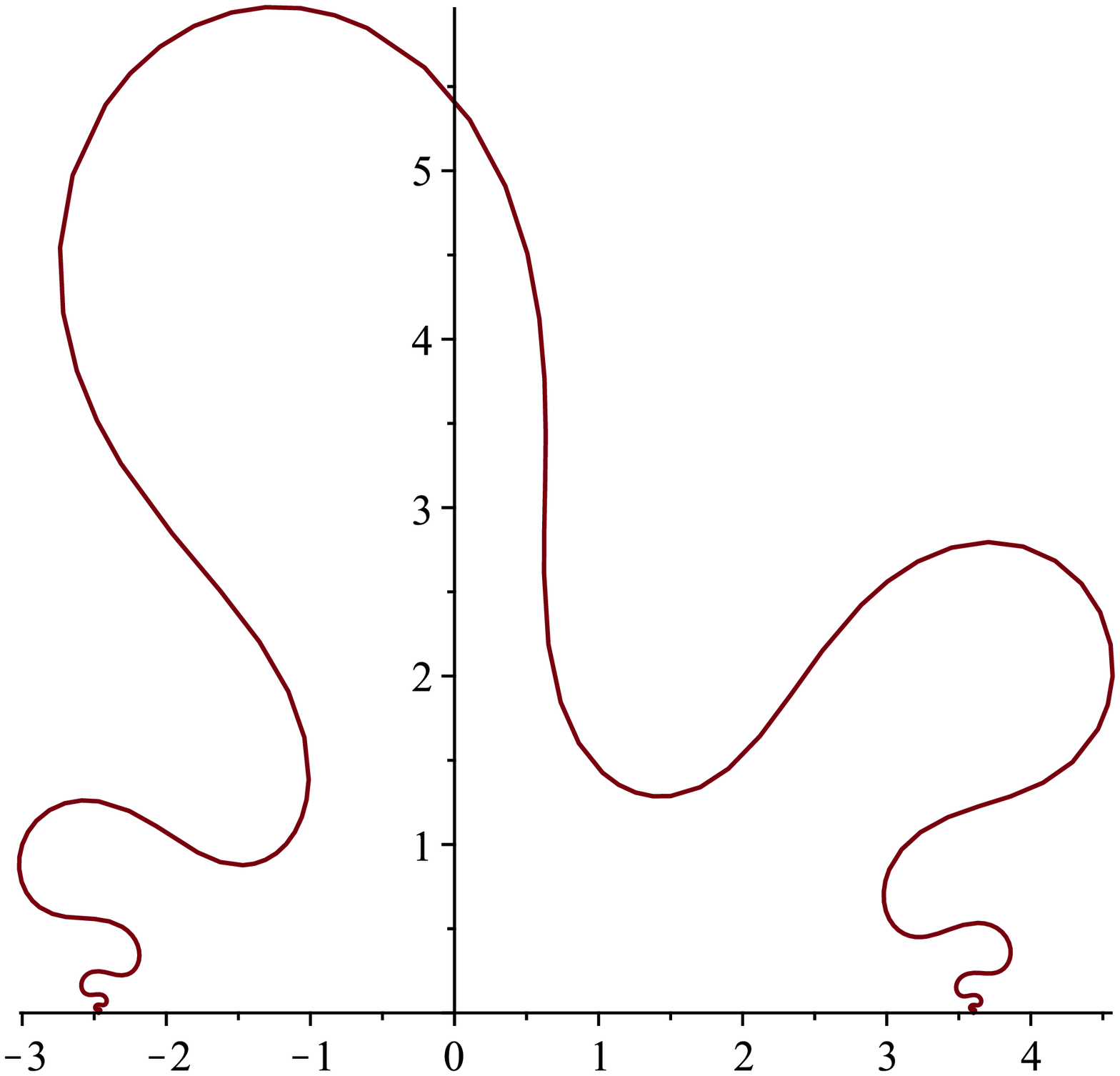} 
    } 
    \caption{Elasticae in $\Hh$}
     \label{Fig:plots_of_curves} 
  \end{figure} 
   
   As we will see below, the circular and (asymptotically) geodesic can be determined explicitly by elementary
   means so that our main concern will be the discussion of orbitlike and wavelike elasticae. This includes
   the derivation of explicit formulas, which turns out to be a nontrivial task. The key to achieve this goal
   is a thorough discussion of the function 
  \begin{equation} \label{eq:def_Z}
    Z(s;P) = \frac{(\gamma_1(s)-P_1)^2+(\gamma_2(s)-P_2)^2}{2P_2\gamma_2(s)} \qquad\text{for all }
    P=(P_1,P_2)\in 
    \C \text{ with }P_2\neq 0,
  \end{equation}
  which for $P\in\Hh$  is a conformal invariant due to $d_{\Hh}(\gamma(s),P) = \Arcosh(1+Z(s;P))$,
  see~\eqref{eq:metric}. We mention that in the context of wavelike elasticae we will study
  $Z(\cdot;P)$ for complex numbers $P\in\C$ with $P_2\neq 0$ despite the fact that it has a clear geometrical
  meaning only for $P\in \Hh$. 
  Now we  recall the relevant information about circular and
  (asymptotically) geodesic elasticae.
  
  \medskip
  
  The circular elasticae given by $\kappa\equiv \sqrt{2}$  generate the conformal class of the
  Clifford-Torus, which plays a special role in the study of Willmore surfaces. Notably the Willmore Conjecture says that the infimum of the Willmore energy in
  the class of immersed tori in $\R^3$ is achieved by the Clifford-Torus. It has been an open problem for
  about fifty years and its proof due to Marques and Neves \cite{MarNev_Wilconj} must be considered as
  a breakthrough in the study of Willmore surfaces. 
  Reduced versions of this conjecture
  were proved in \cite{LanSing_Curves_in_the} (Corollary~5) within the class of Willmore surfaces of revolution
  and in the larger class of canal surfaces of torus type in \cite{HerJerPin_Beweis}. 
  One explicit parametrization of the Clifford torus is given by
	\begin{align*}
 	  \gamma(s) = \Big( \frac{(1+\sqrt{2})\sin(s)}{\sqrt{2}+\cos(s)},
 	  \frac{1+\sqrt{2}}{\sqrt{2}+\cos(s)}\Big),\quad
 	  e^{i\phi(s)} =  \frac{1+\sqrt{2}\cos(s)+i\cdot \sin(s)}{\sqrt{2}+\cos(s)}.
	\end{align*}
  This precise realization of the Clifford torus corresponds to solving \eqref{eq:ODE_solution} for
  $\gamma(0)=(0,1),\phi(0)=0$ and $\kappa\equiv \sqrt{2}$. Geodesics in $\Hh$ are known to be either vertical
  lines or half circles in dependence of the initial position and angle. 
  If $\phi(0)\in (\pm \frac{\pi}{2}+2\pi)\Z$ then the geodesic through $(0,1)$ is vertical and given by
  $\gamma(s)=(0,e^{\pm s}),\phi(s)=\phi(0)$.
  Otherwise, for instance in case $\phi(0)=0$, we find the halfcircles 
  $$
     \gamma(s) = (\tanh(s),\sech(s)),\qquad e^{i\phi(s)}= \sech(s) - i\cdot \tanh(s).
   $$
   The third exceptional family of elasticae is the asymptotically geodesic one which corresponds to the
   conformal class of the catenoid. By definition, it is the graph of the $\cosh$-function and therefore   
   $$
      \gamma(s) = (s,\cosh(s)),\qquad 
      e^{i\phi(s)} = \sech(s)+i\cdot \tanh(s).  
   $$
   In the next two sections we show that both orbitlike and wavelike elasticae admit 
   explicit solutions as well, which, however, are much more involved than the ones above. As mentioned
   above, our analysis is based on an investigation of the distance function $Z$ from \eqref{eq:def_Z}
   associated with a given elastica $\gamma$. The following result provides a useful relation between the
   curvature of $\gamma$ and $Z$.
   \r{The proof consists of elementary calculations, so we only give the main steps.}
   
  \begin{prop}\label{prop:ode_Z}
    Let $\gamma:\R\to\Hh$ be parametrized by hyperbolic arclength and $Z$ be
    given by \eqref{eq:def_Z}. Then, for all $P\in \C^2$ with $P_2\neq 0$, the function
    $Z(\cdot;P)$ satisfies
    \begin{equation}\label{eq:ODE_Z}
      \kappa Z''-2\kappa'Z'+\kappa(Z+1)\equiv 2\mu C
    \end{equation}
    where $C\in\C$ is, for any $s_0\in\R$,  given by 
     \begin{align} \label{eq:def_C}
       \begin{aligned}
       C &=
       \frac{P_2^2+(\gamma_1(s_0)-P_1)^2}{2P_2\gamma_2(s_0)}\cdot\Big(\frac{\kappa(s_0)}{\mu}+\frac{\kappa'(s_0)}{\mu}\sin(\phi(s_0))
       - \frac{\kappa(s_0)^2}{2\mu}\cos(\phi(s_0))\Big) \\
        & + \frac{\gamma_2(s_0)}{2P_2}\cdot \Big(\frac{\kappa(s_0)}{\mu}-\frac{\kappa'(s_0)}{\mu}\sin(\phi(s_0))
       + \frac{\kappa(s_0)^2}{2\mu}\cos(\phi(s_0))\Big) \\
       &+ \frac{\gamma_1(s_0)-P_1}{P_2}\cdot \Big(
       -\frac{\kappa(s_0)^2}{2\mu}\sin(\phi(s_0))-\frac{\kappa'(s_0)}{\mu}\cos(\phi(s_0)) \Big).
       \end{aligned} 
     \end{align}
  \end{prop}
  \begin{proof}
    For convenience we write $Z:=Z(\cdot;P)$. Differentiating $Z$ gives the formula
    \begin{equation} \label{eq:formula_Zprime}
      Z' 
       = \frac{\gamma_1-P_1}{P_2} \cdot \cos(\phi)  +  
      \frac{\gamma_2^2-P_2^2-(\gamma_1-P_1)^2}{2P_2\gamma_2}\cdot   \sin(\phi).
    \end{equation}
    Differentiating once more and using $\phi'+\cos(\phi)=\kappa$ gives
    \begin{align*}
      Z''
      &= \phi'\cdot \Big( - \frac{\gamma_1-P_1}{P_2}\cdot \sin(\phi)
      + \frac{\gamma_2^2-P_2^2-(\gamma_1-P_1)^2}{2P_2\gamma_2}\cdot\cos(\phi)
      \Big)  \\
      &\quad + \frac{\gamma_1'}{P_2}\cos(\phi) +
      \frac{2\gamma_2^2\gamma_2'-2(\gamma_1-P_1)\gamma_2\gamma_1' - (\gamma_2^2-P_2^2-(\gamma_1-P_1)^2)\gamma_2'   }{2P_2\gamma_2^2} \sin(\phi)
      \\
      &= (\kappa-\cos(\phi)) \cdot \Big( - \frac{\gamma_1-P_1}{P_2}\cdot \sin(\phi)
      + \frac{\gamma_2^2-P_2^2-(\gamma_1-P_1)^2}{2P_2\gamma_2}\cdot\cos(\phi)
      \Big) \\
      &\quad + \frac{\gamma_2}{P_2}\cdot \cos^2(\phi)
      + \frac{(\gamma_2^2+P_2^2+ (\gamma_1-P_1)^2) \sin(\phi)
      -2(\gamma_1-P_1)\gamma_2\cos(\phi)}{2P_2\gamma_2} \cdot \sin(\phi)
      \\
      &= (\kappa-\cos(\phi)) \cdot \Big( - \frac{\gamma_1-P_1}{P_2}\cdot \sin(\phi)
      + \frac{\gamma_2^2-P_2^2-(\gamma_1-P_1)^2}{2P_2\gamma_2}\cdot\cos(\phi)
      \Big) \\
      &\quad + \frac{\gamma_2}{P_2}\cdot \cos^2(\phi) +
      \frac{\gamma_2^2+P_2^2+(\gamma_1-P_1)^2}{2P_2\gamma_2} \cdot \sin^2(\phi) -\frac{\gamma_1-P_1}{P_2}\cdot
      \sin(\phi)\cos(\phi)    \\
      &= \frac{\gamma_2^2+P_2^2+(\gamma_1-P_1)^2}{2P_2\gamma_2} - \kappa w  \\
      &= Z+1-\kappa w.
    \end{align*}
    Here, the auxiliary function $w$ is given by
    $$
      w
     : = \frac{\gamma_1-P_1}{P_2} \cdot \sin(\phi) - 
      \frac{\gamma_2^2-P_2^2-(\gamma_1-P_1)^2}{2P_2\gamma_2}
      \cdot \cos(\phi). 
    $$
    From the identity
    \begin{align*}
      w'
      &= \Big(\frac{\gamma_1-P_1}{P_2} \cdot \cos(\phi)  +
      \frac{\gamma_2^2-P_2^2-(\gamma_1-P_1)^2}{2P_2\gamma_2}\cdot   \sin(\phi) \Big) \phi' \\
      &\;+ \Big( \frac{\gamma_2}{P_2}\cdot \sin(\phi)  -
      \Big(\frac{\gamma_2^2-P_2^2-(\gamma_1-P_1)^2}{2P_2\gamma_2}\Big)'\Big) \cdot   \cos(\phi) \\
      &= \Big(\frac{\gamma_1-P_1}{P_2} \cdot \cos(\phi)  +
      \frac{\gamma_2^2-P_2^2-(\gamma_1-P_1)^2}{2P_2\gamma_2}\cdot   \sin(\phi) \Big) \phi' \\
      &\;+ \Big(\frac{\gamma_1-P_1}{P_2} \cdot \cos(\phi)  +
      \frac{\gamma_2^2-P_2^2-(\gamma_1-P_1)^2}{2P_2\gamma_2}\cdot   \sin(\phi) \Big) \cos(\phi) \\
      &= Z' (\phi'+\cos(\phi)) \\
      &= \kappa Z'
    \end{align*}
     we get
   \begin{align*}
     (\kappa Z''-2\kappa'Z'+\kappa Z + \kappa)'
     &=  (-\kappa^2 w-2\kappa'Z'+ 2\kappa Z + 2\kappa)'   \\
     &= Z'(-2\kappa''+2\kappa) -2\kappa^2w'   + 2\kappa' (- \kappa w -Z'' +  Z+1) \\
     &\stackrel{\eqref{eq:ODE_kappa}}{=}  Z'(-2\kappa''+2\kappa-\kappa^3) \\
 	 &=   0
   \end{align*}
    so that $\kappa Z''-2\kappa'Z'+\kappa Z + \kappa$ must be constant. If we define 
    $C\in\C$ to be the complex number with $\kappa Z''-2\kappa'Z'+\kappa Z + \kappa\equiv
    2\mu C$, then we finally find for all $s_0\in\R$
    \begin{align*}
      2\mu C 
      &= \kappa(s_0) Z''(s_0)-2\kappa'(s_0)Z'(s_0)+\kappa(s_0)(Z(s_0)+1) \\
      &= 2\kappa(s_0) (Z(s_0)+1) -\kappa(s_0)^2 w(s_0)  -2\kappa'(s_0)Z'(s_0) \\
      &= \frac{P_2^2+(\gamma_1(s_0)-P_1)^2}{2P_2\gamma_2(s_0)}\cdot\Big(2\kappa(s_0) +2\kappa'(s_0) \sin(\phi(s_0))
       -  \kappa(s_0)^2 \cos(\phi(s_0))\Big) \\
        & + \frac{\gamma_2(s_0)}{2P_2}\cdot \Big(2\kappa(s_0) -2\kappa'(s_0) \sin(\phi(s_0))
       +  \kappa(s_0)^2 \cos(\phi(s_0))\Big) \\
       &+ \frac{\gamma_1(s_0)-P_1}{P_2}\cdot \Big(
       - \kappa(s_0)^2 \sin(\phi(s_0))-2\kappa'(s_0) \cos(\phi(s_0)) \Big).
    \end{align*}
  \end{proof}
  
  \begin{bem} ~
    \begin{itemize}
      \item[(a)] Differentiating \eqref{eq:ODE_Z} twice one finds $2Z^{(iv)} +  \big((3\kappa^2-4)Z'\big)' +
      (2-\kappa^2)(Z+1) =0$. This ODE is, in contrast to \eqref{eq:ODE_Z}, not singular at zeros of $\kappa$.
      In particular $Z$ is smooth even in the wavelike case where $\kappa$ has infinitely many zeros. 
      Since we will not make use of this ODE we omit the proof.
      \item[(b)] A similar result may be shown for elastic curves in the Euclidean plane that have recently
      been studied in \cite{DecGru_bvp,Man_bvp}. Such curves $\hat\gamma$, now parametrized by
      Euclidean arclength with curvature $\hat\kappa$, satisfy $\hat \kappa \hat Z''-2\hat\kappa'\hat Z'=
      const$ where $\hat Z(s):= (\gamma_1(s)-P_1)^2+(\gamma_2(s)-P_2)^2$. 
    \end{itemize}
  \end{bem}

  \section{On Jacobi's elliptic functions} \label{sec:Jacobi}

    In this section we provide the definitions and main properties of Jacobi's elliptic functions and
    related functions that we will need in this paper. The  material is taken from the book by Byrd and
    Friedman \cite{ByrFri_handbook}. The elliptic integrals of the first and second kind are defined by 
    $$
      F(l,k) := \int_0^l \frac{1}{\sqrt{1-t^2}\sqrt{1-k^2t^2}}\,dt, \qquad
      E(l,k) := \int_0^l \frac{\sqrt{1-k^2t^2}}{\sqrt{1-t^2}}\,dt
    $$
    and in the special case $l=1$ we obtain the so-called complete elliptic integrals
    $K(k) := F(1,k)$, $E(k) := E(1,k)$. The derivatives of these functions are known to be (see 710.00 and
    710.02 in \cite{ByrFri_handbook}) 
    \begin{align} \label{eq:derivative_EK}
      K'(k) = \frac{E(k)-(1-k^2)K(k)}{k(1-k^2)},\qquad
      E'(k) = \frac{E(k)-K(k)}{k}.
    \end{align}
    As we will see in \eqref{eq:Jac_Inverse} these functions relate to Jacobi's elliptic functions
    $\sn(\cdot,k),\cn(\cdot,k),\dn(\cdot,k)$ (we omit the arguments for convenience) that can be defined as
    the unique solution of the initial value problem
    \begin{align}  \label{eq:Jacobi_ODE_1terOrdnung}
      \sn' = \cn\dn, \quad
      \cn'=-\sn\dn,\quad
      \dn' = -k^2\sn\cn,\quad
      \sn(0)=0,\;\cn(0)=\dn(0)=1. 
    \end{align}
    Thanks to the identities  
    \begin{align} \label{eq:Jacobi_trigonometric_identities}
      \cn^2+\sn^2 = 1,\qquad \dn^2+k^2\sn^2 = 1,
    \end{align}
    these functions are globally defined and satisfy
    \begin{align} \label{eq:Jacobi_ODE_2terOrdnung}
      \cn'' = (2k^2-1)\cn -2k^2\cn^3, \;
      \sn'' =  -(1+k^2)\sn +2k^2\sn^3,\;
      \dn'' = (2-k^2)\dn -2\dn^3.
    \end{align} 
    The function $\cn$ is even and $2K(k)-$antiperiodic, $\sn$ is odd and $2K(k)-$antiperiodic while $\dn$ is
    $2K(k)-$periodic. Moreover, $\sn:[-K(k),K(k)]\to [-1,1],\cn:[0,K(k)]\to [0,1]$ and $\dn:[0,K(k)]\to
    [\sqrt{1-k^2},1]$ are bijective with inverses given by
     \begin{align} \label{eq:Jac_Inverse}
      \sn^{-1}(z,k) = F(z,k),\quad
      \cn^{-1}(z,k) = F(\sqrt{1-z^2},k),\quad
      \dn^{-1}(z,k) = F(\sqrt{1-z^2}/k,k).
    \end{align}  
    In particular we have $\sn(K(k))=1,\cn(K(k))=0,\dn(K(k))=\sqrt{1-k^2}$ and addition formulas similar to
    the ones for sine and cosine can be proved. We will only need
    \begin{align} \label{eq:Jac_Additionstheoreme}
      \cn(s-K(k)) = \sqrt{1-k^2} \frac{\sn(s)}{\dn(s)},\;
      \sn(s-K(k)) = -\frac{\cn(s)}{\dn(s)},\;
      \dn(s-K(k)) = \frac{\sqrt{1-k^2}}{\dn(s)},
    \end{align}
    see 122.03 in \cite{ByrFri_handbook}. Next we define Heuman's Lambda function  $\Lambda_0$ via
    $$
      \Lambda_0(\arcsin(l),k):=\frac{2}{\pi} (E(k)F(l,k')+K(k)E(l,k')-K(k)F(l,k')),
    $$
    see 150.03 in \cite{ByrFri_handbook}. Here and in the text we will write $k':=\sqrt{1-k^2}$.     
    The formulas  710.11 and 730.04 imply
    \begin{align} \label{eq:derivative_Heuman}
      \begin{aligned}
       \frac{\partial}{\partial l} \Lambda_0(\arcsin(l),k)
       &= \frac{2(E(k)-(1-k^2)l^2K(k))}{\pi l'\sqrt{1-(1-k^2)l^2}},\\
       \frac{\partial}{\partial k} \Lambda_0(\arcsin(l),k)
       &= \frac{2(E(k)-K(k))ll'}{\pi k\sqrt{1-(1-k^2)l^2}}.
       \end{aligned} 
    \end{align}
    Additionally, we will need Jacobi's eta, theta and zeta functions
    $H,\Theta,\Theta_1,\zeta:\C\times(0,1)\to\C$ defined through the following formulas (see 144.01 and 
    1050.01 in \cite{ByrFri_handbook}):
    \begin{align*}
      &\Theta(z,k) := \vartheta_0(\frac{\pi z}{2 K(k)},q_k),
      &&H(z,k) := \vartheta_1(\frac{\pi z}{2 K(k)},q_k),  \\
      &\Theta_1(z,k) := \vartheta_3(\frac{\pi z}{2 K(k)},q_k), 
      &&\zeta(z,k) := \frac{\Theta_z(z,k)}{\Theta(z,k)}. 
    \end{align*}
    Here, the Jacobian Theta Functions $\vartheta_0,\vartheta_1$ and the elliptic nome $q_k$ are given by
    \begin{align} \label{eq:JacobianTheta_series}
      \begin{aligned}
      \vartheta_0(z,q)
      &= 1 + 2 \sum_{n=1}^\infty (-1)^n q^{n^2}\cos(2nz), \\
      \vartheta_1(z,q)
      &= 2   \sum_{n=0}^\infty (-1)^n q^{(n+1/2)^2}\sin((2n+1)z), \\
      \vartheta_3(z,q)
      &= 1 + 2\sum_{n=1}^\infty q^{n^2}\cos(2nz), \\
      q_k  
      &= \exp(-\pi K(k')/K(k)).
       \end{aligned}
    \end{align}
    Exploiting the corresponding properties for ordinary sine and cosine (defined in the complex plane) the
    above functions inherit symmetry properties and further identities such as
    $\Theta(z+K(k),k)=\Theta_1(z,k)$ listed on page 315f. in \cite{ByrFri_handbook}. We will give a precise
    reference when required. 
  
 \section{Orbitlike elasticae}
    
    In this section we collect some partially well-known facts about orbitlike elasticae. To this end we first
    analyze the distance function $Z$ in order to prove the explicit formulas
    for orbitlike elasticae in Theorem~\ref{thm:explicit_formula_orbitlike}.
     These formulas will enable us to prove the properties mentioned in
    Theorem~\ref{thm:LanSin} and to discuss the Dirichlet problem in section~\ref{sec:BVPs}. So from now on let
    always $\gamma:\R\to\Hh$ denote an orbitlike elastica with (hyperbolic)  curvature function
    \begin{equation} \label{eq:def_kappa_orbitlike}
       \kappa(s)=\frac{2}{\sqrt{2-k^2}}\dn(\frac{s+s_*}{\sqrt{2-k^2}},k)
    \end{equation}
    for $0<k<1,s_*\in\R$ and let $\phi,Z$ be defined as in \eqref{eq:ODE_solution},\eqref{eq:def_Z}. We recall
    $\mu= 4(1-k^2)/(2-k^2)^2$. In order to measure the angular progress of an elastica generated by this
    curvature function via \eqref{eq:ODE_solution} we will need
   \begin{equation}\label{eq:def_rotation}
     \Delta\theta_k := \pi- \pi \Lambda_0(\arcsin(k'),k)  + 2\sqrt{1-k^2}\sqrt{2-k^2}K(k),
   \end{equation}
   a function that has been introduced and analyzed in Proposition 5.3 in~\cite{LanSing_the_total}. 
   We mention that Langer and Singer gave two expressions for $\Delta\theta_k$ on page 21 in
   \cite{LanSing_the_total} which, however, do not coincide. We chose the correct second expression
   as a definition. In their first expression the factor $(1-p^2)$ has to replaced by
   $\sqrt{1-p^2}$, as we will see in the proof of Proposition~\ref{prop:properties_vartheta_orbitlike}.

   \subsection{Preliminaries}
   
   The explicit formulas for orbitlike elasticae involve functions $w_1,w_2:\R\to\R$ that can be
   written down explicitly in terms of the special functions introduced in section~\ref{sec:Jacobi}.
   In order to get directly to the point, we only  list the crucial properties of $w_1,w_2$ but postpone the
   precise definition of $w_1,w_2$. The tedious computations which are necessary to prove these properties 
   as well the other results in this subsection will be given in
   section~\ref{subsec:proof_propositions_orbitlike}.
   
   \begin{prop}\label{prop:w1w2_orbitlike}
     For all $k\in (0,1)$ the ordinary differential equation $w''+2\dn^2(\cdot,k)w=0$ has a 
     fundamental system $\{w_1,w_2\}$ on $\R$ such that the following holds:
     \begin{itemize}
       \item[(i)] $w_1$ is odd about 0, $w_2$ is even about 0, $w_2(0)>0>w_1'(0)$, 
       \item[(ii)] $w_1(s)^2+w_2(s)^2=2-k^2 -\dn^2(s,k)$ on $\R$,
       \item[(iii)] $w_1(s)w_1'(s)+w_2(s)w_2'(s) = k^2\dn(s,k)\cn(s,k)\sn(s,k)$ on $\R$,
       \item[(iv)] $w_1(s)w_2'(s)-w_2(s)w_1'(s) = \sqrt{1-k^2}\sqrt{2-k^2}$ on $\R$, 
       \item[(v)] $(w_1+iw_2)(s+2lK(k)) = (w_1+iw_2)(s) e^{il\Delta\theta_k}$ for
       all $s\in\R,l\in\Z$,
	   \item[(vi)] $(w_1+iw_2)(-K(k)) = e^{i(\pi-\Delta\theta_k)/2}$.
     \end{itemize}
   \end{prop}
   
   We remark that the equation $w''+2\dn^2(\cdot,k)w=0$ is a special case ($h=2,n=1$) of the so-called
   Lam\'{e} equation which is typically written in the form $w''=(n(n+1)k^2\sn^2(\cdot,k)-h)w$, where $h$ is
   called the Lam\'{e} eigenvalue. Its significance for elastic curves and knots was
   pointed out by Ivey and Singer \cite{IvSin_knot_types}. We refer to \cite{Ince_the_periodic} and to
   Whittaker's book \cite{Whit_modern_analysis} for more information about Lam\'{e}'s equation.  Now, being
   given $w_1,w_2$ we define
   \begin{align}\label{eq:def_Wj}
     W_j(s) := \frac{1}{\sqrt{2-k^2}} w_j\big(-K(k)+ \frac{s+s_*}{\sqrt{2-k^2}}\big) \qquad (j=1,2).
   \end{align}
   We will make use of the following identities.
   \begin{prop}\label{prop:W1W2_orbitlike}
     Let $\kappa$ be given by \eqref{eq:def_kappa_orbitlike}. Then $\{W_1,W_2\}$ is a
     fundamental system of the ODE $W''+2\mu\kappa^{-2}W=0$ and the following identities hold on $\R$:
     \begin{itemize} 
       \item[(i)] $W_1^2+W_2^2 = \frac{\kappa^2-\mu}{\kappa^2}$,
       \item[(ii)] $W_1'= \frac{\mu \kappa'}{\kappa(\kappa^2-\mu)} W_1
      - \frac{\sqrt\mu \kappa^2}{2(\kappa^2-\mu)}W_2 $,
       \item[(iii)] $W_2' = \frac{\sqrt\mu \kappa^2}{2(\kappa^2-\mu)}W_1 + \frac{\mu
       \kappa'}{\kappa(\kappa^2-\mu)} W_2$,
       \item[(iv)] $W_1W_2'-W_2W_1' = \frac{\sqrt{\mu}}{2}$,
       \item[(v)] $(W_1+iW_2)(s+2l\sqrt{2-k^2}K(k))=(W_1+iW_2)(s)e^{il\Delta\theta_k}$ for $s\in\R,l\in\Z$,
       \item[(vi)] $(W_1+iW_2)(s)= (-W_1+iW_2)(-s-2s_*+2l\sqrt{2-k^2}K(k)) e^{i(l-1)\Delta\theta_k}$
       for $s\in\R,l\in\Z$.
     \end{itemize}
   \end{prop}

  Next we choose an angle function $\vartheta$ as follows:
  \begin{align}\label{eq:def_theta}
    W_2(s)-iW_1(s) = \frac{
                     \sqrt{\kappa^2-\mu}}{\kappa} e^{i\vartheta(s)}.
  \end{align}
  This definition makes sense due to Proposition~\ref{prop:W1W2_orbitlike}~(i).
  Notice that \eqref{eq:def_theta} determines $\vartheta$ only up to additive multiples of $2\pi$ so that we
  require $\vartheta$ to satisfy the inequality $-2\pi<\vartheta(-s_*)\leq 0$. Since many of our
  computations involve this function, it is reasonable to collect a few properties of it. 
  As a byproduct our proof reveals the discrepancy of the formulas for $\Delta\theta_k$ provided by Langer and
  Singer we mentioned earlier.

  \begin{prop}\label{prop:properties_vartheta_orbitlike}
    The following identities hold:
    $$
      \vartheta' = \frac{\sqrt\mu \kappa^2}{2(\kappa^2-\mu)},\qquad
       \matII{W_1}{W_2}{W_1'}{W_2'}
      = 
      \matII{0}{\frac{\sqrt{\kappa^2-\mu}}{\kappa}}{\frac{-\sqrt{\mu}\kappa}{2\sqrt{\kappa^2-\mu}}}{
      \frac{\mu\kappa'}{\kappa^2\sqrt{\kappa^2-\mu}}}
      \matII{\cos(\vartheta)}{\sin(\vartheta)}{-\sin(\vartheta)}{\cos(\vartheta)}.
    $$
    Moreover, we have for $s\in\R,l\in\Z$
    \begin{itemize}
	  \item[(i)] $\vartheta(-s_*+l\sqrt{2-k^2}K(k)) = \frac{l-1}{2}\Delta\theta_k$,
      \item[(ii)] $\vartheta(s+2l\sqrt{2-k^2}K(k))-\vartheta(s) = l\Delta\theta_k$,
      \item[(iii)] $\vartheta(-s_*+s)+\vartheta(-s_*-s) = -\Delta\theta_k$.
    \end{itemize} 
  \end{prop}

  The formula for $\vartheta'$ given above can also be found on p.565 in the paper of Bryant and Griffiths
  \cite{BryGri_Reduction}. The derivation of this formula, however, is not carried out in detail so that  
  our proof in section~\ref{subsec:proof_propositions_orbitlike} may be of interest.   
  
  \subsection{Explicit formulas}
  
   As mentioned earlier the key idea is to derive the explicit formulas for elasticae in the hyperbolic
   plane from formulas for the expression 
   $$
     Z(s;P) = \frac{(\gamma_1(s)-P_1)^2+(\gamma_2(s)-P_2)^2}{2\gamma_2(s)P_2}
   $$
   in dependence of both $s$ and $P$. We notice once again that formally we need not restrict ourselves to the
   geometrically meaningful case $P\in\Hh$, i.e. $P=(P_1,P_2)$ with $P_2>0$, but may also consider points with $P_2<0$ or
   even $P\in\C$ with $P_2\neq 0$, which will become important later in the discussion of wavelike elasticae.
   In order to accurately describe the function $Z$ we will need the following functions: 
   \begin{align} \label{eq:formula_xi}
     \begin{aligned}
     \xi^1(s)
      &:=  \matII{ W_1(s)}{W_2(s)}{ W_1'(s)}{
      W_2'(s)}^{-1}
      \vecII{\frac{1}{\kappa(s)}-\frac{\kappa(s)}{\mu}-\frac{\kappa'(s)}{\mu}\sin(\phi(s))
      +\frac{\kappa(s)^2}{2\mu}\cos(\phi(s))}{-\frac{1}{\kappa(s)}\sin(\phi(s))-\frac{\kappa'(s)}{\kappa(s)^2}},
      \\
      \xi^2(s)
	   &:=   \matII{ W_1(s)}{ W_2(s)}{ W_1'(s)}{
      W_2'(s)}^{-1}
      \vecII{\frac{1}{\kappa(s)}-\frac{\kappa(s)}{\mu}+\frac{\kappa'(s)}{\mu}\sin(\phi(s))
      -\frac{\kappa(s)^2}{2\mu}\cos(\phi(s))}{\frac{1}{\kappa(s)}\sin(\phi(s))-\frac{\kappa'(s)}{\kappa(s)^2}},
      \\
      \xi^3(s)
      &:=  \matII{ W_1(s)}{ W_2(s)}{ W_1'(s)}{
      W_2'(s)}^{-1}  \vecII{\frac{\kappa(s)^2}{2\mu}\sin(\phi(s))
      +\frac{\kappa'(s)}{\mu}\cos(\phi(s))}{\frac{1}{\kappa(s)}\cos(\phi(s))}.
      \end{aligned}
   \end{align}
   
   \begin{prop}\label{prop:formula_Z_orbitlike}
     For $P=(P_1,P_2)\in\R^2$ such that $P_2\neq 0$ we have 
     \begin{equation} \label{eq:formula_Z}
       Z(s;P)
       = -1 + \kappa(s)\big( AW_1(s) + BW_2(s) + C\big)
     \end{equation}
     where $C$ is given by \eqref{eq:def_C} and $A,B\in\R$ satisfy  
     \begin{align} \label{eq:formula_AB}
       \vecII{A}{B}
       &= \frac{(\gamma_1(s_0)-P_1)^2+P_2^2}{2\gamma_2(s_0)P_2} \xi^1(s_0)
         + \frac{\gamma_2(s_0)}{2P_2}\xi^2(s_0)
         + \frac{\gamma_1(s_0)-P_1}{P_2}\xi^3(s_0)
         \quad\forall s_0\in\R.
      \end{align}
   \end{prop}
   \begin{proof}
    
     Let $C$ be given by \eqref{eq:def_C} so that $Z(\cdot;P)$ satisfies the ODE \eqref{eq:ODE_Z}. We define
     the auxiliary function $W(s) := (Z(s;P)+1)\kappa(s)^{-1}-C$. Then we get
     \begin{align*}
       W''(s) 
       &= \frac{1}{\kappa(s)^2} \Big( \kappa(s)Z''(s;P)-2\kappa'(s)Z'(s;P)
       + (Z(s;P)+1)\big(-\kappa''(s)+2\frac{\kappa'(s)^2}{\kappa(s)}\big) \Big) \\
       &\stackrel{\eqref{eq:ODE_Z}}{=} \frac{1}{\kappa(s)^2} \Big( 
       (Z(s;P)+1)\big(-\kappa(s)-\kappa''(s)+2\frac{\kappa'(s)^2}{\kappa(s)} \big)
       +2\mu C\Big) \\
       &\stackrel{\eqref{eq:ODE_kappa}}{=} -\frac{2\mu}{\kappa(s)^2} \Big(
       \frac{Z(s;P)+1}{\kappa(s)}-C\Big) \\
       &= -\frac{2\mu}{\kappa(s)^2}W(s).
     \end{align*}
     By Proposition \ref{prop:W1W2_orbitlike} we can find $A,B\in\R$ such that $W(s)=AW_1(s)+BW_2(s)$ and
     thus
    \begin{align*}
      Z(s;P)
      = -1  + \kappa(s)\big(W(s)+C\big)  
      = -1  + \kappa(s)\big( AW_1(s) + B W_2(s) + C\big)
    \end{align*}
    so that \eqref{eq:formula_Z} is proved. Solving this identity for $W_1,W_2$ and differentiating it with
    respect to~$s$ we get the following equations for $A,B$:
     \begin{align*}
      \matII{ W_1(s_0)}{ W_2(s_0)}{ W_1'(s_0)}{ W_2'(s_0)}\vecII{A}{B}
      = \vecII{\frac{Z(s_0;P)+1}{\kappa(s_0)}-C}{\big(\frac{Z(\cdot;P)+1}{\kappa}-C\big)'(s_0)} 
      =  \vecII{\frac{Z(s_0;P)+1}{\kappa(s_0)}-C}{\frac{Z'(s_0;P)}{\kappa(s_0)}-
        \frac{\kappa'(s_0)(Z(s_0;P)+1)}{\kappa(s_0)^2}}.  
    \end{align*}
    Plugging in the formulas for $Z(s_0;P),Z'(s_0;P)$ and $C$, see
    \eqref{eq:def_Z},\eqref{eq:formula_Zprime},\eqref{eq:def_C}.  
   \end{proof}
   
   So for any given elastica $\gamma:\R\to\Hh$  the equations \eqref{eq:formula_Z},\eqref{eq:formula_AB}
   hold for all $s,s_0\in\R$ and all $P\in \R^2,P_2\neq 0$. Hence, we may derive explicit formulas for
   orbitlike elasticae \r{$(\gamma_1,\gamma_2)$} by equating the coefficients on each side of
   \eqref{eq:formula_Z}.
   
   \begin{thm}\label{thm:explicit_formula_orbitlike}
     A curve $\gamma:\R\to \Hh$ is an orbitlike elastica parametrized by hyperbolic arclength if and
     only if we have
     \begin{align}
       \gamma_1(s)
       &=\frac{b_1  W_1(s)+b_2 W_2(s)+b_3}{a_1 W_1(s)+a_2 W_2(s)+a_3}, 
       \label{eq:explicit_formula_orbitlike_1}
       \\
       \gamma_2(s)
       &=\frac{1}{\kappa(s)(a_1  W_1(s)+a_2 W_2(s)+a_3)}
       \label{eq:explicit_formula_orbitlike_2}  
     \end{align}
     for $a_1,a_2,a_3,b_1,b_2,b_3\in\R$ such that 
     \begin{equation} \label{eq:a_1a_2a_3b_1b_2b_3}
       a_3= \sqrt{a_1^2+a_2^2}>0,\qquad 
        \vecII{b_1}{b_2} = \frac{b_3}{a_3} \vecII{a_1}{a_2} + \frac{1}{\sqrt{\mu}a_3} \vecII{-a_2}{a_1}.
     \end{equation}
     In this case we moreover have for all $s\in\R$  
     \begin{align} 
       \gamma_2(s)(a_1,a_2)^T&= \xi^1(s) \label{eq:a1a2b1b2_1}, \\ 
       \gamma_2(s) (b_1,b_2)^T &= \gamma_1(s) \xi^1(s)+ \gamma_2(s) \xi^3(s),
       \label{eq:a1a2b1b2_2} \\
       \gamma_1(s) &= \frac{b_3}{a_3} +\frac{1}{2\mu
       a_3}\big(\kappa(s)^2\sin(\phi(s))+2\kappa'(s)\cos(\phi(s))\big), 
       \label{eq:b3a3_1}
       \\
       \gamma_2(s) &= \frac{\kappa(s)}{\mu a_3} 
       + \frac{1}{2\mu a_3}\big(2\kappa'(s)\sin(\phi(s))-\kappa(s)^2\cos(\phi(s))\big).
       \label{eq:b3a3_2}
     \end{align}
   \end{thm}

   \begin{proof}
     The strategy of the proof is the following: First we assume $\gamma$ to be an orbitlike elastica 
     in~$\Hh$ parametrized by hyperbolic arclength and derive the formulas
     \eqref{eq:explicit_formula_orbitlike_1}--\eqref{eq:b3a3_2}. This will prove the ''only if''-part of our
     claim.   
     Then we check that the formulas \eqref{eq:explicit_formula_orbitlike_1}--\eqref{eq:a_1a_2a_3b_1b_2b_3}
     indeed define orbitlike elasticae which, as solutions of \eqref{eq:ODE_kappa}, are parametrized by
     hyperbolic arclength.
     
     \medskip
     
     So let us first assume that $\gamma$ is an orbitlike elastica in $\Hh$ so that identity
     \eqref{eq:formula_Z} holds for all $s,s_0\in\R$ by Proposition~\ref{prop:formula_Z_orbitlike}.  We multiply \eqref{eq:formula_Z}
     with $2P_2\gamma_2(s)$ and obtain
     \begin{align*}
      (\gamma_1(s)-P_1)^2+\gamma_2(s)^2+P_2^2 
      = 2\gamma_2(s)\kappa(s)P_2\big( A W_1(s) + B  W_2(s) +  C \big) 
     \end{align*}
     where $A,B,C$ are given by the formulas \eqref{eq:formula_AB},\eqref{eq:def_C}. 
     Equating the coefficients of the monomial $P_2^2$ we get
     \begin{align*}
       1 &= \frac{\gamma_2(s)\kappa(s)}{\gamma_2(s_0)} \Big( \xi^1(s_0)_1 W_1(s) + \xi^1(s_0)_2 W_2(s) \\
       &\quad + \frac{\kappa(s_0)}{\mu}+\frac{\kappa'(s_0)}{\mu}\sin(\phi(s_0))
       - \frac{\kappa(s_0)^2}{2\mu}\cos(\phi(s_0))\Big)  \\
       &= \gamma_2(s)\kappa(s)\big(a_1 W_1(s)+a_2W_2(s)+a_3\big)
     \end{align*}
     and thus \eqref{eq:explicit_formula_orbitlike_2} is proved. Notice that the above identity holds for
     all $s_0\in\R$, which proves the formulas \eqref{eq:a1a2b1b2_1},\eqref{eq:b3a3_2} for $a_1,a_2,a_3$. 
     Similarly, equating the coefficients of $P_1$ we get
     \begin{align*}
       -2\gamma_1(s) 
       &= \frac{2\gamma_2(s)\kappa(s)}{\gamma_2(s_0)} \Big(
       - \big(\gamma_1(s_0)\xi^1(s_0)+\gamma_2(s_0)\xi^3(s_0)\big)_1 W_1(s)  \\
       &-\big(\gamma_1(s_0)\xi^1(s_0)+\gamma_2(s_0)\xi^3(s_0)\big)_2 W_2(s) \\
        &- \gamma_1(s_0) 
         \Big(\frac{\kappa(s_0)}{\mu}+\frac{\kappa'(s_0)}{\mu}\sin(\phi(s_0)) -
         \frac{\kappa(s_0)^2}{2\mu}\cos(\phi(s_0))\Big) \\
         &-\gamma_2(s_0)\Big(-\frac{\kappa(s_0)^2}{2\mu}\sin(\phi(s_0))-\frac{\kappa'(s_0)}{\mu}\cos(\phi(s_0))\Big)
         \Big) \\
        &= -2\gamma_2(s)\kappa(s)\big(b_1W_1(s)+b_2W_2(s)+b_3\big).
     \end{align*}
      In the same way as above this implies the formulas
      \eqref{eq:explicit_formula_orbitlike_1},\eqref{eq:a1a2b1b2_2},\eqref{eq:b3a3_1} involving
      $b_1,b_2,b_3$. It therefore remains to prove \eqref{eq:a_1a_2a_3b_1b_2b_3} to finish the ''only
      if''-part. To this end we introduce the matrices
  \begin{align*}
    \mathcal{W}:=\matII{W_1}{W_2}{W_1'}{W_2'},\qquad
    \mathcal{R}(\vartheta):=\matII{\cos(\vartheta)}{\sin(\vartheta)}{-\sin(\vartheta)}{\cos(\vartheta)}.
  \end{align*}
  Then Proposition~\ref{prop:properties_vartheta_orbitlike} gives the following:
    \begin{align*}
      \mathcal{W}
      &= 
      \matII{0}{\frac{\sqrt{\kappa^2-\mu}}{\kappa}}{\frac{-\sqrt{\mu}\kappa}{2\sqrt{\kappa^2-\mu}}}{
      \frac{\mu\kappa'}{\kappa^2\sqrt{\kappa^2-\mu}}} \mathcal{R}(\vartheta),\qquad
      \mathcal{W}^{-1}
      = 
      \mathcal{R}(-\vartheta)
      \matII{\frac{2\sqrt{\mu}\kappa'}{\kappa^2\sqrt{\kappa^2-\mu}}}{-\frac{2\sqrt{\kappa^2-\mu}}{\sqrt{\mu}\kappa}}
      {\frac{\kappa}{\sqrt{\kappa^2-\mu}}}{0}. \\
  \intertext{Combining this with \eqref{eq:formula_xi} and $(\kappa')^2=\kappa^2-\frac{1}{4}\kappa^4-\mu$ we get} 
     \xi^1
      &= 
      \mathcal{R}(-\vartheta)
      \vecII{\frac{1}{2\sqrt{\mu(\kappa^2-\mu)}}(2\kappa'\cos(\phi)+\kappa^2\sin(\phi))}{ 
      \frac{1}{2\mu\sqrt{\kappa^2-\mu}}(2\mu-2\kappa^2-2\kappa\kappa'\sin(\phi)+\kappa^3\cos(\phi))
      },  \\
      \xi^3
      &= 
      \mathcal{R}(-\vartheta)
      \vecII{\frac{1}{2\sqrt{\mu(\kappa^2-\mu)}}(2\kappa'\sin(\phi)-\kappa^2\cos(\phi))}{ 
      \frac{\kappa}{2\mu\sqrt{\kappa^2-\mu}}(\kappa^2\sin(\phi)+2\kappa'\cos(\phi))
      }.  
  \end{align*} 
  Writing $\skp{\cdot}{\cdot}$ for the standard inner product in $\R^2$ we get from 
  $(\kappa')^2=\kappa^2-\frac{\kappa^4}{4}-\mu$ and 
  the fact that the rotation matrices $\mathcal{R}(-\vartheta)$ are Euclidean isometries 
  \begin{align} 
    \skp{\xi^1}{\xi^1}
    &= \frac{(2\kappa'\cos(\phi)+\kappa^2\sin(\phi))^2}{4\mu (\kappa^2-\mu)}
      + \frac{(2\mu-2\kappa^2-2\kappa\kappa'\sin(\phi)+\kappa^3\cos(\phi))^2}{4\mu^2 (\kappa^2-\mu)}
      \notag    \\
    &= \frac{\cos^2(\phi)(\kappa^4+4\mu) +
    \sin^2(\phi)(4\kappa^2-\kappa^4) - 4\kappa^2\kappa'\sin(\phi)\cos(\phi)}{4\mu^2 }  \notag \\
    &\;+\frac{ 4\kappa^2-4\mu - 4\kappa^3 \cos(\phi)+8\kappa\kappa'\sin(\phi)}{4\mu^2} \notag \\
    &=
    \frac{(2\kappa+2\kappa'\sin(\phi)-\kappa^2\cos(\phi))^2}{4\mu^2} \notag \\
    &\stackrel{\eqref{eq:b3a3_2}}{=}  a_3^2\gamma_2^2 .  \label{eq:formulas_scpr_xi_1}\\
    \intertext{Similar computations yield $((\xi^1)^\perp:=(-\xi^1_2,\xi^1_1))$}
    \skp{\xi^1}{\xi^3}
    &= \frac{\kappa^2\kappa'\cos^2(\phi)
      - 2 \kappa^2\kappa'\sin^2(\phi)
      + (\kappa^4-4(\kappa')^2)\sin(\phi)\cos(\phi)
    }{4\mu^2} \notag \\
    &\;+ \frac{-2\kappa^3\sin(\phi)
      - 4\kappa\kappa'\cos(\phi)
    }{4\mu^2} \notag \\
    &= -  \frac{(\kappa^2\sin(\phi)+2\kappa'\cos(\phi))
    (2\kappa+2\kappa'\sin(\phi)-\kappa^2\cos(\phi))}{4\mu^2 } \notag \\
    &\stackrel{\eqref{eq:b3a3_1},\eqref{eq:b3a3_2}}{=}  a_3(b_3-\gamma_1 a_3),  
    \label{eq:formulas_scpr_xi_2}\\
    \skp{(\xi^1)^\perp}{\xi^3} 
    &= \frac{(2\mu-2\kappa^2-2\kappa\kappa'\sin(\phi)+\kappa^3\cos(\phi))
    (2\kappa'\sin(\phi)-\kappa^2\cos(\phi))
    }{ 4\mu^{3/2}(\kappa^2-\mu)} \notag \\
    &\;+
    \frac{(2\kappa'\cos(\phi)+\kappa^2\sin(\phi))(\kappa^3\sin(\phi)+2\kappa\kappa'\cos(\phi))
    }{4\mu^{3/2}(\kappa^2-\mu)} \notag \\
    &= \frac{2\kappa+2\kappa'\sin(\phi)-\kappa^2\cos(\phi)}{2\mu^{3/2} (2-k^2)} \notag \\
    &\stackrel{\eqref{eq:b3a3_2}}{=} \frac{a_3\gamma_2}{\sqrt\mu}.
     \label{eq:formulas_scpr_xi_3}
  \end{align}
   This implies 
  \begin{align*}
    a_1^2+a_2^2
    \stackrel{\eqref{eq:a1a2b1b2_1}}{=} \frac{1}{\gamma_2^2}\skp{\xi^1}{\xi^1} 
   \stackrel{\eqref{eq:formulas_scpr_xi_1}}{=}   a_3^2.
  \end{align*}
  Hence, we find the formula for $a_3$ from \eqref{eq:a_1a_2a_3b_1b_2b_3} since  $a_3$ is positive due to
  \begin{align*}
    a_3
    \stackrel{\eqref{eq:b3a3_2}}{=} 
    \frac{2\kappa+2\kappa'\sin(\phi)-\kappa^2\cos(\phi)}{2\mu \gamma_2} 
    \geq \frac{2\kappa-\sqrt{(2\kappa')^2+\kappa^4}}{2\mu \gamma_2} 
    \stackrel{\eqref{eq:ODE_kappa}}{=} \frac{\kappa-\sqrt{\kappa^2-\mu}}{\mu\gamma_2}
    > 0.
  \end{align*}
  Similarly we get 
  \begin{align*}
    b_1a_1+b_2a_2
     &\hspace{-3mm}\stackrel{\eqref{eq:a1a2b1b2_1},\eqref{eq:a1a2b1b2_2}}{=}
      \frac{\gamma_1}{\gamma_2^2}\skp{\xi^1}{\xi^1} + \frac{1}{ \gamma_2}\skp{\xi^1}{\xi^3}  
     \stackrel{\eqref{eq:formulas_scpr_xi_1},\eqref{eq:formulas_scpr_xi_2}}{=}
       \gamma_1a_3^2  +  a_3(b_3-\gamma_1 a_3)  
      =  b_3a_3, \\
    -b_1a_2+b_2a_1
    &\hspace{-3mm}\stackrel{\eqref{eq:a1a2b1b2_1},\eqref{eq:a1a2b1b2_2}}{=}
      \frac{1}{\gamma_2}\skp{(\xi^1)^\perp}{\xi^3}  
    \stackrel{\eqref{eq:formulas_scpr_xi_3}}{=} \frac{a_3}{\sqrt\mu}.
  \end{align*}
  Combining these equations we obtain the formulas for $b_1,b_2$ and \eqref{eq:a_1a_2a_3b_1b_2b_3} is proved. 
 
     \medskip

     Now we prove the sufficiency of our conditions, so let $a_1,\ldots,b_3$ be given as in
     \eqref{eq:a_1a_2a_3b_1b_2b_3} we define $\gamma_1,\gamma_2$ by the formulas
     \eqref{eq:explicit_formula_orbitlike_1},\eqref{eq:explicit_formula_orbitlike_2}. Motivated 
     by the formulas~\eqref{eq:b3a3_1},\eqref{eq:b3a3_2}  derived in the ''only if''-part we define
     $\phi$ via
     \begin{equation} \label{eq:formula_phi}
         \vecII{\sin(\phi)}{\cos(\phi)}
         := \frac{\mu a_3}{2(\kappa^2-\mu)}\matII{\kappa^2}{2\kappa'}{2\kappa'}{-\kappa^2}
         \vecII{\gamma_1-\frac{b_3}{a_3}}{\gamma_2-\frac{\kappa}{\mu a_3}}. 
     \end{equation}
     We have to verify that such a definition makes sense by checking that the Euclidean norm of the vector on
     the right hand side is 1. We introduce the notation   
     $$
       aW:=a_1W_1+a_2W_2,\quad a^\perp W:=-a_2W_1+a_1W_2,  \quad bW:= b_1W_1+b_2W_2
     $$ 
     so that we can use the following facts:
     $$
       \gamma_1 \stackrel{\eqref{eq:explicit_formula_orbitlike_1}}{=} \frac{bW+b_3}{aW+a_3},\quad
       \gamma_2 \stackrel{\eqref{eq:explicit_formula_orbitlike_2}}{=} \frac{1}{\kappa(aW+a_3)},\quad
       b \stackrel{\eqref{eq:a_1a_2a_3b_1b_2b_3}}{=} \frac{b_3}{a_3} a + \frac{1}{\sqrt\mu a_3} a^\perp,\quad
       a_3bW-b_3 aW \stackrel{\eqref{eq:a_1a_2a_3b_1b_2b_3}}{=} \frac{1}{\sqrt\mu} a^\perp W.
     $$
     This implies
     \begin{align*}
       &\frac{\mu^2 a_3^2}{4(\kappa^2-\mu)^2} \Big( 
       \Big( \kappa^2\big(\gamma_1-\frac{b_3}{a_3}\big)+2\kappa'\big(\gamma_2-\frac{\kappa}{\mu
       a_3}\big)\Big)^2 
       + \Big( 
       2\kappa'\big(\gamma_1-\frac{b_3}{a_3}\big)-2\kappa^2\big(\gamma_2-\frac{\kappa}{\mu
       a_3}\big)   \Big)^2   \Big) \\
       &= \frac{\mu^2 a_3^2 (\kappa^4+4(\kappa')^2)}{4(\kappa^2-\mu)^2} 
       \Big( \big(\gamma_1-\frac{b_3}{a_3}\big)^2+ \big(\gamma_2-\frac{\kappa}{\mu
       a_3}\big)^2\Big)  \\
       &= \frac{\mu^2 a_3^2}{\kappa^2-\mu} 
       \Big( \big( \frac{a_3bW-b_3aW}{a_3(aW+a_3)}\big)^2+ \big( \frac{-\kappa^2aW +
       (\mu-\kappa^2)a_3}{\mu\kappa a_3(aW+a_3)} \big)^2\Big)    \\
       &= \frac{\big( \kappa\mu(a_3bW-b_3aW)\big)^2+ \big(-\kappa^2aW +
       (\mu-\kappa^2)a_3  \big)^2}{\kappa^2(\kappa^2-\mu)(aW+a_3)^2} \\
       &= \frac{ \kappa^2\mu(a^\perp W)^2
         + \kappa^4(aW)^2 +2\kappa^2(\kappa^2-\mu)a_3aW+ (\kappa^2-\mu)^2a_3^2
        }{\kappa^2(\kappa^2-\mu)(aW+a_3)^2} \\
       &= \frac{ \kappa^2\mu |a|^2|W|^2 
         + (\kappa^4-\kappa^2\mu)(aW)^2 +2\kappa^2(\kappa^2-\mu)a_3aW+ (\kappa^2-\mu)^2a_3^2
        }{\kappa^2(\kappa^2-\mu)(aW+a_3)^2} \\
       &= \frac{(aW)^2 +2a_3aW+   a_3^2}{(aW+a_3)^2} \\ 
       &= 1.
     \end{align*}
     Here we have used $\kappa^2|a|^2|W|^2=(\kappa^2-\mu)a_3^2$, see  
     \eqref{eq:a_1a_2a_3b_1b_2b_3} and Proposition~\ref{prop:W1W2_orbitlike}~(i).
     It remains to check that $(\gamma_1,\gamma_2,\phi)$ satisfies
     the ODE~\eqref{eq:ODE_solution}. To this end, we first rewrite the formulas for $\sin(\phi),\cos(\phi)$
     from \eqref{eq:formula_phi} in the following way:
     \begin{align*}
       \sin(\phi)
       &= \frac{\mu a_3}{2(\kappa^2-\mu)} \Big(
         \kappa^2\big( \frac{bW+b_3}{aW+a_3}-\frac{b_3}{a_3}\big)
         +2\kappa'\big( \frac{1}{\kappa(aW+a_3)}-\frac{\kappa}{\mu a_3}\big)
       \Big) \\
       &= \frac{\mu a_3}{2(\kappa^2-\mu)} \Big(
         \frac{\kappa^2( bW a_3-aW b_3)}{a_3(aW+a_3)}
         + \frac{2\kappa'( \mu a_3-\kappa^2 (aW+a_3))}{\mu\kappa a_3(aW+a_3)}  \Big) \\
       &= \frac{\mu a_3}{2(\kappa^2-\mu)} \Big(
         \frac{\kappa^2 a^\perp W}{\sqrt\mu a_3(aW+a_3)}
         + \frac{2\kappa'( (\mu-\kappa^2) a_3- \kappa^2 aW)}{\mu\kappa a_3(aW+a_3)}  \Big) \\
       &=  \frac{\sqrt\mu \kappa^3 a^\perp W
         - 2\kappa^2\kappa' aW
         - 2\kappa'(\kappa^2-\mu) a_3}{2\kappa(\kappa^2-\mu)(aW+a_3)},  \\
        \cos(\phi)
        &= \frac{2\sqrt\mu \kappa' a^\perp W + \kappa^3 aW +
        \kappa^2(\kappa^2-\mu)a_3}{2\kappa(\kappa^2-\mu)(aW+a_3)}.
     \end{align*}
     For the calculation of $\gamma_1',\gamma_2'$ we use the formula
     \begin{align*}
       W' = \frac{\mu\kappa'}{\kappa(\kappa^2-\mu)} W + \frac{\sqrt\mu \kappa^2}{2(\kappa^2-\mu)}W^\perp
     \end{align*}
     from Proposition~\ref{prop:W1W2_orbitlike}~(ii),(iii). We obtain
     \begin{align*}
       \frac{\gamma_2'}{\gamma_2}
       &= - \frac{\kappa'}{\kappa} - \frac{aW'}{aW+a_3} \\
       &=  - \frac{\kappa'}{\kappa} - \frac{2\mu\kappa' aW -
       \sqrt\mu\kappa^3 a^\perp W }{2\kappa(\kappa^2-\mu)(aW+a_3)} \\
       &= \frac{ \sqrt\mu \kappa^3 a^\perp W - 2\kappa^2\kappa' aW
       - 2\kappa'(\kappa^2-\mu)a_3}{2\kappa(\kappa^2-\mu)(aW+a_3)} \\
       &= \sin(\phi). \\
       \intertext{Similarly, $|a|^2=a_1^2+a_2^2=a_3^2$ implies}
       \frac{\gamma_1'}{\gamma_2}
       &= \kappa(aW+a_3) \Big( \frac{bW+b_3}{aW+a_3}\Big)' \\
       &= \kappa bW' -\kappa aW' \frac{bW+b_3}{aW+a_3} \\
       &= \kappa W' \big( -\frac{a^\perp W}{\sqrt\mu a_3(aW+a_3)} a + \frac{1}{\sqrt\mu
       a_3}a^\perp \big) \\
       &= \frac{(2\sqrt\mu\kappa' W + \kappa^3 W^\perp)
       (- (a^\perp W) a + (aW+a_3)a^\perp)}{2a_3(aW+a_3)(\kappa^2-\mu)}    \\
       &= \frac{\kappa^3 (a^\perp W)^2+\kappa^3 (aW)^2
        + a_3(\kappa^3 aW + 2\sqrt\mu\kappa' a^\perp W)}{2a_3(aW+a_3)(\kappa^2-\mu)}   \\
        &= \frac{\kappa^3 |a|^2|W|^2
        + a_3(\kappa^3 aW + 2\sqrt\mu\kappa' a^\perp W)}{2a_3(aW+a_3)(\kappa^2-\mu)}     \\
        &= \frac{a_3\kappa(\kappa^2-\mu) + \kappa^3 aW + 2\sqrt\mu\kappa' a^\perp W}{2(aW+a_3)(\kappa^2-\mu)} 
        \\
        &= \cos(\phi).
     \end{align*}
     Here, once again we used $\kappa^3|a|^2|W|^2 = \kappa(\kappa^2-\mu)$.      
     So we see that the equations for $\gamma_1,\gamma_2$ from \eqref{eq:ODE_solution} are satisfied and it
     remains to verify the ODE for $\phi$. Having defined $\phi$ in \eqref{eq:formula_phi} in such a way that
     \eqref{eq:b3a3_1},\eqref{eq:b3a3_2} holds, we may differentiate these equations in order to derive the
     formula for $\phi'$. Differentiating \eqref{eq:b3a3_1} and using $\gamma_1'=\gamma_2\cos(\phi)$ 
     as well as  \eqref{eq:b3a3_2} we find
     \begin{align*}
       2\kappa'(\kappa-\phi')\sin(\phi)
       + (\kappa^2\phi'+2\kappa-\kappa^3)\cos(\phi)
       = 2\kappa\cos(\phi)+2\kappa'\sin(\phi)\cos(\phi)-\kappa^2\cos^2(\phi) 
     \end{align*} 
     Differentiating \eqref{eq:b3a3_2} and using $\gamma_2'=\gamma_2\sin(\phi)$ as well as \eqref{eq:b3a3_2}
     we get
     \begin{align*}
       -2\kappa'(\kappa-\phi')\cos(\phi)
       + (\kappa^2\phi'+2\kappa-\kappa^3)\sin(\phi)
       = -2\kappa'\cos(\phi)+2\kappa \sin(\phi) -\kappa^2\sin(\phi)\cos(\phi) 
     \end{align*} 
     Multiplying the first equation with $\sin(\phi)$, the second with $-\cos(\phi)$ and adding up the
     resulting equations we arrive at $2\kappa'(\kappa-\phi')=2\kappa'\cos(\phi)$ and thus
     $\phi'+\cos(\phi)=\kappa$. So $(\gamma_1,\gamma_2,\phi)$ is a solution to the the ODE
     \eqref{eq:ODE_solution} and hence an orbitlike elastica, which is all we had to prove.   
    \end{proof}
    
   
   Theorem~\ref{thm:explicit_formula_orbitlike} shows that an orbitlike elastica with curvature
       is determined by fixing the three coefficients $a_1,a_2,b_3$. This is not surprising given the fact
       that the solutions of~\eqref{eq:ODE_solution} are uniquely determined by fixing   three initial
       conditions for $\gamma_1,\gamma_2,\phi$. 
     
  \subsection{Oscillation between circles} \label{subsec:oscillation}
    
    In \cite{LanSing_Curves_in_the}, Theorem~3 the authors state that closed orbitlike elasticae oscillate
    between circles. In accordance with the general scope of this paper we intend to give quantitative
    information related to this phenomenon, i.e. we determine the parameters of the circles. 
    In order to do this we have to explain how the function $Z$ can be used how to prove such a
    property. From the definition of $Z$ in \eqref{eq:def_Z} we deduce the following: The inequality
    $Z(s;P)\geq -1+\rho>0$ holds if and only if $\gamma$ lies outside the Euclidean ball with center $(P_1,\rho P_2)$ and
    radius $\sqrt{\rho^2-1}|P_2|$. The same way, we have $Z(s;P)\leq -1+\rho$ if and only if $\gamma$ lies
    inside this ball. This property will be crucial for proving the statement that an orbitlike elastica is
    confined between two balls in Proposition~\ref{prop:enclosures}. Notice moreover that the above
    reasoning shows  that hyperbolic circles with center $P$ and radius $\Arcosh(\rho)$ are
    nothing but Euclidean balls with center $(P_1,\rho P_2)$ and radius $\sqrt{\rho^2-1}|P_2|$ since both are
    characterized by the identity $Z(s;P)\equiv -1+\rho$, see~\eqref{eq:metric}. First we single out the point
    $P$ for which we have $A=B=0$ in formula \eqref{eq:formula_Z}. 
       
   \begin{prop} \label{prop:choice_of_P}
      For $P_1=\frac{b_3}{a_3},P_2=\frac{1}{\sqrt{\mu}a_3}$ we have for all $s\in\R$  
      $$
        Z(s;P) = -1 +  \frac{\kappa(s)}{\sqrt\mu}.
      $$ 
   \end{prop}
   \begin{proof} 
     We will use the notation introduced in the proof of Theorem~\ref{thm:explicit_formula_orbitlike}. For
     the above choice of $P_1,P_2$ we deduce from Theorem~\ref{thm:explicit_formula_orbitlike} and from
     $\kappa^2|a|^2||W|^2=(\kappa^2-\mu)a_3^2$
     \begin{align*}
       \big(\gamma_1-P_1\big)^2+\big(\gamma_1-P_2\big)^2
       &= \Big(\frac{a_3bW-b_3aW}{a_3(aW+a_3)}\Big)^2 + 
        \Big( \frac{(\sqrt\mu-\kappa)a_3-\kappa aW}{\sqrt\mu \kappa a_3(aW+a_3)}\Big)^2 \\
       &= \frac{\kappa^2 (a^\perp W)^2
          +(\kappa-\sqrt\mu)^2a_3^2
          +2\kappa(\kappa-\sqrt\mu)a_3aW
          + \kappa^2(aW)^2
       }{\mu\kappa^2 a_3^2(aW+a_3)^2}  \\
       &= \frac{\kappa^2 |a|^2|W|^2 +2\kappa(\kappa-\sqrt\mu)a_3aW+(\kappa-\sqrt\mu)^2a_3^2
       }{\mu\kappa^2 a_3^2(aW+a_3)^2}  \\
       &= \frac{2\kappa(\kappa-\sqrt\mu) a_3^2 +2\kappa a_3(\kappa-\sqrt\mu)a_3aW}{\mu\kappa^2
       a_3^2(aW+a_3)^2}
       \\
       &= \frac{2}{\sqrt\mu a_3 \kappa(aW+a_3)} \big(-1+\frac{\kappa}{\sqrt\mu}\big)  \\
       &= 2P_2\gamma_2\big(-1+\frac{\kappa}{\sqrt\mu}\big),
     \end{align*}
     which is all we had to show.
  \end{proof}

  Following the reasoning from the beginning of this section this leads to the following result:
   
  \begin{prop}\label{prop:enclosures}
    We have $\gamma(\R)\subset \overline{B_{R_1}(Q_1)}\sm B_{R_2}(Q_2)$ where
    \begin{align*}
      Q_1 &= \Big( \frac{b_3}{a_3},\frac{\sqrt{2+2\sqrt{1-\mu}}}{\mu a_3}\Big),\qquad
      R_1 = \frac{\sqrt{2-\mu+2\sqrt{1-\mu}}}{\mu a_3}, \\
      Q_2 &= \Big( \frac{b_3}{a_3},\frac{\sqrt{2-2\sqrt{1-\mu}}}{\mu a_3}\Big),\qquad
      R_2 = \frac{\sqrt{2-\mu-2\sqrt{1-\mu}}}{\mu a_3}.
    \end{align*}
    The curve $\gamma$ touches $\partial B_{R_1}(Q_1)$ exactly when $\kappa$ attains its maximum
    and it touches $\partial B_{R_2}(Q_2)$ precisely when $\kappa$ attains its minimum. 
  \end{prop}
  \begin{proof}
    
    We recall $\mu=\frac{4(1-k^2)}{(2-k^2)^2}$ and hence $k^2=\frac{2}{\mu}+\frac{2\sqrt{1-\mu}}{\mu}$.
    Then, choosing $P=(P_1,P_2)$ as in Proposition \ref{prop:choice_of_P} we obtain the inequalities
    \begin{align*}
      Z(s;P)
      &\leq -1 + \frac{\max \kappa}{\sqrt{\mu}} 
      \stackrel{\eqref{eq:def_kappa_orbitlike}}{=} 
      -1 + \frac{2}{\sqrt{2-k^2}\sqrt\mu}
      = -1 + \frac{\sqrt{2+2\sqrt{1-\mu}}}{\sqrt{\mu}}
      =:-1+\rho_1,   \\
      Z(s;P)
      &\geq -1 + \frac{\min \kappa}{\sqrt{\mu}} 
      \stackrel{\eqref{eq:def_kappa_orbitlike}}{=} 
      -1 + \frac{2\sqrt{1-k^2}}{\sqrt{2-k^2}\sqrt\mu}
      = -1 + \frac{\sqrt{2-2\sqrt{1-\mu}}}{\sqrt{\mu}}
      =: -1+\rho_2.
    \end{align*}
    Hence we deduce $\gamma(\R)\subset \overline{B_{R_1}(Q_1)}\sm B_{R_2}(Q_2)$ for 
    $$
      Q_1:= (P_1,\rho_1 P_2),\quad 
      Q_2:= (P_1,\rho_2 P_2),\quad 
      R_1:= \sqrt{\rho_1^2-1}|P_2|,\quad
      R_2:= \sqrt{\rho_2^2-1}|P_2|,
    $$
    which proves the result.
  \end{proof} 
  
  We point out that Eichmann determined the parameters of these circles with a different method without
  knowing explicit formulas for the elasticae, see section~5.4 in~\cite{Eich_nonuniqueness}.

 \subsection{Closed elasticae and winding numbers} \label{subsec:closed_elasticae}
 
 In this section we \r{briefly} discuss which orbitlike elasticae close up and how they can be
 characterized by their winding numbers. To this end we will use the ''rotation'' 
 $\Delta\theta_k$ from \eqref{eq:def_rotation}.
 
 \begin{prop}[see Proposition 5.3 in \cite{LanSing_the_total}] \label{prop:rotation}
   The function $k\mapsto \Delta\theta_k$ decreases on $(0,1)$ from its upper limit $\sqrt{2}\pi$ to its lower
   limit $\pi$.
 \end{prop}
 \begin{proof}
   
   We recall  
   $$
     \Delta\theta_k = \pi - \pi\Lambda_0(\arcsin(k'),k)+2\sqrt{1-k^2}\sqrt{2-k^2}K(k).
   $$
   The derivative of this function with respect to $k$ can be computed as follows. From the formulas
   \eqref{eq:derivative_EK},\eqref{eq:derivative_Heuman} we get
   \begin{align*}
     \frac{d}{d k} \big(\Lambda_0(\arcsin(k'),k)\big)
     &= -\frac{1}{\sqrt{1-k^2}} \frac{2(E(k)-(k')^4K(k))}{\pi\sqrt{1-(k')^4}} 
       + \frac{2(E(k)-K(k))k'}{\pi\sqrt{1-(k')^4}}  \\
     &= \frac{2}{\pi k'\sqrt{2k^2-k^4}}\big(-E(k)+(k')^4K(k)+(k')^2E(k)-(k')^2K(k)\big) \\
     &= -\frac{2k}{\pi k'\sqrt{2-k^2}}(E(k)+(k')^2K(k)), \\
    \frac{d}{d k} \big(\sqrt{1-k^2}\sqrt{2-k^2}K(k) \big)
     &=  -\frac{k\sqrt{2-k^2}}{\sqrt{1-k^2}}K(k)  -\frac{k\sqrt{1-k^2}}{\sqrt{2-k^2}}K(k) \\
     &\; +
     \sqrt{1-k^2}\sqrt{2-k^2} \frac{E(k)-(k')^2K(k)}{k(k')^2} \\
     &= \frac{\sqrt{2-k^2}}{k\sqrt{1-k^2}}E(k) + \frac{k^4-2}{k\sqrt{1-k^2}\sqrt{2-k^2}}K(k).
   \end{align*}
   From these identities we get
   $$
    \frac{d}{d k} \big( \Delta\theta_k\big) 
    = \frac{2(2E(k)+(k^2-2)K(k))}{k\sqrt{1-k^2}\sqrt{2-k^2}},
   $$
   and we have to show that the numerator is negative. In fact, the function $g(k):=2E(k)+(k^2-2)K(k)$
   satisfies $g(0)=0$ as well as $g'(k)=k (k')^{-2} (-E(k)+(k')^2K(k))<0$. Indeed, $h(k):=-E(k)+(k')^2K(k)$
   satisfies $h(0)=0$ as well as $h'(k)=-kK(k)<0$, see \eqref{eq:derivative_EK} for the formulas for
   $E'(k),K'(k)$. This proves $h<0$, hence $g'<0,g<0$ on $(0,1)$ and thus the
   monotonicity of $k\mapsto \Delta\theta_k$. Finally, evaluating $\Delta\theta_k$ for $k=0,k=1$ 
   and using 
   $K(0)=\frac{\pi}{2},\lim_{k\to 1} K(k)\sqrt{1-k^2}=0$ from 112.01 as well as
   $\Lambda_0(0,1)=0,\Lambda_0(\frac{\pi}{2},0)=1$ from 151.01 in \cite{ByrFri_handbook}
   we find the values $\sqrt 2 \pi$ respectively $\pi$ so that the claim is proved.
 \end{proof}
 
 In the following result we show that an orbitlike elastica $\gamma$ closes up if and
 only if $\Delta\theta_k$ is a rational multiple of $2\pi$. Recall that $\gamma$ and $k$ are related to
 each other by the curvature function~\eqref{eq:def_kappa_orbitlike}. Therefore we define $k_{m,n}\in (0,1)$
 to be the unique solution of the equation
 \begin{equation}\label{eq:closednedd_condition}
     \Delta\theta_k = \frac{2\pi m}{n}
 \end{equation} 
 where $m,n\in\N$ are natural numbers such that $1<\frac{2m}{n}<\sqrt 2$, see Proposition~\ref{prop:rotation}.
 The associated elastica with initial conditions $\gamma_1(0)=0,\gamma_2(0)=1,\phi(0)=0$ and $s_*=0$, say,
 will be denoted by $\gamma_{m,n}$. Notice that this choice is somewhat arbitrary since in fact none of the
 following results changes when these data are changed. We now prove that for $k=k_{m,n}$   the
 elastica is closed with hyperbolic length  
 $$
   \mathcal{L}_{m,n}:=2n\sqrt{2-k_{m,n}^2}K(k_{m,n}).
 $$ 
    
   \begin{prop}\label{prop:closed_and_winding_number}
     The elastica $\gamma$ is closed if and only if $k=k_{m,n}$
     for (w.l.o.g.) coprime integers $m,n\in\N$. In this case $\gamma=\gamma_{m,n}$ closes up after $n$
     periods of its curvature function and its winding number around
     $(P_1,P_2)=(\frac{b_3}{a_3},\frac{1}{\sqrt{\mu}a_3})$ is $m$.
   \end{prop}
   \begin{proof}
     
     We first assume that $\gamma$ is closed, i.e. we have $\gamma(s+T)=\gamma(s)$ for all
     $s\in\R$ and some $T>0$. From Proposition \ref{prop:choice_of_P} we get that $T$ must be a multiple of
     the period of the curvature function, i.e. $T=2n\sqrt{2-k^2}K(k)$ for some $n\in\N$. Additionally,
     Theorem \ref{thm:explicit_formula_orbitlike} implies $W_1(s)=W_1(s+T),W_2(s)=W_2(s+T)$ because of 
     $$
       \det\matII{a_1}{a_2}{b_1}{b_2}
       = a_1 b_2-b_1a_2 
       \stackrel{\eqref{eq:a_1a_2a_3b_1b_2b_3}}{=} \frac{a_1^2+a_2^2}{\sqrt\mu a_3} > 0.
     $$
     Hence, by Proposition~\ref{prop:W1W2_orbitlike}~(v), we have $n\Delta\theta_k\in
     2\pi\Z$. Since $\Delta\theta_k$ is positive we find an $m\in\N$ such that $n\Delta\theta_k=2\pi m$, i.e.
     $k=k_{m,n}$ and $T=\mathcal{L}_{m,n}$ for some $m,n\in\N$ such that $1<\frac{2m}{n}<\sqrt 2$. This proves
     the necessity of these conditions. Vice versa, $k=k_{m,n}$ is also sufficient
     for $\gamma_{m,n}$ to be closed. Indeed, from \eqref{eq:def_kappa_orbitlike} and
     Proposition~\ref{prop:W1W2_orbitlike}~(v) we get that $\kappa,W_1,W_2$ are $\mathcal{L}_{m,n}$-periodic. Hence, $\gamma$
     is $\mathcal{L}_{m,n}$-periodic thanks to the explicit formulas from
     Theorem~\ref{thm:explicit_formula_orbitlike}. 
     
     \medskip

     It remains to calculate the winding number of $\gamma:=\gamma_{m,n}$, set $k:=k_{m,n}$. To this end we
     use the homotopy invariance of the winding number and first calculate the winding number for a more
     convenient curve, namely (in complex notation)
     $$
       \eta(s):= \gamma_1(s)-\frac{b_3}{a_3} + i\cdot
       \Big(\gamma_2(s)\frac{\kappa(s)}{\sqrt{\mu}}-\frac{1}{\sqrt{\mu}a_3}\Big).
     $$
    From Theorem~\ref{thm:explicit_formula_orbitlike} we get, using the notation from the proof of
    Theorem~\ref{thm:explicit_formula_orbitlike},
     \begin{align*}
        \eta
       &= \frac{bW+b_3}{aW+a_3}-\frac{b_3}{a_3} + i \Big( \frac{1}{\sqrt\mu (aW+a_3)} - \frac{1}{\sqrt\mu
       a_3} \Big) \\
       &= \frac{bWa_3 - b_3aW}{a_3(aW+a_3)} - i \frac{aW}{\sqrt\mu a_3(aW+a_3)} \\
       &\stackrel{\eqref{eq:a_1a_2a_3b_1b_2b_3}}{=}
         \frac{a^\perp W}{\sqrt\mu a_3(aW+a_3)} - i \frac{aW}{\sqrt\mu a_3(aW+a_3)} \\
       &\stackrel{\eqref{eq:def_theta}}{=} \frac{a_2\sin(\vartheta)+a_1\cos(\vartheta)}{N}
       + i\cdot \frac{a_1\sin(\vartheta)-a_2\cos(\vartheta)}{N}
     \end{align*}
     where $N:= \frac{\sqrt\mu\kappa}{\sqrt{\kappa^2-\mu}} a_3(aW+a_3)$ is a positive and
     $2n\sqrt{2-k^2}K(k)$-periodic function. Hence, we get from
     Proposition~\ref{prop:properties_vartheta_orbitlike}~(ii)  
     \begin{align*}
       \frac{1}{2\pi i} \int_0^{2n\sqrt{2-k^2}K(k)} \frac{\eta'(s)}{\eta(s)}\,ds 
       &= \frac{1}{2\pi i} \int_0^{2n\sqrt{2-k^2}K(k)} \big( - \frac{N'(s)}{N(s)} + i\vartheta'(s)\big) \,ds\\
       &= \frac{1}{2\pi} \big(\vartheta(2n\sqrt{2-k^2}K(k))-\vartheta(0)\big)  \\
       &=  \frac{n}{2\pi}  \Delta\theta_k, \\
       &=  m.
     \end{align*}
     We now define the homotopy 
     $$
       (t,s)\mapsto \gamma_1(s)-\frac{b_3}{a_3} + i\cdot \Big(
       \gamma_2(s)\big(t\cdot \frac{\kappa(s)}{\sqrt\mu}+1-t\big)-\frac{1}{\sqrt\mu a_3} \Big) 
     $$
     that connects the curve $\eta$ (at $t=1$) and the curve $\gamma-P_1-iP_2$ (at $t=0$). We intend to show
     that this homotopy is admissible, i.e. it is zero-free. To this end we choose $s\in\R$  such that the 
     real part vanishes at $s$ and deduce that the imaginary part is non-zero. For a zero $s\in\R$ of the real
     part we have $a_3 \gamma_1(s)=b_3$, hence  
     \begin{align*}
       a^\perp W(s)
       &\stackrel{\eqref{eq:a_1a_2a_3b_1b_2b_3}}{=} \sqrt\mu a_3 (a_3bW(s) - b_3aW(s)) \\
       &= \sqrt\mu (a_3(bW(s)+b_3)-b_3(aW(s)+a_3)) \\
       &\stackrel{\eqref{eq:explicit_formula_orbitlike_1}}{=} \sqrt\mu  (aW(s)+a_3)(a_3\gamma_1(s)- b_3) \\
       &= 0.
     \end{align*}
     Due to \eqref{eq:a_1a_2a_3b_1b_2b_3} and Proposition~\ref{prop:W1W2_orbitlike}~(i) this implies 
     $aW(s)= \pm |a||W(s)|= \pm a_3  \sqrt{\kappa(s)^2-\mu}/\kappa(s)$ and thus
     $$
       \gamma_2(s) 
       \stackrel{\eqref{eq:explicit_formula_orbitlike_2}}{=} \frac{1}{\kappa(s)(aW(s)+a_3)}
       = \frac{1}{a_3(\pm\sqrt{\kappa(s)^2-\mu}+\kappa(s))}.
     $$
     From this we deduce
     \begin{align*}
       \gamma_2(s)\big(t\cdot \frac{\kappa(s)}{\sqrt\mu}+1-t\big)-\frac{1}{\sqrt\mu a_3}
       &= \frac{t(\kappa(s)-\sqrt\mu)+\sqrt\mu -(\pm\sqrt{\kappa(s)^2-\mu}+\kappa(s))}{\sqrt\mu
       a_3(\pm\sqrt{\kappa(s)^2-\mu}+\kappa(s))}    \\
       &= \frac{\sqrt{\kappa(s)-\sqrt\mu} \cdot \big( (t-1)\sqrt{\kappa(s)-\sqrt\mu}\mp
       \sqrt{\kappa(s)+\sqrt{\mu}} \big)}{\sqrt\mu a_3(\pm\sqrt{\kappa(s)^2-\mu}+\kappa(s))}
       \\
       &\neq 0
     \end{align*}
     because of $|(t-1)\sqrt{\kappa(s)-\sqrt\mu}|\leq \sqrt{\kappa(s)-\sqrt\mu}<\sqrt{\kappa(s)+\sqrt\mu}$.
     Hence, the homotopy is admissible so that the winding number of $\gamma$ around $P$ is $m$.
 \end{proof}

%

  \subsection{Selfintersections}
  
  In this section we compute the number of selfintersections for any given closed orbitlike elastica
  $\gamma_{m,n}$ from the last section. Such a selfintersection occurs if there are $s,t$ with $0\leq s<t<\mathcal{L}_{m,n}$
  and $\gamma_{m,n}(s)=\gamma_{m,n}(t)$. The number of $s$, for which such $t\in (s,\mathcal{L}_{m,n})$ 
  exist, is defined as the number of selfintersections. Our aim is to prove the
  assertion of Theorem~3 in \cite{LanSing_Curves_in_the}, namely that the number of selfintersections of the
  closed orbitlike elasticae $\gamma_{m,n}$ equals $n(m-1)$. Our computations
  even reveal that these intersections are simple in the sense that for any given such $s$ there is precisely
  one $t\in (s,\mathcal{L}_{m,n})$ with $\gamma(s)=\gamma(t)$. We remark that in the nonclosed case
  essentially the same proof shows that the elasticae selfintersect infinitely many times. We recall that
  for convenience we chose $\gamma_{m,n}$ such that the formulas for $\kappa,W_1,W_2,\vartheta$ hold
  with $s_*=0$. 

   \begin{prop}
     Let $m,n\in \N$ be coprime such that $1<\frac{2m}{n}<\sqrt 2$. Then the closed orbitlike elastica
     $\gamma_{m,n}$ has precisely $n(m-1)$ points of selfintersection characterized by the unique solutions
     $0\leq s<\mathcal{L}_{m,n}$ of the $n(m-1)$ equations
      \begin{equation}\label{eq:selfintersections}
         \frac{\vartheta(s)}{\pi} =\frac{(l-1)m}{n}-p 
      \end{equation} 
      where $l,p\in\Z$ satisfy the following conditions:
      \begin{align} \label{eq:selfintersection_lp_conditions}
        \begin{aligned}
        l\in\{1,\ldots,2n-1\},\qquad
        p\in \big\{1,\ldots, \Big\lceil \frac{\min\{l,2n-l\}m}{n}\Big\rceil -1 \big\}.
        \end{aligned}
      \end{align}
   \end{prop}
   \begin{proof}
     The equality $\gamma(s)=\gamma(t)$ for $0\leq s<t<\mathcal{L}_{m,n}$ is equivalent to
     $Z(s;P)=Z(t;P)$ for all $P\in\R^2,P_2\neq 0$. Proposition~\ref{prop:formula_Z_orbitlike} and
     Proposition~\ref{prop:choice_of_P} show that this is in turn equivalent to
     $\kappa(s)=\kappa(t),W_1(s)=W_1(t),W_2(s)=W_2(t)$. So our aim is to show that these equalities and
     $0\leq s<t<\mathcal{L}_{m,n}$ are equivalent to 
     \eqref{eq:selfintersections},\eqref{eq:selfintersection_lp_conditions}. The equation
     $\kappa(s)=\kappa(t)$ and the explicit form of the curvature function
     from~\eqref{eq:def_kappa_orbitlike} give 
     $$
       \text{(i)}\; s= t+2l\sqrt{2-k_{m,n}^2}K(k_{m,n}) \quad\text{or}\quad
       \text{(ii)}\; s =- t +2l\sqrt{2-k_{m,n}^2}K(k_{m,n})        
     $$
     for some $l\in\Z$. In the case (i) we have  by Proposition
     \ref{prop:W1W2_orbitlike}~(v) 
     $$
       W_1(s)+iW_2(s)
       = (W_1(t)+iW_2(t))e^{il\Delta\theta_{k_{m,n}}}
       = (W_1(s)+iW_2(s))e^{i\frac{2\pi lm}{n}}.
     $$
     Since $m,n$ are coprime and $s<t$ this implies $l\in -n\N$ and hence 
     $$
       s  
       = t  + 2l\sqrt{2-k_{m,n}^2}K(k_{m,n})
       \leq t  - 2n\sqrt{2-k_{m,n}^2}K(k_{m,n})
       = t-\mathcal{L}_{m,n}
       < 0,
     $$
     a contradiction. So this is impossible. In the case (ii) we get from
     Proposition~\ref{prop:W1W2_orbitlike}~(vi)
     \begin{align*} 
       i|W(s)|e^{i\vartheta(s)}
       &\stackrel{\eqref{eq:def_theta}}{=} i(W_2(s)-iW_1(s)) \\   
       &= W_1(s)+iW_2(s) \\
       &\stackrel{\ref{prop:W1W2_orbitlike}.(vi)}{=} 
         \big( -W_1(t) + i W_2(t)\big)e^{i(l-1)\Delta\theta_{k_{m,n}}}\\
       &= \big( -W_1(s) + i W_2(s)\big)e^{i\frac{2\pi(l-1)m}{n}}\\
       &= i\overline{\big( W_2(s) -i W_1(s)\big)}e^{i\frac{2\pi(l-1)m}{n}}\\
       &= i |W(s)| e^{i(-\vartheta(s)+\frac{2\pi(l-1)m}{n})}. 
     \end{align*} 
     This shows that intersection points occur precisely at those $s\in [0,\mathcal{L}_{m,n})$ where
     \eqref{eq:selfintersections} holds for some $p\in\Z$.
     Notice that $\vartheta$ is increasing by Proposition~\ref{prop:properties_vartheta_orbitlike} 
     so that a unique solution $s$ in $\R$ exists for any given $p\in\Z$. So it remains to show that the
     condition~\eqref{eq:selfintersection_lp_conditions} is equivalent to $0\leq s<t<\mathcal{L}_{m,n}$ 
     and that these conditions are satisfied for precisely $n(m-1)$ distinct pairs $(l,p)$. To this end
     we use Proposition~\ref{prop:properties_vartheta_orbitlike}~(i). Thanks to this result the condition
     $s\geq 0$ is equivalent to
     \begin{equation}\label{eq:selfintersections1}
       \vartheta(s) \geq \vartheta(0) = -\frac{\Delta\theta_k}{2}=-\frac{\pi m}{n}.
     \end{equation}
     The condition $t<\mathcal{L}_{m,n}$ is equivalent to $s>2(l-n)\sqrt{2-k^2}K(k)$ and thus to 
     \begin{equation}\label{eq:selfintersections2}
       \vartheta(s)
       > \vartheta(2(l-n)\sqrt{2-k_{m,n}^2}K(k_{m,n}))
       = \frac{2l-2n-1}{2}\Delta\theta_k
       = \frac{\pi(2l-2n-1)m}{n}.
     \end{equation}   
     The last condition $s<t$ holds if and only if $s<l\sqrt{2-k^2}K(k)$ and thus precisely if 
     \begin{equation}\label{eq:selfintersections3}
       \vartheta(s)<\vartheta(l\sqrt{2-k_{m,n}^2}K(k_{m,n}))= \frac{\pi(l-1)m}{n}.
     \end{equation}
     Altogether the conditions \eqref{eq:selfintersections1}--\eqref{eq:selfintersections3} and 
     $\vartheta(s) = \frac{\pi(l-1)m}{n} - \pi p$ lead to the conditions 
     $$
       0<p\leq \frac{lm}{n}\quad\text{and}\quad p< \frac{m(2n-l)}{n}. 
     $$
     Since $m,n$ are coprime these inequalities for integers $p,l\in\Z$ yield the conditions
     \eqref{eq:selfintersection_lp_conditions}. Notice that the solutions are
     different from each other because $m$ and $n$ are coprime. Indeed, if $(l,p),(\tilde l,\tilde p)$ are two
     pairs satisfying \eqref{eq:selfintersections},\eqref{eq:selfintersection_lp_conditions}, then $l-\tilde
     l$ is a multiple of $n$, say $l-\tilde l=qn$ for $q\in\N_0$. In the case $q\geq 2$ we get a
     contradiction due to 
     $$
       l = qn+\tilde l \geq 2n+\tilde l > 2n-1.
     $$
     In the case $q=1$ we find $m=p-\tilde p$ and thus a contradiction results from 
     $$
       p = \tilde p+m \geq 1+m > \Big\lceil \frac{\min\{l,2n-l\}m}{n}\Big\rceil -1.
     $$
     This proves that we have precisely as many intersection points as pairs $(l,p)$ satisfying
     \eqref{eq:selfintersection_lp_conditions}. So it remains to show that the number of these pairs is
     $n(m-1)$.
     By \eqref{eq:selfintersection_lp_conditions} the number of equations is 
     $$
       \sum_{l=1}^{2n-1} \big( \big\lceil \frac{\min\{l,2n-l\}m}{n}\big\rceil -1 \big) 
       =   2\sum_{l=1}^{n-1} \big\lceil \frac{lm}{n}\big\rceil + m -(2n-1)
     $$
     Since $m,n$ are coprime the sum can be computed as follows:
     \begin{align*}
       \sum_{l=1}^{n-1} \big\lceil\frac{lm}{n}\big\rceil
       &= \sum_{l=1}^{n-1} \big\lceil\frac{(n-l)m}{n}\big\rceil 
       \;=\; \sum_{l=1}^{n-1} \big( m + \big\lceil -\frac{lm}{n}\big\rceil \big) \\
       &= m(n-1) + \sum_{l=1}^{n-1} \big(-\big\lceil\frac{lm}{n}\big\rceil+1\big) \\ 
       &= (m+1)(n-1) - \sum_{l=1}^{n-1} \big\lceil\frac{lm}{n}\big\rceil.
     \end{align*}
     Hence, the number of equations is $m-(2n-1)+(m+1)(n-1)=n(m-1)$, which is all we had to show. 
   \end{proof}

   In Figure~\ref{table:elasticae} some conformally invariant quantities of the closed
   elasticae $\gamma_{m,n}$ for small $m,n$ are given. The table contains $k_{m,n}$, the hyperbolic length
   $\mathcal{L}_{m,n}$, the number of selfintersections $\mathcal{S}_{m,n}=n(m-1)$ and the Willmore energy 
   \begin{align*}
     \mathcal{W}_{m,n} 
     := \mathcal{W}(\Sigma_{\gamma_{m,n}}) 
     = \frac{\pi}{2} \int_0^{\mathcal{L}_{m,n}} \kappa^2 \,ds  
     = \frac{4n\pi E(k_{m,n})}{\sqrt{2-k_{m,n}^2}}.
   \end{align*}
   
    Given that all selfintersections of the $\gamma_{m,n}$ are simple, the energy levels
    $\mathcal{W}_{m,n}$ are much higher than the lower bounds given by the Li--Yau inequality from Theorem~6
    in~\cite{LiYau_conformal}. 
   
   \begin{figure}[h]  
\centering
\begin{tabular}{|c|c|c|c|c|c|c|c}
\hline
$n$ & $m$ & $k_{m,n}$ & $\mathcal{W}_{m,n}$ &  $\mathcal{L}_{m,n}$ & $\mathcal{S}_{m,n}$ \\
\hline
3 & 2 & 0.9362    & ~39.96 &  ~15.77 & ~~3  \\
5 & 3 & 0.9918   & ~63.83  &   ~34.80 & ~10  \\
7 & 4 & 0.9972   & ~88.58 &   ~55.81 & ~21 \\
8 & 5 & 0.9819   & 103.35 &    ~49.96 & ~32 \\
9 & 5 & 0.9986   & 113.54  & ~78.23 & ~36 \\
10 & 7 & 0.7463   & 138.23 &   ~45.78 & ~60 \\
11 & 6 & 0.9992 &  138.57 & 101.73  & ~55  \\
11 & 7 & 0.9745 &  143.08 & ~65.50 & ~66 \\
12 & 7 & 0.9954 & 152.33 &  ~90.20  & ~72  \\
13 & 7 & 0.9995 &  163.63 & 126.12 & ~78 \\
13 & 8 & 0.9865 & 167.09 & ~84.57 & ~91  \\
13 & 9 & 0.8349 &  177.95 & ~61.46 & 104 \\
14 & 9 & 0.9691 & 182.90 &   ~81.15  & 112  \\
\hline
\end{tabular} \qquad
\begin{tabular}{|c|c|c|c|c|c|c| }
\hline
$n$ & $m$ & $k_{m,n}$ & $\mathcal{W}_{m,n}$ &  $\mathcal{L}_{m,n}$ & $\mathcal{S}_{m,n}$ \\
\hline
15 & 8 & 0.9997 & 188.72 &   151.24  & 105   \\
16 & 9 & 0.9981 & 202.09 &   133.74 & 128   \\
16 & 11 & 0.8664 & 217.77 &  ~77.18 & 160  \\
17 & 9 & 0.9998 & 213.82 &   177.00  & 136   \\
17 & 10 & 0.9945 & 216.13 &   124.86  & 153   \\
17 & 11 & 0.9650 & 222.77 &   ~96.85  & 170   \\
17 & 12 & 0.5327 & 236.87 &   ~75.92  & 187   \\
18 & 11 & 0.9881 & 230.89 &    119.28 & 180   \\
19 & 10 & 0.9998 & 238.93 &   203.32  & 171   \\
19 & 11 & 0.9961 & 240.90 &   145.89  & 190   \\
19 & 12 & 0.9779 & 246.39 &   115.42  & 209   \\
19 & 13 & 0.8829 & 257.64 &   ~92.90  & 228   \\
20 & 11 & 0.9990 & 252.08 &  252.09 & 200 \\
\hline
\end{tabular} 
  \caption{Table of quantities associated with $\gamma_{m,n}$}
  \label{table:elasticae}
\end{figure}

\subsection{Stability}
   
   In this section we investigate stability properties of $\gamma_{m,n}$. As pointed out by Langer and
   Singer (see equation (1.4) in \cite{LanSing_the_total}) the stability of $\gamma_{m,n}$
   is determined by the integral
    $$
      I_{m,n}(\phi)
      := \int_0^{\mathcal{L}_{m,n}} 2\phi''(s)^2-(5\kappa(s)^2-4)\phi'(s)^2+
      (6\kappa'(s)^2-\kappa(s)^4+3\kappa(s)^2 + 2)\phi(s)^2 \,ds. 
    $$
    The functional $I_{m,n}$ measures the second variation of the Willmore energy $\mathcal{W}$ in
    direction $\phi N$ where $N$ is the normal vector field associated with $\gamma_{m,n}$ and
    $\phi$ is assumed to be $\mathcal{L}_{m,n}$-periodic and smooth. The elastica $\gamma_{m,n}$ is said to be
    stable if the above functional is nonnegative, otherwise we will say that it is unstable. 
    In the following we show that approximately for $0<k_{m,n}\leq
    0.6869145$ and $n$ large enough $\gamma_{m,n}$ is unstable. Our
    instability criterion results from an appropriate choice of a monochromatic test function, i.e. we choose
    $\phi(s):= \cos(2\pi js/\mathcal{L}_{m,n})$ for some $j\in\N_0$. 
    
    \begin{lem}\label{lem:instability}
      The curve $\gamma_{m,n}$ is unstable provided 
      \begin{align*}
        n > \frac{A(k_{m,n})+\sqrt{A(k_{m,n})^2+4B(k_{m,n})C(k_{m,n})}}{2C(k_{m,n})}
      \end{align*}
      where
      \begin{align*}
        A(k) &:= \frac{\pi\sqrt{5E(k)-(2-k^2)K(k)}}{(2-k^2)K(k)^{3/2}},\\
        B(k) &:=  \frac{\pi^2}{(2-k^2)K(k)^2}, \\
        C(k) &:=  \frac{\sqrt{(16(1-k^2)K(k)^2-44(2-k^2)E(k)K(k)+75E(k)^2)_+}}{\sqrt{3}(2-k^2)K(k)}.
      \end{align*}
    \end{lem} 
    \begin{proof}
      Set $k:=k_{m,n}, T:= \mathcal{L}_{m,n}/2\pi=n\sqrt{2-k^2}K(k)/\pi$. We define the Fourier
      coefficients $\alpha_j,\beta_j$ via
      \begin{align*}
        5\kappa(s)^2-4
        = \sum_{j=0}^\infty \alpha_j \cos\Big( \frac{js}{T}\Big),\qquad
        6\kappa'(s)^2-\kappa(s)^4+3\kappa(s)^2 +   2
        = \sum_{j=0}^\infty \beta_j \cos\Big( \frac{js}{T}\Big).
      \end{align*}
      From 110.03 respectively\footnote{We remark that 361.01 in \cite{ByrFri_handbook} contains a typo:
      $E(u)$ has to be replaced by $E(\sn(u,k),k)$ in our notation and by  
      $E(\arcsin(\sn u),k)$ in the notation used in \cite{ByrFri_handbook}.} 361.01 in \cite{ByrFri_handbook}
      we get
      \begin{align*}
        \int_0^{2\pi T} \kappa(s)^2\,ds
        &\stackrel{\eqref{eq:def_kappa_orbitlike}}{=} 
          \frac{4}{\sqrt{2-k^2}} \int_0^{2nK(k)} \dn(s,k)^2\,ds \\
        &= \frac{8n}{\sqrt{2-k^2}} \int_0^{K(k)} \dn(s,k)^2\,ds  \\
        &= \frac{8n E(k)}{\sqrt{2-k^2}}, \\
        \int_0^{2\pi T} \kappa'(s)^2\,ds
        &= \frac{4k^4}{(2-k^2)^{3/2}} \int_0^{2nK(k)} \sn(s,k)^2\cn(s,k)^2\,ds \\
        &= \frac{8nk^4}{(2-k^2)^{3/2}} \int_0^{K(k)} \sn(s,k)^2\cn(s,k)^2\,ds \\
        &= \frac{8n\big((2-k^2)E(k)-2(1-k^2)K(k)\big)}{3(2-k^2)^{3/2}}. 
      \end{align*}
      From these identities and $\kappa^4=4(\kappa^2-(\kappa')^2-\mu)$ we deduce
      \begin{align*}
        \alpha_0
        &= \frac{1}{2\pi T} \int_0^{2\pi T} (5\kappa(s)^2-4)\,ds 
        \;=\; \frac{20E(k)}{(2-k^2)K(k)}-4,\\
        \beta_0
        &= \frac{1}{2\pi T} \int_0^{2\pi T} (10\kappa'(s)^2-\kappa(s)^2 +   2+4\mu)\,ds \\
        &= 10\cdot \frac{2\big( (2-k^2)E(k)-4(1-k^2)K(k)\big)}{3(2-k^2)^2K(k)}
           - \frac{4E(k)}{(2-k^2)K(k)} + 2+4\mu \\
        &= \frac{28(2-k^2)E(k)-80(1-k^2)K(k)}{3(2-k^2)^2K(k)} +
           \frac{24-24k^2+2k^4}{(2-k^2)^2} \\
        &= \frac{28(2-k^2)E(k)+2(3k^2-2)(k^2+2)K(k)}{3(2-k^2)^2K(k)}.
      \end{align*}
      Now we calculate $I_{m,n}(\phi)$ for the function $\phi(s):= \cos(js/T)$ with $2j\notin
      n\N_0$.
      Since $\alpha_j=\beta_j=0$ for $j\notin n\N_0$ and 
      $$
        \phi(s)^2
        = \frac{1}{2}\big(1+\cos(2js/T)\big),\quad
        \phi'(s)^2=\frac{j^2}{2T^2}\big(1-\cos(2js/T)\big),\quad
        \phi''(s)^2 = \frac{j^4}{2T^4}\big(1+\cos(2js/T)\big)
      $$ 
      we may evaluate the integrals as follows:
      \begin{align*}
        \int_0^{\mathcal{L}_{m,n}} 2\phi''(s)^2\,ds
        &= \frac{2\pi^4 j^4}{(\sqrt{2-k^2}nK(k))^3}, \\
        \int_0^{\mathcal{L}_{m,n}} (5\kappa(s)^2-4)\phi'(s)^2\,ds
        &= \frac{\alpha_0 \pi^2 j^2 }{\sqrt{2-k^2}nK(k)}, \\
        \int_0^{\mathcal{L}_{m,n}} (6\kappa'(s)^2-\kappa(s)^4+3\kappa(s)^2 + 2)\phi(s)^2 \,ds
        &=  \beta_0  \sqrt{2-k^2}nK(k) 
      \end{align*}
      and thus 
      $$
        I_{m,n}(\phi) 
        = 2\sqrt{2-k^2}nK(k) \cdot \Big(
        \Big(\big(\frac{\pi j}{\sqrt{2-k^2}nK(k)}\big)^2- \frac{\alpha_0}{4}\Big)^2 +
        \frac{8\beta_0-\alpha_0^2}{16} \Big). 
      $$
      So the natural choice is 
      $$
        j=\big\lfloor \frac{\sqrt{\alpha_0}\sqrt{2-k^2}nK(k)}{2\pi} \big\rfloor \quad\text{or}\quad
        j= \lceil \frac{\sqrt{\alpha_0}\sqrt{2-k^2}nK(k)}{2\pi} \rceil\quad\text{s.t. } 2j\notin n\N_0.
      $$
      Since one of these $j$ in fact satisfies $2j\notin n\N_0$ 
      (because of $n\geq 3$ due to $\frac{2m}{n}\in (1,\sqrt 2),m\in\N$)
      we obtain $I_{m,n}(\phi)<0$ provided
      
      $$
        \Big(\big(\frac{\sqrt\alpha_0}{2}+\frac{\pi}{\sqrt{2-k^2}nK(k)}\big)^2- \frac{\alpha_0}{4}\Big)^2 +
        \frac{8\beta_0-\alpha_0^2}{16},
      $$
      which in turn is equivalent to 
      $$
         \frac{\pi\sqrt{\alpha_0}}{\sqrt{2-k^2}nK(k)} 
         + \frac{\pi^2}{(2-k^2)n^2K(k)^2}
         < \frac{\sqrt{(\alpha_0^2-8\beta_0)_+}}{4}.  
      $$
      Since this inequality can be rewritten as $\frac{A(k)}{n} + \frac{B(k)}{n^2}<C(k)$ we get the result.
    \end{proof}
    
    This criterion proves the instability of infinitely many of the $\gamma_{m,n}$ but also infinitely many of
    them are not covered. Only one of the curves $\gamma_{m,n}$ listed  in Table
    \ref{table:elasticae} (namely for $n=17,m=12$) can be proved to be unstable via Lemma~\ref{lem:instability}.
    Notice that the above criterion only works for $C(k)>0$ and this is only the case for $0\leq k<0.6869145$.
    With more computational effort the instability of more curves $\gamma_{m,n}$ can be shown both
    analytically and numerically. For instance it is possible to analytically determine the constrained
    minimizer of $I$ in the set of functions $\phi$ given by 
    $$
      \phi(s)=\sum_{j=0}^{n-1} a_j \cos\Big( \frac{js}{T}\Big) + \sum_{j=1}^{n-1}
         b_j \sin\Big( \frac{js}{T}\Big).
    $$
    A reasonable constraint is given by $a_0^2+b_0^2+\ldots+a_{n-1}^2+b_{n-1}^2=1$. The constrained minimizers are of the form
    $a_j \cos\big(\frac{js}{T}\big)+a_{n-j}\cos\big(\frac{(n-j)s}{T }\big)$ 
    or $b_j \sin\big(\frac{js}{T}\big)+b_{n-j}\sin\big(\frac{(n-j)s}{T }\big)$ for some
    $j\in\{1,\ldots,n-1\}$ so that there is a higher chance to find negative values for $I_{m,n}(\phi)$.
    However, the analytical expressions for these minima are much more complicated than the one from Lemma
    \ref{lem:instability} and we could not detect any substantial gain with respect to the monochromatic
    ansatz functions when applying the resulting criterion to elasticae with large $k$. Numerically, one
    clearly obtains much better results by choosing sufficiently large test function spaces 
    and the natural conjecture is that all closed orbitlike elasticae are unstable.

  \subsection{Approximating the Clifford torus}

  Let us finally provide an approximation result for the Clifford torus which we believe to be helpful
  in the understanding of Willmore surfaces of revolution and in particular of the curves $\gamma_{m,n}$ from
  Theorem~\ref{thm:LanSin}. By a gap theorem due to Mondino and Nguyen \cite{MonNgu_gap_theorem} one knows
  that the Clifford torus can not be approximated by Willmore tori with respect to the Willmore energy
  functional. As a consequence it may only be approximated in a weaker sense. We will use the notion
  of Hausdorff convergence defined by $\Sigma_n\to\Sigma$ in the Hausdorff sense if and only if
  $d_H(\Sigma_n,\Sigma)\to 0$ as $n\to\infty$ where 
  $$
    d_H(\Sigma',\Sigma)
    = \max \big\{ \sup_{z'\in\Sigma'}\inf_{z\in\Sigma} d_{\Hh}(z',z), \sup_{z\in\Sigma}\inf_{z'\in\Sigma'}
    d_{\Hh}(z',z)\big\}
  $$
  and $\Sigma,\Sigma',\Sigma_n$ are subsets of the (metric) hyperbolic space $(\Hh,d_{\Hh})$.
  Our previous considerations yield the following:
  
  \begin{lem} \label{lem:approximation_results}
    Let $(m_n)$ be a sequence such that $1<2m_n/n<\sqrt 2$ for all $n\in\N$. Then, as $n\to\infty$,
    the sequence $(\Sigma_{\gamma_{m_n,n}})$ diverges with respect to the Willmore energy and it converges
    (up to M\"obius transformations) to the Clifford torus in the Hausdorff sense if and only if
    $2m_n/n\to \sqrt{2}$.
    In this case, almost all of the Willmore surfaces are unstable.
  \end{lem}
  \begin{proof}
    From $\mathcal{W}_{m,n}\geq \sqrt{2}n\pi^2\to \infty$ as $n\to\infty$ we immediately get the
    divergence property.
    Proposition~\ref{prop:enclosures} shows that the set $\Sigma_{\gamma_{m_n,n}}=\gamma_{m,n}(\R)\subset\Hh$
    converges to the circle in the Hausdorff sense if and only if $\mu=\mu_{k_{m,n}}\to 1$. 
    Indeed, only in this cases the radii and the centers of the inner and outer circles from
    Proposition~\ref{prop:enclosures} converge to each other. Notice that applying suitable M\"obius
    transformations to $\Sigma_{\gamma_{m_n,n}}$ we can ensure that the $b_3,a_3$--parameters of
    $\gamma_{m_n,n}$ from Theorem~\ref{thm:explicit_formula_orbitlike}
    are the same for all $n$. By $\mu_{k_{m,n}}=4(1-k_{m,n}^2)/(2-k_{m,n}^2)^2$ we observe that
    $\mu_{k_{m_n}}\to 1$ happens precisely when $k_{m_n}\to 0$ which, by Proposition~\ref{prop:rotation}
    and the definition of $k_{m,n}$, is equivalent to $2m_n/n\to \sqrt{2}$. Lemma~\ref{lem:instability} shows
    that almost all $\gamma_{m,n}$ are unstable elasticae which proves the result.
  \end{proof}
  
   The consequence is that the Clifford torus may be well approximated by orbitlike surfaces of revolutions
   the shape of which, however, becomes more and more complicated as they approach the torus. Similarly, one
   can study the limit $2m_n/n\to 1$, which corresponds to $k_{m_n}\to 1$. In this case 
   Proposition~\ref{prop:enclosures} tells us that there is no Hausdorff convergence at all, but the catenoid
   is, as a subset of $\Hh$, contained in the set of accumlation points of the curves $\gamma_{m_n,n}(\R)$.
   The latter follows from the fact that the solution of the initial value problem~\eqref{eq:ODE_solution}
   converges locally (but not globally) uniformly towards the catenoid solution as the curvature
   functions approach $\sqrt 2 \sech(s)$ locally uniformly.

 \section{Wavelike elasticae}
 
   In this section we consider wavelike elasticae associated with curvature functions 
   \begin{equation} \label{eq:def_kappa_wavelike}
     \kappa(s) = \frac{2k}{\sqrt{2k^2-1}} \cn\Big( \frac{s+s_*}{\sqrt{2k^2-1}},k\Big),\qquad\text{where }
     \frac{1}{\sqrt 2}<k<1,\;s_*\in\R. 
   \end{equation}
   The functions $\phi$ and $Z(\cdot,P)$ are defined as before and we set
   $$
       \mu:= -(\kappa')^2+\kappa^2-\frac{1}{4}\kappa^4  = -\frac{4k^2(1-k^2)}{(2k^2-1)^2}.
   $$ 
   The derivation of explicit formulas for these elasticae follows the same lines as for the orbitlike ones.
   First we provide the counterpart of Proposition \ref{prop:w1w2_orbitlike}. Its proof will be given in
   section~\ref{subsec:proof_propositions_wavelike}.
   
   \begin{prop}\label{prop:w1w2_wavelike}
     For all $k\in (0,1)$ the ordinary differential equation $w''- \frac{2(1-k^2)}{\cn^2(s,k)}w=0$ has a
     fundamental system $\{w_1,w_2\}$ on $(-K(k),K(k))$ such that the following holds:
     \begin{itemize}
       \item[(i)] $w_1$ is odd about 0, $w_2$ is even about 0, $w_2(0)>0>w_1'(0)$,
       \item[(ii)] $\cn^2(s,k)(w_2(s)^2 - w_1(s)^2) = 1-k^2+(2k^2-1)\cn^2(s,k)$ on $(-K(k),K(k))$,
       \item[(iii)] $\cn^3(s,k)(w_2(s)w_2'(s) - w_1(s)w_1'(s)) = (1-k^2)\sn(s,k)\dn(s,k)$ on $(-K(k),K(k))$,
       \item[(iv)] $w_2(s)w_1'(s)-w_1(s)w_2'(s)= -k\sqrt{(1-k^2)(2k^2-1)}$ on $(-K(k),K(k))$, 
       \item[(v)] The functions $\cn(\cdot,k)w_1,\cn(\cdot,k)w_2$ have unique smooth extensions $\hat
       w_1,\hat w_2:\R\to\R$ such that the functions $w_1:=\frac{\hat
       w_1}{\cn(\cdot,k)},w_2:=\frac{\hat w_2}{\cn(\cdot,k)}$ build a fundamental system on
       each connected component of $\R\sm\{\kappa=0\}$. Moreover, $\hat w_2$ is positive on $\R$, $s\hat
       w_1(s)$ is negative on $\R\sm\{0\}$ and, as $s\to\pm\infty$, we have $\hat w_1(s)\to \mp\infty,\hat
       w_2(s)\to\infty$ exponentially and $\hat w_1(s)/\hat w_2(s)\to \mp 1$.  
     \end{itemize}
   \end{prop}
 
   Similarly to the orbitlike case we define
   
   \begin{align}\label{eq:def_Wj_wavelike}
     W_j(s) := \frac{1}{\sqrt{2k^2-1}}w_j\big( \frac{s+s_*}{\sqrt{2k^2-1}}\big),\quad 
     \hat W_j(s) := \frac{2k}{2k^2-1} \hat w_j\big( \frac{s+s_*}{\sqrt{2k^2-1}}\big).
     \qquad (j=1,2)
   \end{align}
   Notice that Part (v) of the previous proposition ensures that $W_1,W_2$ are well-defined and smooth on
   $\R\sm\{\kappa=0\}$ whereas $\hat W_1,\hat W_2$ are well-defined and smooth on $\R$ with $\hat
   W_j(s)=\kappa(s)W_j(s)$ whenever $\kappa(s)\neq 0$. The exact formulas for $\hat w_1,\hat w_2$ show that
   $\hat W_1,\hat W_2$ look like shifted and rescaled versions of $-\sinh$ and $\cosh$. As in the orbitlike
   case we collect some properties of $W_1,W_2,\kappa$. Since the proof is completely analogous to the one of
   Proposition~\ref{prop:W1W2_orbitlike} and the validity of the identities may be checked using a computer,
   we omit it.

   \begin{prop}\label{prop:W1W2_wavelike}
     Let $\kappa$ be given by \eqref{eq:def_kappa_wavelike}. Then on each connected component of $\R\sm\{0\}$
     the functions $W_1,W_2$ build a fundamental system of $W''+2\mu\kappa^{-2}W=0$ such that the following
     identities hold:
     \begin{itemize}
       \item[(i)] $W_2^2-W_1^2 = \frac{\kappa^2-\mu}{\kappa^2}$,
       \item[(ii)] $W_1'= \frac{\mu \kappa'}{\kappa(\kappa^2-\mu)} W_1
      - \frac{\sqrt{|\mu|}\kappa^2}{2(\kappa^2-\mu)}W_2$,
       \item[(iii)] $W_2'= - \frac{\sqrt{|\mu|}\kappa^2}{2(\kappa^2-\mu)}W_1+
       \frac{\mu \kappa'}{\kappa(\kappa^2-\mu)} W_2$,
       \item[(iv)] $W_2W_1'-W_1W_2' = - \frac{\sqrt{|\mu|}}{2}$.
   \end{itemize}
   Moreover, $\hat W_2$ is positive on $\R$ and we have $\hat W_1(s)/\hat W_2(s) \to \mp 1$ as
   $s\to\pm\infty$.
   \end{prop}

  The angular progress $\vartheta$ of a wavelike elastica can again be measured by evaluating an integral
  depending only on the curvature. By Proposition~\ref{prop:W1W2_wavelike}~(i)
  and~\eqref{eq:def_Wj_wavelike} there is a smooth function $\vartheta$ such that
  \begin{align} \label{eq:def_theta_wavelike}
      \hat W_2(s)-i\hat W_1(s) = \sqrt{\kappa(s)^2-\mu} \big( \cosh(\vartheta(s))+i
      \sinh(\vartheta(s)) \big).
  \end{align} 
  Again $\vartheta$ is thus defined only up to additive multiples of $2\pi$ so that we may require $\vartheta$
  to satisfy the inequality $-2\pi<\vartheta(-s_*)\leq 0$. In section~\ref{subsec:proof_propositions_wavelike}
  we will derive the following properties of $\vartheta$.

   \begin{prop}\label{prop:properties_vartheta_wavelike}
    For $W_1,W_2,\vartheta$ as above the following identities hold on $\R\sm\{\kappa=0\}$:
    $$
      \vartheta'
       = \frac{\sqrt{|\mu|} \kappa^2}{2(\kappa^2-\mu)},\quad
       \matII{W_1}{W_2}{W_1'}{W_2'}
      =  
      \matII{0}{\frac{\sqrt{\kappa^2-\mu}}{\kappa}}{-\frac{\sqrt{|\mu|}\kappa}{2\sqrt{\kappa^2-\mu}}}{
      \frac{\mu\kappa'}{\kappa^2\sqrt{\kappa^2-\mu}}}
      \matII{\cosh(\vartheta)}{-\sinh(\vartheta)}{-\sinh(\vartheta)}{\cosh(\vartheta)}.
    $$
  \end{prop}

  We mention that the formula for $\vartheta'$ also appears on p.569 in \cite{BryGri_Reduction}, but no
  detailed proof is given there. As in the orbitlike case the above results enable us to derive
  explicit formulas first for $Z(\cdot;P)$ (now for $P\in \C^2$) and then for the curve itself. The
  counterpart of Proposition~\ref{prop:formula_Z_orbitlike} reads as follows.  
   
   \begin{prop}\label{prop:formula_Z_wavelike}
     For $P\in\C^2,P_2\neq 0$ we have 
     \begin{equation} \label{eq:formula_Z_wavelike}
       Z(s;P)
       = -1 + A \hat W_1(s) + B\hat W_2(s) + C\kappa(s) 
     \end{equation}
     where $A,B,C\in\C$ are given by \eqref{eq:formula_AB},\eqref{eq:def_C} whenever $\kappa(s_0)\neq 0$.
   \end{prop}
   \begin{proof}
      Defining $W(s) := \frac{Z(s;P)+1}{\kappa(s)}-C$ one finds 
     as in Proposition~\ref{prop:formula_Z_orbitlike} $W''(s)+\frac{2\mu}{\kappa(s)^2}W(s)=0$ whenever
     $s+s_*\in ((2j-1)K(k),(2j+1)K(k))$ for some $j\in\Z$. With Proposition \ref{prop:W1W2_wavelike} we
     get \eqref{eq:formula_Z_wavelike} so that the formula for $Z,A,B$ follows as in the orbitlike case. Notice
     that the smoothness of $Z,\hat W_1,\hat W_2$ guarantees that the coefficients $A,B,C$ are the same on all
     intervals. 
   \end{proof}
   
   \begin{thm}\label{thm:explicit_formula_wavelike}
     A curve $\gamma:\R\to \Hh$ is a  wavelike elastica parametrized by hyperbolic arclength if and
     only if we have
     \begin{align} \label{eq:explicit_formula_wavelike}
       \begin{aligned}
       \gamma_1(s)
       &=\frac{b_1  \hat W_1(s)+b_2 \hat W_2(s)+b_3 \kappa(s)}{a_1 \hat W_1(s)+a_2 \hat W_2(s)+a_3\kappa(s)},   \\
       \gamma_2(s)
       &=\frac{1}{a_1 \hat W_1(s)+a_2\hat W_2(s) + a_3\kappa(s)}.
       \end{aligned}
     \end{align}
     for $a_1,a_2,a_3,b_1,b_2,b_3\in\R$ such that 
     \begin{align} \label{eq:wavelike_a_1a_2a_3b_1b_2b_3}
       \begin{aligned}
       &\text{either}\qquad a_3\neq 0, \;\; a_2=\sqrt{a_1^2+a_3^2},\; \;
       \vecII{b_1}{b_2} = \frac{b_3}{a_3} \vecII{a_1}{a_2} - \frac{1}{\sqrt{|\mu|}a_3} \vecII{a_2}{a_1} \\
       &\text{or }\qquad\quad\,
        a_3=0,\;\; a_2=|a_1|,\;\; b_1=\sign(a_1)b_2,\;\; |b_3| =  \frac{1}{\sqrt{|\mu|}}.
        \end{aligned} 
     \end{align}
     In this case the identities \eqref{eq:a1a2b1b2_1},\eqref{eq:a1a2b1b2_2} hold on $\R\sm\{\kappa=0\}$ and for all $s\in\R$ we have
     \begin{align} 
       a_3\gamma_1(s) &= b_3 +\frac{1}{2\mu}\big(\kappa(s)^2\sin(\phi(s))+2\kappa'(s)\cos(\phi(s))\big),  
       \label{eq:wavelike_b3a3_1}  \\
       a_3\gamma_2(s) &= \frac{\kappa(s)}{\mu} 
       + \frac{1}{2\mu}\big(2\kappa'(s)\sin(\phi(s))-\kappa(s)^2\cos(\phi(s))\big).
       \label{eq:wavelike_b3a3_2}
     \end{align}
   \end{thm}
   \begin{proof}
   The proof follows the same lines as in the orbitlike case, so we only sketch the main differences. 
     Assuming first that $\gamma$ is a wavelike elastica in $\Hh$ we get from
     Proposition~\ref{prop:formula_Z_wavelike}
     \begin{align*}
      (\gamma_1(s)-P_1)^2+\gamma_2(s)^2+P_2^2 
      = 2\gamma_2(s)P_2\big( A \hat W_1(s) + B  \hat W_2(s) +  C\kappa(s)\big) 
     \end{align*}
      for all $s\in\R$ and $A,B,C$ given by \eqref{eq:formula_AB},\eqref{eq:def_C}. As in the orbitlike case
      this leads to
     \begin{align*}
       1 = \gamma_2(s)\big(a_1 \hat W_1(s)+a_2\hat W_2(s)+a_3\kappa(s)\big),\quad
       \gamma_1(s) 
        &= \gamma_2(s)\big(b_1\hat W_1(s)+b_2\hat W_2(s)+b_3\kappa(s)\big)
     \end{align*}
     where the constants $a_1,a_2,a_3,b_1,b_2,b_3$ satisfy
     \eqref{eq:a1a2b1b2_1},\eqref{eq:a1a2b1b2_2},\eqref{eq:b3a3_1},\eqref{eq:b3a3_2} on $\R\sm\{\kappa=0\}$.
     The latter two equations imply \eqref{eq:wavelike_b3a3_1},\eqref{eq:wavelike_b3a3_2} for all $s\in\R$. 
     It remains to prove \eqref{eq:wavelike_a_1a_2a_3b_1b_2b_3} to finish the ''only if''-part. 
 
       \medskip

  To this end we introduce the matrices
  \begin{align*}
    \mathcal{W}:=\matII{W_1}{W_2}{W_1'}{W_2'},\qquad
    \mathcal{R}_h(\vartheta):=\matII{\cosh(\vartheta)}{-\sinh(\vartheta)}{-\sinh(\vartheta)}{\cosh(\vartheta)}.
  \end{align*}
  From Proposition~\ref{prop:properties_vartheta_wavelike} we get
    \begin{align*}
      \mathcal{W}
      &=  
      \matII{0}{\frac{\sqrt{\kappa^2-\mu}}{\kappa}}{-\frac{\sqrt{|\mu|}\kappa}{2\sqrt{\kappa^2-\mu}}}{\frac{\mu\kappa'}{\kappa^2\sqrt{\kappa^2-\mu}}}
      \mathcal{R}_h(\vartheta), \qquad
      \mathcal{W}^{-1}
      =  \mathcal{R}_h(-\vartheta)
      \matII{-\frac{2\sqrt{|\mu|}\kappa'}{\kappa^2\sqrt{\kappa^2-\mu}}}{-\frac{2\sqrt{\kappa^2-\mu}}{\sqrt{|\mu|}\kappa}}
      {\frac{\kappa}{\sqrt{\kappa^2-\mu}}}{0},  \\
    \intertext{and combining these formulas with the ones for $\xi^1,\xi^3$ from \eqref{eq:formula_xi} we get}
      \xi^1
      &=  \mathcal{R}_h(-\vartheta)
      \vecII{\frac{1}{2\sqrt{|\mu|(\kappa^2-\mu)}}(2\kappa'\cos(\phi)+\kappa^2\sin(\phi))}{
      \frac{1}{2\mu\sqrt{\kappa^2-\mu}}(2\mu-2\kappa^2-2\kappa\kappa'\sin(\phi)+\kappa^3\cos(\phi))
      },  \\
      \xi^3
      &=  \mathcal{R}_h(-\vartheta)
      \vecII{\frac{1}{2\sqrt{|\mu|(\kappa^2-\mu)}}(2\kappa'\sin(\phi)-\kappa^2\cos(\phi))}{
      \frac{\kappa}{2\mu\sqrt{\kappa^2-\mu}}(\kappa^2\sin(\phi)+2\kappa'\cos(\phi))
      }.
  \intertext{
  Writing $\skp{\cdot}{\cdot}_h$ for the inner product in $\Hh$ and $(w_1,w_2)^\perp=(w_2,w_1)$ we get from
  $(\kappa')^2=\kappa^2-\frac{\kappa^4}{4}-\mu$ and the fact that the rotation matrices
  $\mathcal{R}_h(\vartheta)$ are hyperbolic isometries}
  \skp{\xi^1}{\xi^1}_h
      &= -\frac{(2\kappa+2\kappa'\sin(\phi)-\kappa^2\cos(\phi))^2 }{4\mu^2}
      = - a_3^2\gamma_2^2,\\
      \skp{\xi^1}{\xi^3}_h
      &=
      \frac{(\kappa^2\sin(\phi)+2\kappa'\cos(\phi))(2\kappa+2\kappa'\sin(\phi)-\kappa^2\cos(\phi))}{4\mu^2}
      = \gamma_2 a_3(a_3\gamma_1-b_3),   \\
      \skp{(\xi^1)^\perp}{\xi^3}_h
      &= -\frac{2\kappa+2\kappa'\sin(\phi)-\kappa^2\cos(\phi)}{2\mu\sqrt{|\mu|}}
       = -\frac{a_3\gamma_2}{\sqrt{|\mu|}}.
  \end{align*}
   This implies
  \begin{align*}
    a_1^2-a_2^2
    &= \skp{a}{a}_h
     \stackrel{\eqref{eq:a1a2b1b2_1}}{=} \frac{1}{\gamma_2^2} \skp{\xi^1}{\xi^1}_h
     = -a_3^2, \\
    a_1b_1-a_2b_2
    &= \skp{a}{b}_h
    \stackrel{\eqref{eq:a1a2b1b2_1},\eqref{eq:a1a2b1b2_2}}{=} \frac{\gamma_1}{\gamma_2^2} \skp{\xi^1}{\xi^1}_h
    + \frac{1}{\gamma_2} \skp{\xi^1}{\xi^3}_h = -b_3a_3, \\
   a_2b_1-a_1b_2
    &= \skp{a^\perp}{b}_h
     \stackrel{\eqref{eq:a1a2b1b2_1},\eqref{eq:a1a2b1b2_2}}{=} \frac{1}{\gamma_2}\skp{(\xi^1)^\perp}{\xi^3}_h
    = -\frac{a_3}{\sqrt{|\mu|}}.
  \end{align*}
  Moreover, we have $a_2\geq |a_1|$ because, by Proposition~\ref{prop:W1W2_wavelike}, $\hat W_2$ is positive 
  and thus 
  $$
    0 \leq \limsup_{s\to\pm\infty} \gamma_2(s)\hat W_2(s) 
    = \limsup_{s\to\pm\infty} \frac{1}{a_1 \frac{\hat W_1(s)}{\hat
    W_2(s)}+a_2+\frac{a_3\kappa(s)}{\hat W_2(s)}}
    = \frac{1}{\mp a_1+a_2}. 
  $$ 
  Hence, in the case $a_3\neq 0$ we obtain the conditions from
  \eqref{eq:wavelike_a_1a_2a_3b_1b_2b_3}.
  In the case $a_3=0$ the above equations imply $a_2=|a_1|,b_1=\sign(a_1)b_2$ and 
  \eqref{eq:wavelike_b3a3_1},\eqref{eq:wavelike_b3a3_2} yield 
  \begin{align*}
    b_3^2
    &= \frac{(\kappa^2\sin(\phi)+2\kappa'\cos(\phi))^2}{4\mu^2}
    = \frac{\kappa^4+4(\kappa')^2-(2\kappa'\sin(\phi)-\kappa^2\cos(\phi))^2}{4\mu^2} \\
    &= \frac{\kappa^4+4(\kappa')^2-4\kappa^2}{4\mu^2}
    = \frac{1}{|\mu|}.
  \end{align*} 
  So \eqref{eq:wavelike_a_1a_2a_3b_1b_2b_3} is entirely proved.
  
  \medskip
   
  Now let $a_1,\ldots,b_3$ be given as in \eqref{eq:wavelike_a_1a_2a_3b_1b_2b_3} and
  $\gamma=(\gamma_1,\gamma_2)$ the associated elastica. In analogy to 
  \eqref{eq:formula_phi} the angle function $\phi$ is defined via
     \begin{equation*} 
         \vecII{\sin(\phi)}{\cos(\phi)}
         := \frac{\mu}{2(\kappa^2-\mu)}\matII{\kappa^2}{2\kappa'}{2\kappa'}{-\kappa^2}
         \vecII{a_3\gamma_1- b_3}{a_3\gamma_2-\frac{\kappa}{\mu}},
     \end{equation*}
   Such a definition is possible because of 
   $\kappa^2((aW)^2-(a^\perp W)^2)= \kappa^2(a_1^2-a_2^2)(W_1^2-W_2^2)=(\kappa^2-\mu)a_3^2$
     \begin{align*}
       &\frac{\mu^2}{4(\kappa^2-\mu)^2} \Big(
       \Big( \kappa^2\big(a_3\gamma_1- b_3\big)+2\kappa'\big(a_3\gamma_2-\frac{\kappa}{\mu
       }\big)\Big)^2
       + \Big(
       2\kappa'\big(a_3\gamma_1- b_3\big)-2\kappa^2\big(a_3\gamma_2-\frac{\kappa}{\mu}\big)   \Big)^2   \Big) \\
       &= \frac{\big( \kappa|\mu|(a_3bW-b_3aW)\big)^2+ \big(-\kappa^2aW +
       (\mu-\kappa^2)a_3  \big)^2}{\kappa^2(\kappa^2-\mu)(aW+a_3)^2} \\
       &= \frac{ |\mu|\kappa^2(a^\perp W)^2
         + \kappa^4(aW)^2 +2\kappa^2(\kappa^2-\mu)a_3aW+ (\kappa^2-\mu)^2a_3^2
        }{\kappa^2(\kappa^2-\mu)(aW+a_3)^2} \\
       &= \frac{ \kappa^2\mu |a|_h^2|W|_h^2
         + (\kappa^4-\kappa^2\mu)(aW)^2 +2\kappa^2(\kappa^2-\mu)a_3aW+ (\kappa^2-\mu)^2a_3^2
        }{\kappa^2(\kappa^2-\mu)(aW+a_3)^2} \\
       &= \frac{(aW)^2 +2a_3aW+ a_3^2}{(aW+a_3)^2} \\
       &= 1.
     \end{align*}
     It therefore remains to check that $(\gamma_1,\gamma_2,\phi)$ satisfies the ODE~\eqref{eq:ODE_solution}.
     We have
     \begin{align*}
       \sin(\phi)
       &= \frac{\mu}{2(\kappa^2-\mu)} \Big(
         \kappa^2\big( a_3\frac{bW+b_3}{aW+a_3}-b_3\big)
         +2\kappa'\big( \frac{a_3}{\kappa(aW+a_3)}-\frac{\kappa}{\mu}\big)
       \Big) \\
       &= \frac{\mu}{2(\kappa^2-\mu)} \Big(
         -\frac{\kappa^2 a^\perp W}{\sqrt{|\mu|} (aW+a_3)}
         + \frac{2\kappa'( (\mu-\kappa^2) a_3- \kappa^2 aW)}{\mu\kappa a_3(aW+a_3)}  \Big) \\
       &=  \frac{\sqrt{|\mu|} \kappa^3 a^\perp W
         - 2\kappa^2\kappa' aW
         - 2\kappa'(\kappa^2-\mu) a_3}{2\kappa(\kappa^2-\mu)(aW+a_3)},  \\
        \cos(\phi)
        &= \frac{2\sqrt{|\mu|} \kappa' a^\perp W + \kappa^3 aW +
        \kappa^2(\kappa^2-\mu)a_3}{2\kappa(\kappa^2-\mu)(aW+a_3)}.
     \end{align*}
     Next we use
     \begin{align*}
       W' = \frac{\mu\kappa'}{\kappa(\kappa^2-\mu)} W - \frac{\sqrt{|\mu|}\kappa^2}{2(\kappa^2-\mu)}W^\perp
     \end{align*}
     from Proposition~\ref{prop:W1W2_wavelike}. We obtain
     \begin{align*}
       \frac{\gamma_2'}{\gamma_2}
       &=  - \frac{\kappa'}{\kappa} - \frac{2\mu\kappa' aW +
       \sqrt{|\mu|} \kappa^3 a^\perp W }{2\kappa(\kappa^2-\mu)(aW+a_3)} \\
       &= \frac{-\sqrt{|\mu|} \kappa^3 a^\perp W - 2\kappa^2\kappa' aW
       - 2\kappa'(\kappa^2-\mu)a_3}{2\kappa(\kappa^2-\mu)(aW+a_3)} \\
       &= \sin(\phi), \\
       \frac{\gamma_1'}{\gamma_2}
       &= \kappa W' \frac{b(aW+a_3)-a(bW+b_3)}{aW+a_3} \\
       &= \kappa W' \frac{ (a^\perp W) a -(aW+a_3)a^\perp}{\sqrt{|\mu|} a_3(aW+a_3)} \\
       &= \frac{(-2\sqrt{|\mu|}\kappa' W - \kappa^3 W^\perp)
       ( (a^\perp W) a - (aW+a_3)a^\perp)}{2a_3(aW+a_3)(\kappa^2-\mu)}    \\
       &= \frac{-\kappa^3 (a^\perp W)^2+\kappa^3 (aW)^2
        + a_3(\kappa^3 aW + 2\sqrt{|\mu|}\kappa' a^\perp W)}{2a_3(aW+a_3)(\kappa^2-\mu)}   \\
        &= \frac{\kappa^3 |a|_h^2|W|_h^2
        + a_3(\kappa^3 aW + 2\sqrt{|\mu|}\kappa' a^\perp W)}{2a_3(aW+a_3)(\kappa^2-\mu)}     \\
        &= \frac{a_3\kappa(\kappa^2-\mu) + \kappa^3 aW + 2\sqrt{|\mu|}\kappa' a^\perp
        W}{2(aW+a_3)(\kappa^2-\mu)} \\
        &= \cos(\phi).
     \end{align*}
     The verification of $\phi'+\cos(\phi)=\kappa$ is the same as in the orbitlike case so that $\gamma$ is
     indeed a wavelike elastica.
   \end{proof}
%
        
   The oscillating behaviour of wavelike elasticae is again a consequence of a special choice for $P$ in
   Proposition~\ref{prop:formula_Z_wavelike}. For the moment we only consider the case $a_3\neq 0$ and discuss
   the exceptional case $a_3=0$ later in Proposition~\ref{prop:wavelike_a30}.
       
   \begin{prop} \label{prop:choice_of_P_wavelike}
      Assume $a_3\neq 0$. For $P=(\frac{b_3}{a_3},\pm \frac{i}{\sqrt{|\mu|}a_3})$ we have 
      $$
        Z(s;P) = -1 \pm  i \frac{\kappa(s)}{\sqrt{|\mu|}}.
      $$ 
   \end{prop}
   \begin{proof}
      As in the proof of the previous theorem we use $\kappa^2((a^\perp W)^2-(aW)^2 )
     = (\mu-\kappa^2)a_3^2$ by \eqref{eq:wavelike_a_1a_2a_3b_1b_2b_3} and
     Proposition~\ref{prop:W1W2_wavelike}~(i). Then the assertion follows from
     \begin{align*} 
      &\Big(\gamma_1-\frac{b_3}{a_3}\Big)^2+\Big(\gamma_2\mp \frac{i}{\sqrt{|\mu|}a_3}\Big)^2 \\
     &= \frac{(a^\perp W)^2}{|\mu|a_3^2(aW+a_3)^2} +
       \frac{|\mu|a_3^2-\kappa^2(aW+a_3)^2}{\kappa^2|\mu|a_3^2(aW+a_3)^2} \mp
       \frac{2i\gamma_2}{\sqrt{|\mu|}a_3} \\
       &= \frac{\kappa^2((a^\perp W)^2-(aW)^2) - \mu a_3^2-2\kappa^2
       a_3aW - \kappa^2a_3^2}{\kappa^2|\mu|a_3^2(aW+a_3)^2} \mp \frac{2i\gamma_2}{\sqrt{|\mu|}a_3} \\
       &= -\frac{2\kappa^2a_3^2-2\kappa^2 a_3aW}{\kappa^2|\mu|a_3^2(aW+a_3)^2} \mp
       \frac{2i\gamma_2}{\sqrt{|\mu|}a_3}    \\
       &= -\frac{2}{|\mu|a_3(aW+a_3)} \mp
       \frac{2i\gamma_2}{\sqrt{|\mu|}a_3}    \\
       &\stackrel{\eqref{eq:explicit_formula_orbitlike_2}}{=} 
         -\frac{2\kappa\gamma_2}{|\mu|a_3} \mp   \frac{2i\gamma_2}{\sqrt{|\mu|}a_3}    \\
       &= \pm \frac{2i\gamma_2}{\sqrt{|\mu|}a_3} \big( -1 \pm i
       \frac{\kappa}{\sqrt{|\mu|}} \big) \\
       &=2P_2\gamma_2\big( -1 \pm i
       \frac{\kappa}{\sqrt{|\mu|}} \big).
     \end{align*}
  \end{proof}

    These formulas for the distance function again allow us to deduce that the elastica is enclosed between
    two circles. It is remarkable that in the case of wavelike elasticae the centers of these circles lie on
    opposite sides of the horizontal axis and thus do not entirely belong to the hyperbolic space. We refer to
    section~5.3 in Eichmann's paper \cite{Eich_nonuniqueness} for a derivation of these bounds
    based on the analysis of Killing vector fields along the elasticae.  
     
  \begin{prop}\label{prop:oscillations_wavelike}
    Assume $a_3\neq 0$. Then we have $\gamma(\R)\subset \overline{B_{R_1}(Q_1)}\sm B_{R_2}(Q_2)$ where
    \begin{align*}
      Q_1 &= \Big(\frac{b_3}{a_3},~~~\frac{\sqrt{2+2\sqrt{1-\mu}}}{|\mu| |a_3|}\Big),\qquad
      R_1 = \frac{\sqrt{2-\mu+2\sqrt{1-\mu}}}{|\mu| |a_3|}, \\
      Q_2 &= \Big( \frac{b_3}{a_3},-\frac{\sqrt{2+2\sqrt{1-\mu}}}{|\mu| |a_3|}\Big),\qquad
      R_2 = \frac{\sqrt{2-\mu+2\sqrt{1-\mu}}}{|\mu| |a_3|}.
    \end{align*}
    The curve $\gamma$ touches $\partial B_{R_1}(Q_1)$ exactly when $\kappa$ attains its maximum
    and it touches $\partial B_{R_2}(Q_2)$ precisely when $\kappa$ attains its minimum. Moreover, we have
    \begin{align} \label{eq:limits_wavelike}
      \begin{aligned}
      \Big\{\lim_{s\to\infty} \gamma(s), \lim_{s\to -\infty} \gamma(s)\Big\}
      &= \partial B_{R_1}(Q_1) \cap \partial B_{R_2}(Q_2) \\ 
      &= \Big\{ \Big(\frac{b_3}{a_3}-\frac{1}{\sqrt{|\mu|} a_3},0\Big), 
      \Big(\frac{b_3}{a_3}+\frac{1}{\sqrt{|\mu|} a_3},0 \Big) \Big\}.
      \end{aligned} 
    \end{align} 
  \end{prop}
  \begin{proof}
    Let $P=(P_1,P_2)$ be given as in Proposition \ref{prop:choice_of_P_wavelike}. By definition
    of $Z(\cdot;P)$ and $P_2^2=-|P_2|^2$ we find
    $(\gamma_1-P_1)^2+\gamma_2^2-|P_2|^2=2\gamma_2|P_2|\frac{\kappa}{\sqrt{|\mu|}}$ and thus
    \begin{align*}
      (\gamma_1-P_1)^2+(\gamma_2+\rho|P_2|)^2 &\geq |P_2|^2(1+\rho^2), \\
      (\gamma_1-P_1)^2+(\gamma_2-\rho|P_2|)^2 &\leq |P_2|^2(1+\rho^2),\\
      \text{where }\quad
      \rho:= \frac{\max |\kappa|}{\sqrt{|\mu|}}
      &= \frac{2k}{\sqrt{2k^2-1}\sqrt{|\mu|}} = \frac{\sqrt{2+2\sqrt{1-\mu}}}{\sqrt{|\mu|}}.
    \end{align*}
    Here equality holds if and only if $\kappa$ attains its minimum/maximum, respectively. 
    So if we set
    $$
      Q_1:=(P_1,\rho|P_2|),\quad Q_2:=(P_1,-\rho|P_2|),\quad R_1:=R_2:=|P_2|\sqrt{1+\rho^2}
    $$
    then $\gamma(\R)\subset \overline{B_{R_1}(Q_1)}\sm B_{R_2}(Q_2)$, which is the first statement of the
    theorem. From the explicit formulas from Theorem \ref{thm:explicit_formula_wavelike} and
    Proposition~\ref{prop:W1W2_wavelike} we infer $\gamma_2(s) \to  0$ as $s\to \pm\infty$ and 
    \begin{align*}
      \lim_{s\to\pm\infty}
      \gamma_1(s) 
       \stackrel{\eqref{eq:explicit_formula_wavelike}}{=} 
        \lim_{s\to\pm\infty} \frac{b_1 \frac{\hat W_1(s)}{\hat W_2(s)}+b_2+b_3 \frac{\kappa(s)}{\hat W_2(s)}
         }{a_1    \frac{\hat W_1(s)}{\hat W_2(s)}+a_2+a_3 \frac{\kappa(s)}{\hat W_2(s)}} 
         = \frac{\mp b_1 +b_2}{\mp a_1 +a_2} 
       \stackrel{\eqref{eq:wavelike_a_1a_2a_3b_1b_2b_3}}{=} 
         \frac{b_3}{a_3} \pm \frac{1}{\sqrt{|\mu|} a_3}.
    \end{align*}
  \end{proof} 
  
  So we conclude as in \cite{LanSing_the_total} that wavelike elasticae with
  $a_3\neq 0$ oscillate around the halfcircle around $(b_3/a_3,0)$ with radius
  $1/(\sqrt{|\mu|}a_3)$. As we show now, a somewhat similar phenomenon occurs in the case $a_3=0$ . Here, 
  $\gamma$ is an unbounded curve lying inside a cone with apex  $(b_1/a_1,0)$ and opening angle
  $2\arctan( \sqrt{2+2\sqrt{1-\mu}}/\sqrt{|\mu|})$ and it oscillates around the vertical axis of the cone,
  see Figure~\ref{Fig:orbitlike_oscillations} or Fig.~3 in the Eichmann's paper \cite{Eich_nonuniqueness}
  where the same phenomenon has been observed. The precise statement for the case $a_3=0$ is the following:

  \begin{prop} \label{prop:wavelike_a30}
    Assume $a_3=0$. Then the elastica $\gamma$ from Theorem~\ref{thm:explicit_formula_wavelike} satisfies
    $$
      \gamma(\R) 
      \subset \mathcal{C}:= \Big\{ (x,y)\in\Hh : \frac{\big|x-\frac{b_1}{a_1}\big|}{y} \leq 
      \frac{\sqrt{2+2\sqrt{1-\mu}}}{\sqrt{|\mu|}} \Big\}.
    $$
    It touches the boundary of the cone $\mathcal{C}$ at points where $\kappa$ attains its maximum or minimum.
    Moreover, we have  
    $$
      \big\{\lim_{s\to\infty} \gamma_2(s), \lim_{s\to -\infty} \gamma_2(s)\big\} = \{0,\infty\}. 
    $$
  \end{prop}
  \begin{proof}
    The conditions~\eqref{eq:wavelike_a_1a_2a_3b_1b_2b_3} from Theorem~\ref{thm:explicit_formula_wavelike}
    imply 
    $$
      \gamma_1(s)
      = \frac{b_1(\hat W_1(s)+\sign(a_1) \hat W_2(s))+b_3 \kappa(s)}{a_1 \hat W_1(s)+|a_1| \hat W_2(s)}
      = \frac{b_1}{a_1} + \frac{\sign(b_3)\kappa(s) \gamma_2(s)}{\sqrt{|\mu|}}.
    $$
    This identity and 
    $$
      - \min_\R \kappa = \max_\R \kappa = \frac{2k}{\sqrt{2k^2-1}} = \sqrt{2+2\sqrt{1-\mu}}
    $$
    proves the first two assertions. Additionally, Proposition~\ref{prop:W1W2_wavelike} implies $\hat
    W_2(s)\to\infty$ as $s\to\pm\infty$ so that the claim follows from 
    $$
      \gamma_2(s)\hat W_2(s) 
      = \frac{\hat W_2(s)}{a_1\hat W_1(s)+|a_1|\hat W_2(s)}
      \to \frac{1}{\mp a_1+|a_1|} \quad\text{as }s\to\pm\infty. 
    $$ 
  \end{proof}

  \begin{figure}[!htb] 
    \centering 
    \subfigure[$a_3\neq 0$ 
    ]{
      \includegraphics[scale=.38]{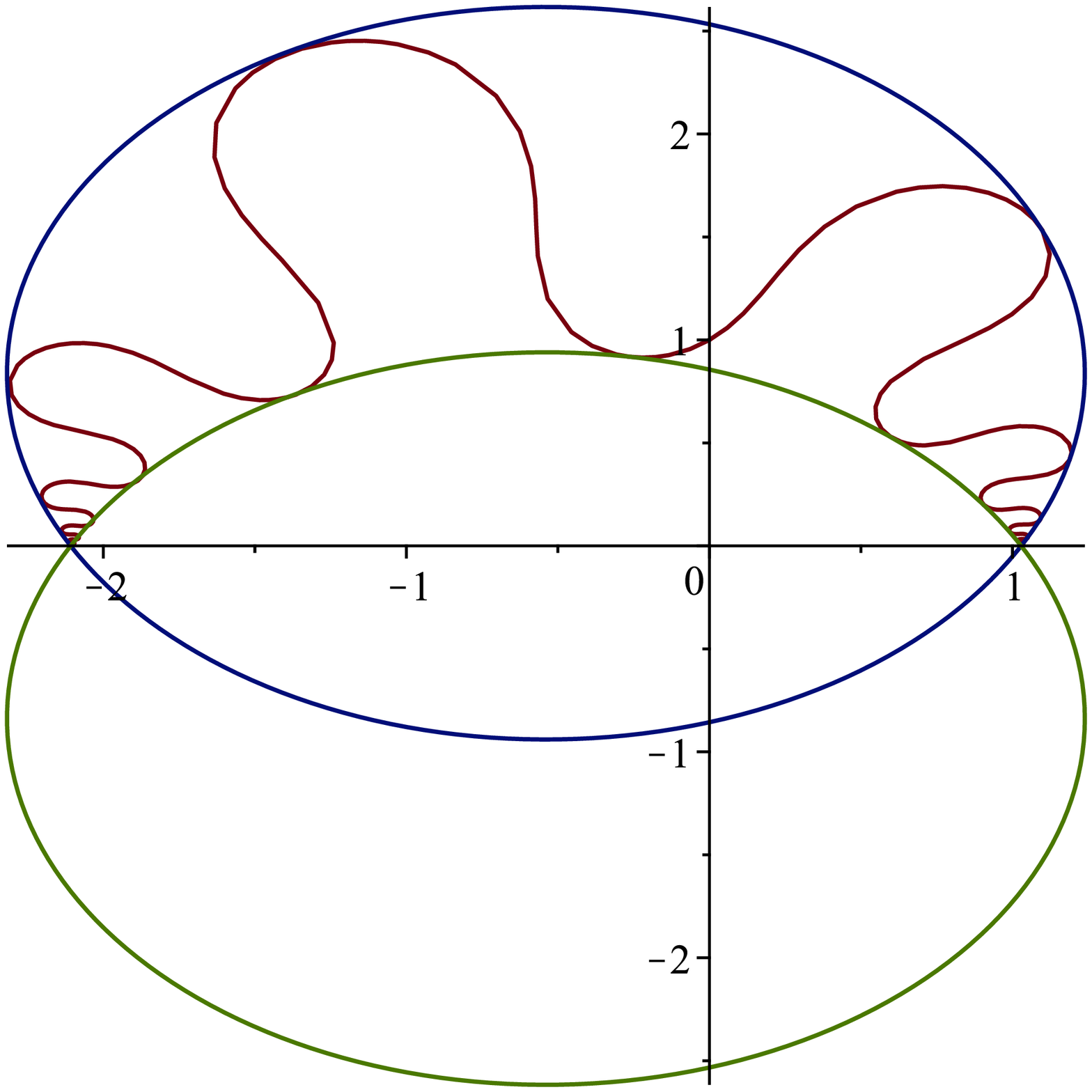} 
    }
    \;
    \subfigure[$a_3=0$ 
    ]{ 
      \includegraphics[scale=.38]{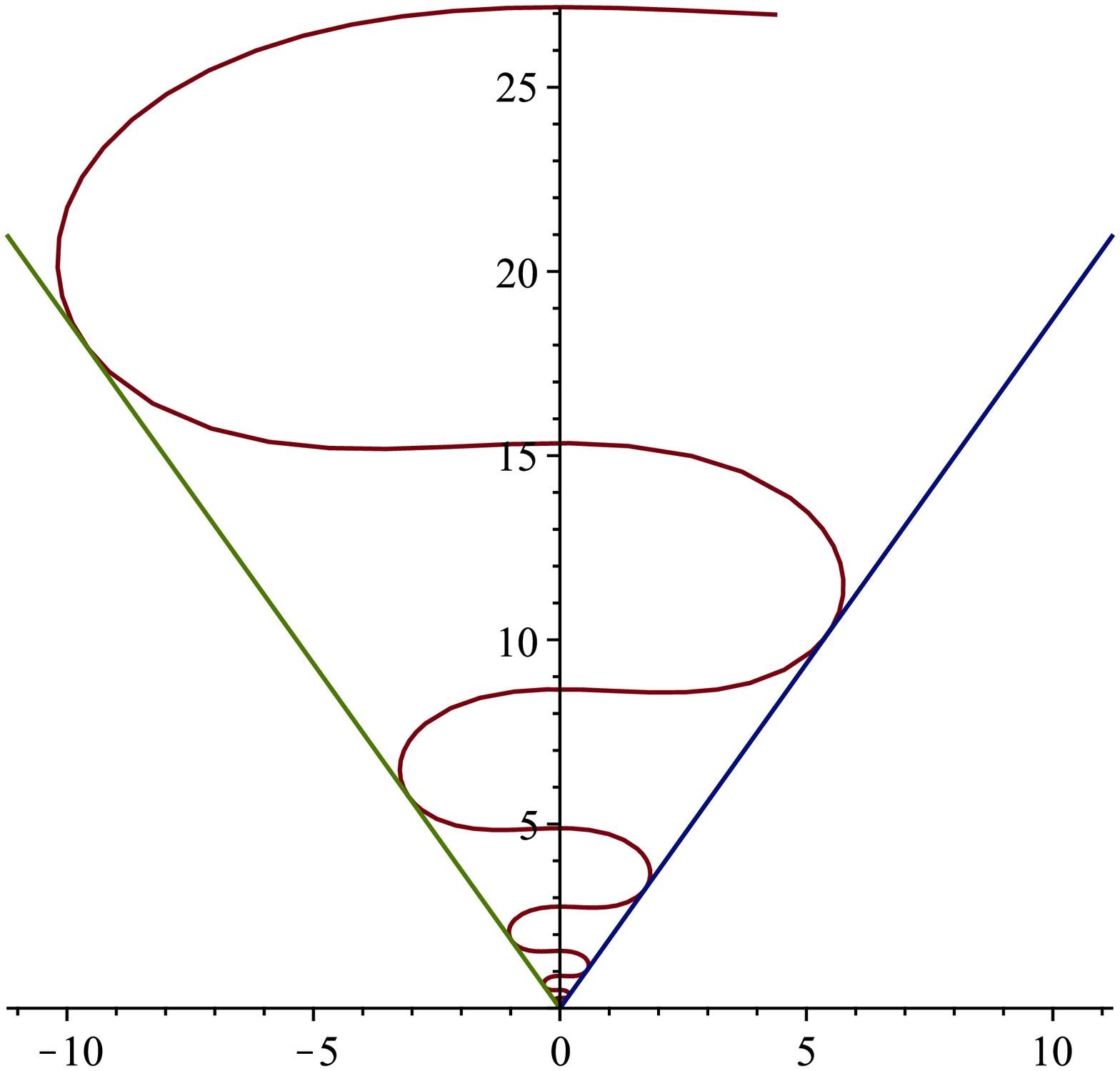} 
    } 
    \caption{Wavelike elasticae in $\Hh$ }
    \label{Fig:orbitlike_oscillations}
  \end{figure}

  \section{The Dirichlet boundary value problem, Proof of Theorem~\ref{thm:Symmetry_breaking}
  and~\ref{thm:Symmetry}} \label{sec:BVPs}
  
  In this section we investigate the Dirichlet problem for orbitlike elasticae. As in our previous
  considerations we assume the elastica to be positively oriented so that its curvature function is
  positive and given by \eqref{eq:def_kappa_orbitlike}.  Using the natural parametrization of elasticae this
  amounts to solving the equations
  \begin{equation} \label{eq:DirichletBVP2}
    \mathcal{W}'(\Sigma_\gamma)=0,\quad  \gamma(0)=(A_1,A_2),\quad \gamma(L)=(B_1,B_2),\quad
    e^{i\phi(0)}=e^{i\phi_A},\quad e^{i\phi(L)}=e^{i\phi_B} 
  \end{equation}
  for given initial/end positions $(A_1,A_2),(B_1,B_2)\in \Hh$ and directions measured
  by $\phi_A,\phi_B\in\R$. Recall that $\gamma,\phi$ are related to each other and to the curvature function
  $\kappa$ via the ODE system \eqref{eq:ODE_solution}. For the sake of shortness we only deal with orbitlike
  elasticae $\gamma$, since the analysis of the Dirichlet problem for wavelike elasticae can be done
  similarly. 
%
%
  We recall 
  \begin{equation} \label{eq:DirProblem_W_2W_1}
    \kappa(s)=\frac{2}{\sqrt{2-k^2}}\dn(\frac{s+s_*}{\sqrt{2-k^2}},k),\qquad
    W_2(s)-iW_1(s) = \frac{\sqrt{\kappa(s)^2-\mu}}{\kappa(s)}e^{i\vartheta(s)}.
  \end{equation}
  It is convenient to use the following shorthand notations:
   
  \begin{align}\label{eq:sigmas}
    \begin{aligned}
     &\sigma_1 := \frac{1-\mu(a_3^2A_2^2+(a_3A_1-b_3)^2)}{2A_2a_3}, 
     &&\sigma_2:= \sqrt\mu (A_1a_3-b_3),\\
     &\sigma_3 :=   \frac{1-\mu(a_3^2B_2^2+(a_3B_1-b_3)^2)}{2B_2a_3}, 
     &&\sigma_4:=\sqrt\mu (B_1a_3-b_3).
     \end{aligned}
  \end{align}
  The following theorem reduces the Dirichlet problem for orbitlike elasticae to a
  system of two equations with two unknowns. Although this system is rather complicated, we will see that
  it may be significantly simplified when special cases are considered.  
  The strategy is to determine the coefficients $a_1,\ldots,b_3$ from
  Theorem~\ref{thm:explicit_formula_orbitlike} as well as the hyperbolic length $L$ in dependence of $s_*,k$
  and the Dirichlet data so that two equations for $s_*,k$ are left to solve. For notational
  convenience we suppress the dependency of the coefficients $a_1,\ldots,b_3$ and of
  $\sigma_1,\ldots,\sigma_1,\kappa(0),\kappa'(0),\mu$ on $s_*\in\R$ and $k\in (0,1)$.
  
  \begin{thm}  \label{thm:DirProblemI}
    Let $(A_1,A_2),(B_1,B_2)\in\Hh,\phi_A,\phi_B\in\R$ be given. 
    For $s_*\in\R,k\in (0,1)$ let the functions $\kappa,\vartheta$ be defined by
    \eqref{eq:def_kappa_orbitlike},\eqref{eq:def_theta}. Then the elastica given by
    \eqref{eq:explicit_formula_orbitlike_1},\eqref{eq:explicit_formula_orbitlike_2} solves the Dirichlet
    problem \eqref{eq:DirichletBVP2} if and only if there is an $l\in\Z$ such that
    \begin{align} 
      &a_3 = \frac{2\kappa(0)+2\kappa'(0)\sin(\phi_A)-\kappa(0)^2\cos(\phi_A)}{2\mu A_2}, 
      \label{eq:DirProblem_choice_a3}\\
      &b_3 = a_3 A_1 - \frac{\kappa(0)^2\sin(\phi_A)+2\kappa'(0)\cos(\phi_A)}{2\mu}, 
      \label{eq:DirProblem_choice_b3}  \\
     & a_1+ia_2 = a_3
     \frac{\sigma_2+i\sigma_1}{\sqrt{\sigma_1^2+\sigma_2^2}}e^{i\vartheta(0)}, 
     \label{eq:DirProblem_choice_a1a2}\\
     &b_1+ib_2 = \big( \frac{b_3}{a_3}+\frac{i}{\sqrt\mu a_3}\big)(a_1+ia_2)
     \label{eq:DirProblem_choice_b1b2} \\
     &L= -s_*+\sqrt{2-k^2}\Big( 2lK(k)-\sigma\dn^{-1}\Big(\frac{\sqrt{2-k^2}}{2}\cdot 
         \frac{1+\mu(B_2^2a_3^2+(a_3B_1-b_3)^2)}{2B_2a_3} ,k\Big)\Big)
         \label{Gl:DirProblem_III}
     \intertext{for $\sigma:= \sign\Big(\mu(a_3B_1-b_3)\cos(\phi_B) -\frac{1+\mu(-a_3^2B_2^2+ (a_3B_1-b_3)^2)}{2B_2a_3}
      \sin(\phi_B)\Big)$ and $s_*,k$ satisfy the equations}
      &\Big(\frac{1+\mu(B_2^2a_3^2+(a_3B_1-b_3)^2)}{2B_2a_3}\Big)^2 
      = 2\mu(a_3B_1-b_3)\sin(\phi_B)  \notag \\
     &\hspace{6.2cm} + \frac{1+\mu(-B_2^2a_3^2+ (a_3B_1-b_3)^2)}{B_2a_3}\cos(\phi_B),  
      \label{Gl:DirProblem_I}\\
     &e^{i(\vartheta(L)-\vartheta(0))}
      = \frac{(\sigma_1-i\sigma_2)(\sigma_3+i\sigma_4)}{
       \sqrt{(\sigma_1^2+\sigma_2^2)(\sigma_3^2+\sigma_4^2)}}. \label{Gl:DirProblem_II} 
   \end{align}
  \end{thm}
  \begin{proof}
    In a first step we prove that every solution of \eqref{eq:DirichletBVP2} satisfies
    \eqref{eq:DirProblem_choice_a3}--\eqref{Gl:DirProblem_II} for some $l\in\Z$. In the
    second step we show that these conditions indeed provide solutions of the Dirichlet problem. So let
    $\gamma$ be a solution of the Dirichlet problem with angle function $\phi$. By 
    Theorem~\ref{thm:explicit_formula_orbitlike} the formulas
    \eqref{eq:b3a3_1},\eqref{eq:b3a3_2} and  $e^{i\phi(0)}=e^{i\phi_A},e^{i\phi(L)}=e^{i\phi_B}$ yield
    \begin{align*}
      a_3 &= \frac{2\kappa(0)+2\kappa'(0)\sin(\phi_A)-\kappa(0)^2\cos(\phi_A)}{2\mu A_2}, \\
      a_3 &= \frac{2\kappa(L)+2\kappa'(L)\sin(\phi_B)-\kappa(L)^2\cos(\phi_B)}{2\mu B_2}, \\
      b_3 &= a_3 A_1 - \frac{\kappa(0)^2\sin(\phi_A)+2\kappa'(0)\cos(\phi_A)}{2\mu}, \\
      b_3 &= a_3 B_1 - \frac{\kappa(L)^2\sin(\phi_B)+2\kappa'(L)\cos(\phi_B)}{2\mu}.
    \end{align*}
    These equations imply \eqref{eq:DirProblem_choice_a3},\eqref{eq:DirProblem_choice_b3} as well as
    \begin{align}
      \matII{\cos(\phi_A)}{-\sin(\phi_A)}{\sin(\phi_A)}{\cos(\phi_A)}\vecII{\kappa(0)^2}{2\kappa'(0)}
      &= \vecII{2\kappa(0)-2\mu A_2 a_3}{2\mu(a_3A_1-b_3)}, 
      \label{eq:DirProblem_Formula1a}\\
      \matII{\cos(\phi_B)}{-\sin(\phi_B)}{\sin(\phi_B)}{\cos(\phi_B)}\vecII{\kappa(L)^2}{2\kappa'(L)}
      &= \vecII{2\kappa(L)-2\mu B_2 a_3}{2\mu(a_3B_1-b_3)}.
      \label{eq:DirProblem_Formula1b}
    \end{align}
    Taking the modulus on each side and using $\kappa^4+4(\kappa')^2=4(\kappa^2-\mu)$ we find
    that the quadratic term $4\kappa(0)^2,4\kappa(L)^2$ cancels and we get
    \begin{align}
       \kappa(0) &= \frac{1+\mu(a_3^2A_2^2+(a_3A_1-b_3)^2)}{2A_2a_3},
       \label{eq:DirProblem_Formula2a} \\
       \kappa(L) &= \frac{1+\mu(a_3^2B_2^2+(a_3B_1-b_3)^2)}{2B_2a_3}.
       \label{eq:DirProblem_Formula2b}
    \end{align}
    Plugging these formulas into \eqref{eq:DirProblem_Formula1a},\eqref{eq:DirProblem_Formula1b} we 
    may solve for $\kappa(0)^2,2\kappa'(0),\kappa(L)^2,2\kappa'(L)$ and find
    \begin{align}  
       \kappa(0)^2
    	 &= 2\mu(a_3A_1-b_3)\sin(\phi_A)+ \frac{1+\mu(-a_3^2A_2^2+ (a_3A_1-b_3)^2)}{A_2a_3}\cos(\phi_A),   
    	 \label{eq:DirProblem_Formula3a} \\
      \kappa(L)^2
    	 &= 2\mu(a_3B_1-b_3)\sin(\phi_B)+ \frac{1+\mu(-a_3^2B_2^2+ (a_3B_1-b_3)^2)}{B_2a_3}\cos(\phi_B),   
    	 \label{eq:DirProblem_Formula3b}
    	 \\
    	2\kappa'(0)
    	&= 2\mu(a_3A_1-b_3)\cos(\phi_A) -\frac{1+\mu(-a_3^2A_2^2+ (a_3A_1-b_3)^2)}{A_2a_3} \sin(\phi_A),
    	\label{eq:DirProblem_Formula4a} \\	 
    	2\kappa'(L)
    	&= 2\mu(a_3B_1-b_3)\cos(\phi_B) -\frac{1+\mu(-a_3^2B_2^2+ (a_3B_1-b_3)^2)}{B_2a_3} \sin(\phi_B).
    	\label{eq:DirProblem_Formula4b}
    \end{align}
    Since the values for $\kappa(L)^2$ given by
    \eqref{eq:DirProblem_Formula2b},\eqref{eq:DirProblem_Formula3b} have to coincide, we get
    \eqref{Gl:DirProblem_I} as a necessary condition. Moreover, using the explicit formula for $\kappa$ from
    \eqref{eq:DirProblem_W_2W_1} as well as \eqref{eq:DirProblem_Formula2b},\eqref{eq:DirProblem_Formula4b} we
    obtain the formula for the hyperbolic length $L$ from \eqref{Gl:DirProblem_III}.
   
   \medskip

   It remains to derive
   \eqref{eq:DirProblem_choice_a1a2},\eqref{eq:DirProblem_choice_b1b2},\eqref{Gl:DirProblem_II}. 
   The equation \eqref{eq:DirProblem_choice_b1b2} is a direct consequence of \eqref{eq:a_1a_2a_3b_1b_2b_3}.
   From $\gamma(0)=(A_1,A_2)$ and $\gamma(L)=(B_1,B_2)$ and the explicit formulas for $\gamma_1,\gamma_2$ from
   Theorem~\ref{thm:explicit_formula_orbitlike} and in particular the relation
   $b=\frac{b_3}{a_3}a+\frac{1}{\sqrt\mu a_3}a^\perp$ from \eqref{eq:a_1a_2a_3b_1b_2b_3} imply 
   \begin{align*}
     \matII{a_1}{a_2}{-a_2}{a_1}\vecII{W_1(0)}{W_2(0)}
     =
     \frac{1}{A_2\kappa(0)}  \vecII{\sigma_1}{\sigma_2}, \qquad
     \matII{a_1}{a_2}{-a_2}{a_1}\vecII{W_1(L)}{W_2(L)}
     = \frac{1}{B_2\kappa(L)}\vecII{\sigma_3}{\sigma_4}.
   \end{align*}
   Notice that we used here $\sigma_1=1-a_3A_2\kappa(0),\sigma_2=1-a_3A_2\kappa(L)$, see
   \eqref{eq:sigmas} and \eqref{eq:DirProblem_Formula2a},\eqref{eq:DirProblem_Formula2b}.
   Using $a_1^2+a_2^2=a_3^2$ from~\eqref{eq:a_1a_2a_3b_1b_2b_3} we may solve this linear problem and arrive at
   \begin{align} \label{eq:DirProblem_equationsforW_1W_2}
	\vecII{W_1(0)}{W_2(0)}
	= \frac{1}{a_3^2A_2\kappa(0)}\vecII{ a_1\sigma_1 - a_2\sigma_2}{ a_2\sigma_1 + a_1\sigma_2},
	\qquad
      \vecII{W_1(L)}{W_2(L)}
      = \frac{1}{a_3^2B_2\kappa(L)}  \vecII{ a_1\sigma_3 - a_2\sigma_4}{ a_2\sigma_3 + a_1\sigma_4}.
    \end{align}
   From \eqref{eq:DirProblem_W_2W_1} and \eqref{eq:DirProblem_equationsforW_1W_2} we get
   \begin{align*}
     e^{i\vartheta(0)}
     =  \frac{a_2\sigma_1+a_1\sigma_2+i(a_2\sigma_2-a_1\sigma_1)}{
     \sqrt{(a_1^2+a_2^2)(\sigma_1^2+\sigma_2^2)}},   \qquad
     e^{i\vartheta(L)}
     =  \frac{a_2\sigma_3+a_1\sigma_4+i(a_2\sigma_4-a_1\sigma_3)}{
     \sqrt{(a_1^2+a_2^2)(\sigma_3^2+\sigma_4^2)}}.   
   \end{align*}
   The first equation and $a_1^2+a_2^2=a_3^2$ imply \eqref{eq:DirProblem_choice_a1a2}. Dividing 
   both equations we finally end up with \eqref{Gl:DirProblem_II}. This finishes the ''only if''-part of our
   statement. 
   
   \medskip
   
   Now we assume the conditions \eqref{eq:DirProblem_choice_a3}--\eqref{Gl:DirProblem_II} to hold and
   define $\gamma_1,\gamma_2$ by the formulas
   \eqref{eq:explicit_formula_orbitlike_1},\eqref{eq:explicit_formula_orbitlike_2}. Then, by definition of
   $a_1,a_2,b_1,b_2$, the conditions \eqref{eq:a_1a_2a_3b_1b_2b_3} hold so that $(\gamma_1,\gamma_2)$ is an
   orbitlike elastica by Theorem~\ref{thm:explicit_formula_orbitlike} with angle function $\phi$ 
   that satisfies \eqref{eq:formula_phi}. Exploiting the identities~\eqref{eq:b3a3_1},\eqref{eq:b3a3_2} 
   we have
    \begin{align}\label{eq:sigma12_sq}
      \begin{aligned}
      \sigma_1^2+\sigma_2^2
      &= (1-A_2a_3\kappa(0))^2 + \mu(A_1a_3-b_3)^2 \\
      &=
      \frac{(2\mu-2\kappa(0)^2-2\kappa(0)\kappa'(0)\sin(\phi_A)+\kappa(0)^2\cos(\phi_A))^2}{4\mu^2} \\
      &+ \frac{\mu(\kappa(0)^2\sin(\phi_A)+2\kappa'(0)\cos(\phi_A))^2}{4\mu^2}\\
      &= \frac{(\kappa(0)^2-\mu)(2\kappa(0)+2\kappa'(0)\sin(\phi_A)-\kappa(0)^2\cos(\phi_A))^2}{4\mu^2} \\
      &= (\kappa(0)^2-\mu)A_2^2a_3^2,
      \end{aligned} 
    \end{align}
    where we used once again $4(\kappa')^2=4\kappa^2-\kappa^4-4\mu$. 
    This entails 
    \begin{align}\label{eq:aW(0)}
      \begin{aligned}
      &\;aW(0)+ia^\perp W(0) \\
      &= (W_2(0)-iW_1(0))(a_2+ia_1) \\
      &\stackrel{\eqref{eq:def_theta},\eqref{eq:DirProblem_choice_a1a2}}{=} 
        \frac{\sqrt{\kappa(0)^2-\mu}}{\kappa(0)}e^{i\vartheta(0)}\cdot  a_3
        \frac{\sigma_1+i\sigma_2}{\sqrt{\sigma_1^2+\sigma_2^2}}e^{-i\vartheta(0)}   \\
      &\stackrel{\eqref{eq:sigma12_sq}}{=} \frac{\sigma_1+i\sigma_2}{\kappa(0)A_2}. 
      \end{aligned}
    \end{align}
    From this identity and  
    $\sigma_1=1-a_3A_2\kappa(0),\sigma_2= \sqrt\mu (A_1a_3-b_3)$ we deduce  
    \begin{align*}
       \gamma_2(0)
      &\stackrel{\eqref{eq:explicit_formula_orbitlike_2}}{=} \frac{1}{\kappa(0)(aW(0)+a_3)}
      = \frac{1}{\kappa(0)(\frac{\sigma_1}{\kappa(0)A_2}+a_3)} 
      = A_2, \\
      \gamma_1(0) 
      &\stackrel{\eqref{eq:explicit_formula_orbitlike_1}}{=} \frac{bW(0)+b_3}{aW(0)+a_3} \\ 
      &= A_2\kappa(0)(bW(0)+b_3) \\
      &\stackrel{\eqref{eq:DirProblem_choice_b1b2}}{=} 
        A_2\kappa(0) \Big( \frac{b_3}{a_3} aW(0)+ \frac{1}{\sqrt\mu a_3}
      a^\perp W(0)+b_3\Big) \\
      &\stackrel{\eqref{eq:aW(0)}}{=} \frac{b_3}{a_3} \sigma_1 + \frac{1}{\sqrt\mu a_3} \sigma_2 +
      b_3A_2\kappa(0)  \\
      &= A_1.
    \end{align*}
    The same way we find $\sigma_3^2+\sigma_4^2 = (\kappa(L)^2-\mu)B_2^2a_3^2$ and deduce
    \begin{align*}
      aW(L)+ia^\perp W(L) 
      &= (a_2+ia_1)(W_2(L)-iW_1(L)) \\
      &\stackrel{\eqref{eq:DirProblem_W_2W_1}}{=} (a_2+ia_1)(W_2(0)-iW_1(0))
      \frac{\kappa(0)\sqrt{\kappa(L)^2-\mu}}{\kappa(L)\sqrt{\kappa(0)^2-\mu}} e^{i(\vartheta(L)-\vartheta(0))} \\
      &\stackrel{\eqref{Gl:DirProblem_II}}{=} \big(aW(0)+ia^\perp W(0)\big)
      \frac{\kappa(0)\sqrt{\kappa(L)^2-\mu}}{\sqrt{\kappa(0)^2-\mu}\kappa(L)}
      \frac{(\sigma_1-i\sigma_2)(\sigma_3+i\sigma_4)}{\sqrt{\sigma_1^2+\sigma_2^2}\sqrt{\sigma_3^2+\sigma_4^2}}
      \\
      &\stackrel{\eqref{eq:aW(0)}}{=} 
      \frac{\sqrt{\kappa(L)^2-\mu}(\sigma_1+i\sigma_2)}{A_2\kappa(L)\sqrt{\kappa(0)^2-\mu}}
      \frac{(\sigma_1-i\sigma_2)(\sigma_3+i\sigma_4)}{\sqrt{\sigma_1^2+\sigma_2^2}\sqrt{\sigma_3^2+\sigma_4^2}}
      \\
      &=   \frac{\sqrt{\kappa(L)^2-\mu}}{A_2\kappa(L)\sqrt{\kappa(0)^2-\mu}}
      \frac{\sqrt{\sigma_1^2+\sigma_2^2}(\sigma_3+i\sigma_4)}{\sqrt{\sigma_3^2+\sigma_4^2}}   \\
      &\stackrel{\eqref{eq:sigma12_sq}}{=}   \frac{\sigma_3+i\sigma_4}{B_2\kappa(L)}.
    \end{align*}
    As above we deduce $\gamma_1(L)=B_1,\gamma_2(L)=B_2$ so that it remains to show
    $e^{i\phi(0)}=e^{i\phi_A},e^{i\phi(L)}=e^{i\phi_B}$. To this end we prove the formulas
    \eqref{eq:DirProblem_Formula2a}--\eqref{eq:DirProblem_Formula4b} for $\kappa(0),\ldots,2\kappa'(L)$ from
    the ''only if''-part. First we note that the formula for $\kappa(L)$ from
    \eqref{eq:DirProblem_Formula2b} holds by our definition of $L$ from \eqref{Gl:DirProblem_III}. Then the
    formula for $\kappa(L)^2$ from \eqref{eq:DirProblem_Formula3b} is a direct consequence of assumption
    \eqref{Gl:DirProblem_I} and the formula for $2\kappa'(L)$ from \eqref{eq:DirProblem_Formula4b} follows
    from $\sign(\kappa'(L))=\sigma$ and 
    \begin{align*}
      4\kappa'(L)^2
      &= 4\kappa(L)^2-4\mu -\kappa(L)^4 \\
      &= 4\cdot \frac{1+\mu(B_2^2a_3^2+(a_3B_1-b_3)^2)}{2B_2a_3}-4\mu \\
      &\;-   \Big(2\mu(a_3B_1-b_3)\sin(\phi_B)+ \frac{1+\mu(-a_3^2B_2^2+
        (a_3B_1-b_3)^2)}{B_2a_3}\cos(\phi_B)\Big)^2 \\
      &= 4\Big(\mu(a_3B_1-b_3)\cos(\phi_B)
        -\frac{1+\mu(-a_3^2B_2^2+ (a_3B_1-b_3)^2)}{B_2a_3} \sin(\phi_B)\Big)^2.
    \end{align*}
    The formulas for $\kappa(0),\kappa(0)^2,2\kappa'(0)$ follow from the definition of $a_3,b_3$ and
    $4(\kappa')^2=4\kappa^2-\kappa^4-4\mu$. With these formulas we
    deduce $e^{i\phi(0)}=e^{i\phi_A},e^{i\phi(L)}=e^{i\phi_B}$
    from~\eqref{eq:b3a3_1},\eqref{eq:b3a3_2}, which is all what remained to be shown. 
  \end{proof}
  
  We remark that the dependency of the equation \eqref{Gl:DirProblem_II} with respect to $l$ can be made
  ''more explicit'' because of $\vartheta(s+2l\sqrt{2-k^2}K(k))=\vartheta(s)+l\Delta\theta_k$ for all
  $s\in\R$, see Proposition~\ref{prop:properties_vartheta_orbitlike}~(ii). In the following corollary we
  apply the above theorem to the symmetric Dirichlet  boundary value problem for positively oriented orbitlike
  elasticae from Theorem~\ref{thm:Symmetry}. We recall that positive orientation refers to the fact that we
  only discuss orbitlike elasticae with positive curvature function.

  \begin{cor}\label{cor:Symmetry}
    Assume $A_1=-B_1\neq 0, A_2=B_2>0,\phi_A+\phi_B\in 2\pi\Z$ such that
    $$
      \pm\big(\frac{A_2}{A_1}\sin(\phi_A) + \cos(\phi_A)\big) \notin (0,2).
    $$
    Then all positively $(\pm=+)$ respectively negatively $(\pm=-)$ oriented orbitlike solutions
    of~\eqref{eq:DirichletBVP} are symmetric.
  \end{cor}
  \begin{proof}
    We first reduce the problem to the case of positively oriented solutions. So let $\gamma$ be a
  negatively oriented solution of~\eqref{eq:DirichletBVP} with $-\big( \frac{A_2}{A_1}\sin(\phi_A) +
    \cos(\phi_A)\big) \notin (0,2)$. Let $-\kappa$ be its curvature function and $\phi$ its associated angle
    function via~\eqref{eq:ODE_solution}. Then $\tilde\gamma:=(-\gamma_1,\gamma_2)$ is a
   positively oriented solution of the same boundary value problem with 
  $(A_1,A_2,\phi_A)$ replaced by $(-A_1,A_2,\pi-\phi_A)$ because the curvature of $\tilde\gamma$ is $\kappa$
  and its angle function $\tilde\phi$ satisfies $\phi+\tilde\phi\in \pi(2\Z+1)$, see~\eqref{eq:ODE_solution}.
  In particular the initial data of $\tilde\gamma,\tilde\phi$ satisfy 
    $$
      \frac{\tilde\gamma_2(0)}{\tilde\gamma_1(0)} \sin(\tilde\phi(0)) + \cos(\tilde\phi(0))
      = \frac{A_2}{-A_1} \sin(\pi-\phi_A) + \cos(\pi-\phi_A)
      = - \big(\frac{A_2}{A_1} \sin(\phi_A) + \cos(\phi_A)\big)
      \notin (0,2)
    $$
    as well as $(\pi-\phi_A)+(\pi-\phi_B)\in 2\pi\Z$. 
    So once we have shown that the symmetry result is true for positively oriented orbitlike solutions, we
    also obtain that $\tilde\gamma$ and hence $\gamma$ is symmetric. So it remains to prove the result for
    positively oriented orbitlike solutions. 
    
    \medskip
    
     Let $\gamma$ be a positively oriented solution of the given Dirichlet problem so  that
    $a_1,\ldots,b_3,L$ satisfy \eqref{eq:DirProblem_choice_a3}--\eqref{Gl:DirProblem_II} by 
    Theorem~\ref{thm:DirProblemI}. We intend to show that $\gamma$ is symmetric about the point $s=L/2$ where
    $L$ denotes the hyperbolic length of the elastica. This will be done in the following steps: First we show
    $b_3=0$, then we verify that
    $\gamma_1(\tfrac{L}{2})=\sin(\phi(\tfrac{L}{2}))=0$ and that $\kappa$ is symmetric about
    $s=\tfrac{L}{2}$. Invoking the unique solvability of the ODE system~\eqref{eq:ODE_kappa} with such
    initial data we then deduce 
    $$
      \gamma_1(L/2+s)=-\gamma_1(L/2-s),\quad
      \gamma_2(L/2+s)=\gamma_2(L/2-s),\quad
      \phi(L/2+s)+\phi(L/2-s)\in 2\pi\Z
    $$
    and the proof is finished. 
    
    \medskip
    
    So we first prove $b_3=0$. Using $A_1=-B_1,A_2=B_2,\phi_A+\phi_B\in 2\pi\Z$ and the formulas for
    $\kappa(0),\kappa(0)^2,\kappa(L),\kappa(L)^2$ from 
    \eqref{eq:DirProblem_Formula2a},\eqref{eq:DirProblem_Formula2b},\eqref{eq:DirProblem_Formula3a},\eqref{eq:DirProblem_Formula3b}
    we deduce
    $$
      \kappa(0)-\kappa(L)=-2\mu b_3 \frac{A_1}{A_2},\qquad
      \kappa(0)^2-\kappa(L)^2=-4\mu b_3 \big( \sin(\phi_A)+\frac{A_1}{A_2}\cos(\phi_A)\big)
    $$
    So if we had $b_3\neq 0$ then $0<\kappa\leq \|\kappa\|_\infty=\frac{2}{\sqrt{2-k^2}}<2$ would imply 
    $$
      (0,4) \ni \kappa(0)+\kappa(L) = 2\big( \frac{A_2}{A_1}\sin(\phi_A)+\cos(\phi_A)\big) \notin (0,4),
    $$
    a contradiction. This proves $b_3=0$. 
    
    \medskip
    
    From $b_3=0$ we get $\sigma_1=\sigma_3,\sigma_2=-\sigma_4$ by~\eqref{eq:sigmas} 
    as well as $\kappa(0)=\kappa(L),\kappa'(0)=-\kappa'(L)$,
    see~\eqref{eq:DirProblem_Formula2a}--\eqref{eq:DirProblem_Formula4b}. So there is an $l\in\Z$ such that 
    $$
      L = -2s_* + 2l\sqrt{2-k^2}K(k) 
    $$
    by the explicit form of the curvature function from~\eqref{eq:DirProblem_W_2W_1}. 
    In particular, $\kappa$ is symmetric about $\tfrac{L}{2}$ and our next aim is to show
    $\gamma_1(\tfrac{L}{2})=\sin(\phi(\tfrac{L}{2}))=0$. To this end we notice that \eqref{Gl:DirProblem_II}
    implies 
    \begin{equation}\label{eq:cor:angle_condition}
      e^{i(\vartheta(L)-\vartheta(0))} 
      = \frac{(\sigma_1-i\sigma_2)^2}{\sigma_1^2+\sigma_2^2}, \quad\text{hence }
      \frac{\sigma_1-i\sigma_2}{\sqrt{\sigma_1^2+\sigma_2^2}}
      = \pm e^{i(\vartheta(L)-\vartheta(0))/2}
      = \pm e^{i(\vartheta(L/2)-\vartheta(0))}.
    \end{equation}
    Indeed, the formula for $L$ and Proposition~\ref{prop:properties_vartheta_orbitlike} imply
    $$
      \vartheta(L)-\vartheta(0)
       = \vartheta(-2s_*)+l\Delta\theta_k-\vartheta(0) 
       = (l-1)\Delta\theta_k-2\vartheta(0) 
      = 2  \big( \vartheta(\tfrac{L}{2}) -\vartheta(0)\big).
    $$
    Additionally, by the choice of $a_1,a_2$ from~\eqref{eq:DirProblem_choice_a1a2} we get
    \begin{align*}
      a^\perp W(\tfrac{L}{2})
      &= \Re\Big( (a_1+ia_2)\overline{(W_2+iW_1)(\tfrac{L}{2})}\Big) \\
      &= a_3
      \Re\Big(
      \frac{\sigma_2+i\sigma_1}{\sqrt{\sigma_1^2+\sigma_2^2}}e^{i\vartheta(0)}\cdot
      \frac{\sqrt{\kappa(\tfrac{L}{2})^2-\mu}}{\kappa(\tfrac{L}{2})} 
       e^{i(-\vartheta(L/2)))} \Big)  \\
      &\stackrel{\eqref{eq:cor:angle_condition}}{=} \frac{ a_3\sqrt{\kappa(\tfrac{L}{2})^2-\mu}}{\kappa(\tfrac{L}{2})}
      \Re\big( \pm i e^{i\vartheta(L/2))}e^{-i\vartheta(L/2))} \big) \\
      &= 0.
    \end{align*}
    As a consequence, $b_3=0$ implies
    $$
      \gamma_1(\tfrac{L}{2})
      \stackrel{\eqref{eq:explicit_formula_orbitlike_1}}{=} 
        \frac{bW(\tfrac{L}{2})+b_3}{aW(\tfrac{L}{2})+a_3}
       \stackrel{\eqref{eq:a_1a_2a_3b_1b_2b_3}}{=} 
         \frac{1}{\sqrt\mu a_3} \frac{a^\perp
       W(\tfrac{L}{2})}{aW(\tfrac{L}{2})+a_3} = 0
    $$
    so that $\kappa'(\tfrac{L}{2})=0$ and~\eqref{eq:a1a2b1b2_2} yield 
    $$
      0 
      = \gamma_1(\tfrac{L}{2})
      = \frac{b_3}{a_3} +\frac{1}{2\mu a_3}
      \big(\kappa(\tfrac{L}{2})^2\sin(\phi(\tfrac{L}{2}))+2\kappa'(\tfrac{L}{2})\cos(\phi(\tfrac{L}{2}))\big)
      = \frac{\kappa(\tfrac{L}{2})^2\sin(\phi(\tfrac{L}{2}))}{2\mu a_3}, 
    $$ 
    whence $\sin(\phi(L/2))=0$. As outlined at the beginning  of the proof we conclude that the elastica is
    symmetric. 
  \end{proof}
 
  \begin{cor} \label{cor:Symmetry_breaking}
    Let $A_1=B_1=0,A_2=B_2>0,\phi_A=\phi_B=0$. Then all orbitlike solutions of the Dirichlet problem
    \eqref{eq:DirichletBVP2} are given by 
    $$
      k=k_{m,n},\; L=\mathcal{L}_{m,n},\; s_*\in\R \qquad\text{where }m,n\in\N \text{ satisfy
      }1<\frac{2m}{n}<\sqrt 2. 
    $$
    These solutions are symmetric if and only if $s_*\in \sqrt{2-k^2}K(k)\Z$. In particular, there are
    uncountably many nonsymmetric solutions. 
  \end{cor}
  \begin{proof}
    Let the coefficients $a_1,\ldots,b_3$ and $\sigma,L$ be given as in the theorem so that it remains to
    identify all $s_*\in\R,k\in (0,1),l\in\Z$ such that
    \eqref{Gl:DirProblem_I},\eqref{Gl:DirProblem_II} are satisfied.
    Due to $\sin(\phi_B)=0,\cos(\phi_B)=1$ and
    the formulas for $\kappa(L),\kappa(L)^2$ from
    \eqref{eq:DirProblem_Formula2b},\eqref{eq:DirProblem_Formula3b} the equation \eqref{Gl:DirProblem_I} is
    satisfied for any choice of $s_*,k,l$. For our boundary data we have $\sigma_1=\sigma_3,\sigma_2=\sigma_4=-\sqrt\mu b_3$
    so that equation \eqref{Gl:DirProblem_II} is equivalent to $\vartheta(L)-\vartheta(0)= 2m\pi$ for some
    $m\in\N$. Moreover, the formulas
    \eqref{eq:DirProblem_Formula2a},\eqref{eq:DirProblem_Formula2b},\eqref{eq:DirProblem_Formula4a},\eqref{eq:DirProblem_Formula4b}
    imply $\kappa(0)=\kappa(L),\kappa'(0)=\kappa'(L)$ and thus $L=2n\sqrt{2-k^2}K(k)$ for some $n\in\N$.
    So we deduce from Proposition~\ref{prop:properties_vartheta_orbitlike}~(i)
    $$
      2m\pi = \vartheta(L)-\vartheta(0)=n\Delta\theta_k,\quad
      \text{hence } 1<\frac{2m}{n}<\sqrt 2,\; k=k_{m,n}, L=\mathcal{L}_{m,n}.
    $$
    (In other words, the elastica closes up at $s=L$, see section~\ref{subsec:closed_elasticae}.)
    This yields the above characterization of the solutions. Moreover, as in the previous theorem one 
    checks that these solutions are symmetric if and only if $\kappa$ is symmetric about $s=\frac{L}{2}$,
    which is true precisely when $s_*\in\sqrt{2-k^2}K(k)\Z$. 
    Since this is a countable set of real numbers, we obtain uncountably many
    solutions for any choice $s_*\notin \sqrt{2-k^2}K(k)\Z$, see~Figure~\ref{Fig:Symmetry_breaking}.
  \end{proof}

  \begin{figure}[!htb]  
    \centering
    \subfigure[$s_*=0.2,(n,m)=(3,2),k=k_{m,n}$ 
    ]{
      \includegraphics[scale=.38]{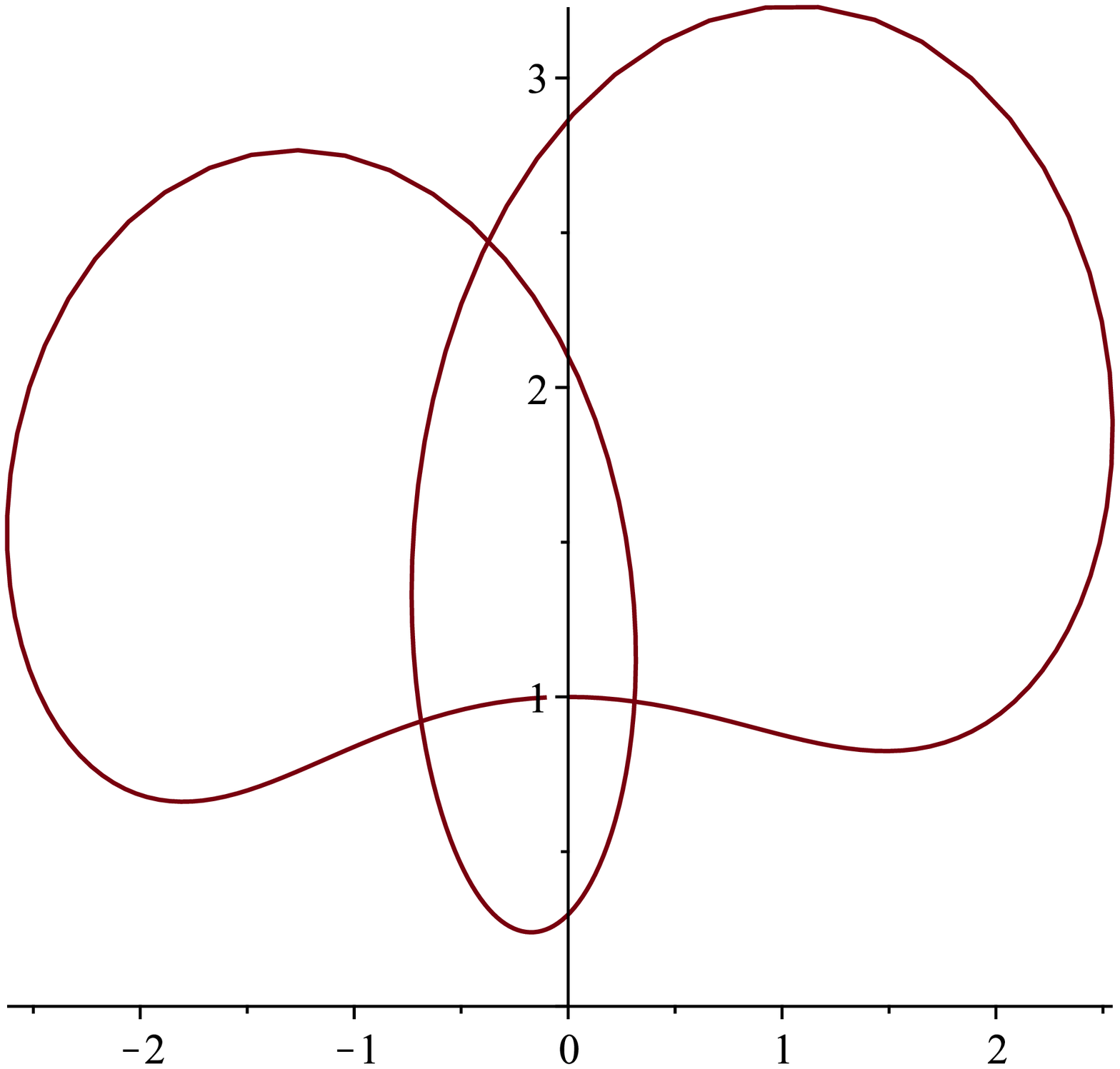} 
    }
    \;
    \subfigure[$s_*=-0.5,(n,m)=(5,3),k=k_{m,n}$ 
    ]{ 
      \includegraphics[scale=.38]{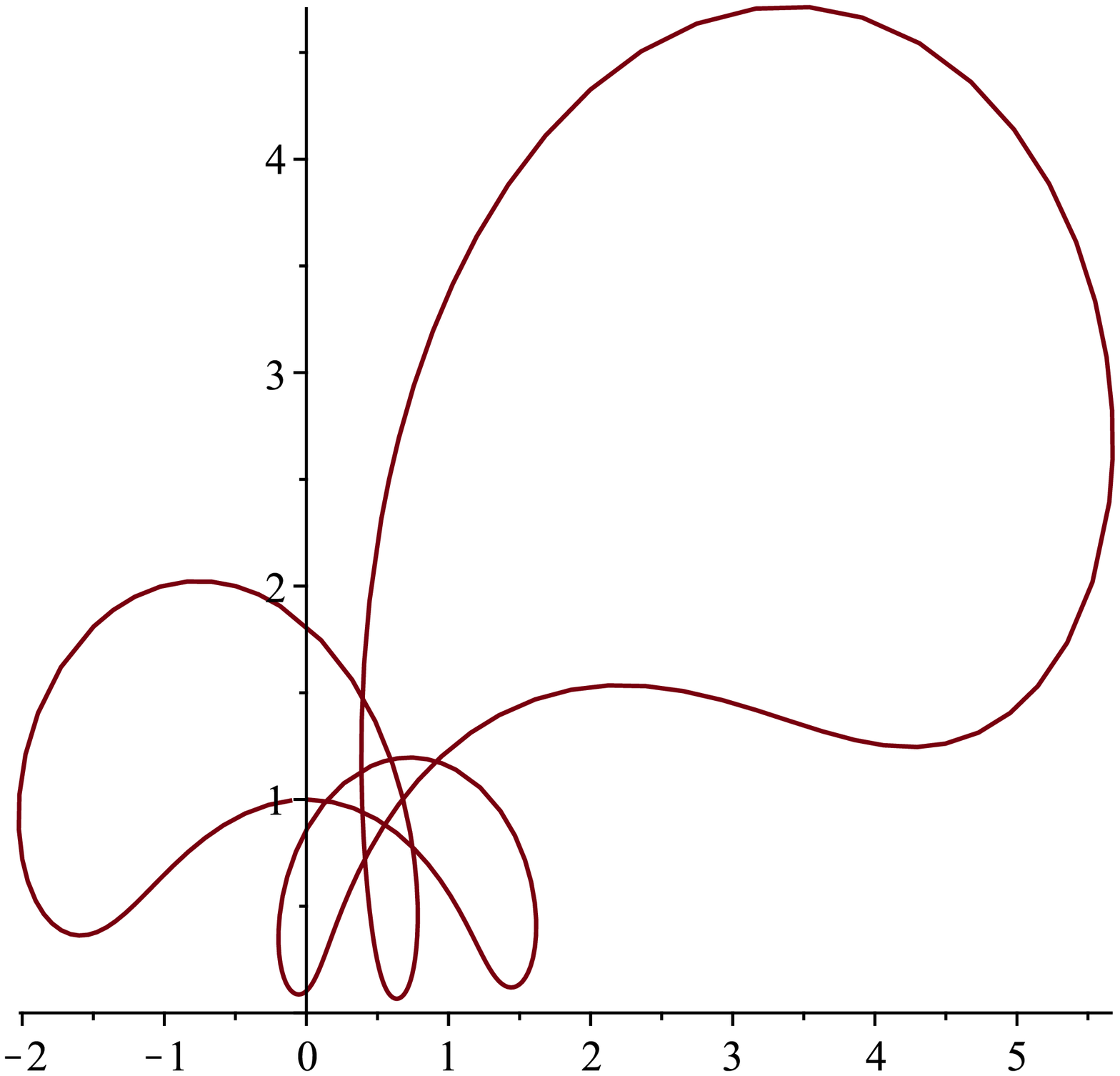} 
    } 
    \caption{Symmetry breaking closed elasticae in $\Hh$}
    \label{Fig:Symmetry_breaking}
  \end{figure}

  \section{Appendix}
    
    \subsection{Proof of Proposition \ref{prop:w1w2_orbitlike}, Proposition~\ref{prop:W1W2_orbitlike} and
    Proposition~\ref{prop:properties_vartheta_orbitlike}} \label{subsec:proof_propositions_orbitlike}
   
   In order to prove Proposition \ref{prop:w1w2_orbitlike} we have to show that for all $k\in (0,1)$ the
   ODE $w''+2\dn^2(\cdot,k)w=0$ has a fundamental system $\{w_1,w_2\}$ on $\R$ such that
     \begin{itemize}
       \item[(i)] $w_1$ is odd about 0, $w_2$ is even about 0, $w_2(0)>0>w_1'(0)$, 
       \item[(ii)] $w_1(s)^2+w_2(s)^2=2-k^2 -\dn^2(s,k)$ on $\R$,
       \item[(iii)] $w_1(s)w_1'(s)+w_2(s)w_2'(s) = k^2\dn(s,k)\cn(s,k)\sn(s,k)$ on $\R$,
       \item[(iv)] $w_1'(s)^2+w_2'(s)^2=1-k^2+\dn^4(s,k)$ on $\R$,
       \item[(v)] $w_1(s)w_2'(s)-w_2(s)w_1'(s) = \sqrt{1-k^2}\sqrt{2-k^2}$ on $\R$, 
       \item[(vi)] $(w_1+iw_2)(s+2lK(k))= (w_1+iw_2)(s)e^{il\Delta\theta_k}$ for
       all $s\in\R,l\in\Z$.
       \item[(vii)] $(w_1+iw_2)(-K(k))= e^{i(\pi-\Delta\theta_k)/2}$.
     \end{itemize}
    
   \medskip
   
   In the proof we make use of the special functions introduced in section~\ref{sec:Jacobi}.  
   From \cite{Whit_modern_analysis}, p.570--575 we obtain that every solution of $w''+2\dn^2(\cdot,k)w=0$ 
   is a linear combination of the so-called Halphen-Hermite solutions $w_+,w_-:\R\to\C$ (first discovered in
   \cite{Her_extrait,Hal_fonc_ell}) which are given by
    \begin{equation} \label{eq:Hermite_Halphen}
       w_\pm(s) = \frac{H(s\mp i\alpha,k)}{\Theta(s,k)}e^{s\zeta(\pm i\alpha,k)}\quad\text{for }
       \alpha\in\C \text{ s.t.} \dn^2(i\alpha,k)=2-k^2.
   \end{equation}
   Here, the functions $H,\Theta,\zeta:\C\times(0,1)\to\C$ denote Jacobi's eta,
    theta and zeta functions from section~\ref{sec:Jacobi}. 
    By Jacobi's imaginary transformation (see 161.01 in \cite{ByrFri_handbook}) we have  
    $$
      \dn(i\alpha,k)^2 
      = \frac{\dn^2(\alpha,k')}{\cn^2(\alpha,k')}
      \stackrel{\eqref{eq:Jacobi_trigonometric_identities}}{=}
      \frac{1-(k')^2\sn^2(\alpha,k')}{1-\sn^2(\alpha,k')} 
      = 1 + \frac{k^2\sn^2(\alpha,k')}{1-\sn^2(\alpha,k')} \qquad (k':=\sqrt{1-k^2})
    $$
    so that one possible choice is $\alpha=\alpha_k$ where 
    $\sn(\alpha_k,k')=k'$, i.e. $\alpha_k:=F(k',k')$, see \eqref{eq:Jac_Inverse}. 
    From the series representation of these special functions in \eqref{eq:JacobianTheta_series} we deduce 
    $w_-=\bar{w}_+$ so that 141.01 and 1051.02 in \cite{ByrFri_handbook} imply that $\Re(w_\pm)$ is odd and
    $\Im(w_\pm)$ is even. We then define $w_1,w_2:\R\to\R$ through the formula
    \begin{equation}\label{eq:def_w1w2}
      w_1(s)+iw_2(s) := \frac{\sqrt{1-k^2}}{\Im(w_+(0))} w_+(s)
    \end{equation}
    so that (i) holds.
    
    \medskip
    
    By Proposition 6 in \cite{Val_Heun} for $m_1=m_2=m_3=0,N=1,m_0=1,s=\frac{1}{2},\sigma=2$ the product of
    the two Hermite-Halphen solutions, i.e. $w_+w_-=|w_+|^2$ and thus $w_1^2+w_2^2$, is a real multiple of
    the function $2-k^2-\dn^2(\cdot,k)$. Hence, \eqref{eq:def_w1w2} and the oddness
    of $w_1$ yield $w_1(0)=0,w_2(0)=\sqrt{1-k^2}$ so that (ii) holds. Differentiating (ii) and
    using \eqref{eq:Jacobi_ODE_1terOrdnung},\eqref{eq:Jacobi_ODE_2terOrdnung}
    we get (iii) and (iv). Part (v) follows from the constancy of the Wronskian and
    $w_1(0)=w_2'(0)=0,w_2(0)=\sqrt{1-k^2},w_1'(0)=-\sqrt{2-k^2}$ by (i),(ii),(iv).
     
     \medskip
     
     We now prove (vi) and (vii). Using 1051.04 in \cite{ByrFri_handbook} we get 
     \begin{align*}
       w_+(s+2lK(k))
       &= \frac{H(s+2lK(k)-i\alpha_k,k)}{\Theta(s+2lK(k),k)}e^{(s+2lK(k))\zeta(i\alpha_k,k)} \\
       &= \frac{(-1)^l H(s-i\alpha_k,k)}{\Theta(s,k)}e^{s\zeta(i\alpha_k,k)} e^{2lK(k)\zeta(i\alpha_k,k)}
       \\
       &= w_+(s) e^{l(2K(k)\zeta(i\alpha_k,k)+\pi i)}. 
     \end{align*}
     Using various identities from \cite{ByrFri_handbook} (see below for details) we find
    \begin{align*}
      &\;  K(k)\zeta(i\alpha_k,k) \\
      &= K(k)i\cdot \Big( -\zeta(\alpha_k,k') -  \frac{\pi  \alpha_k}{2K(k)K'(k)}
      +\frac{\sn(\alpha_k,k')\dn(\alpha_k,k')}{\cn(\alpha_k,k')}  \Big)    \\
      &= K(k)i\cdot \Big( - \zeta(\alpha_k,k') - \frac{\pi F(k',k')}{2K(k)K'(k)} 
       + \sqrt{1-k^2}\sqrt{2-k^2}  \Big) \\
     &=    K(k)i\cdot \Big( - E(k',k') + \frac{E(k')}{K(k')}F(k',k') - \frac{\pi F(k',k')}{2K(k)K'(k)} 
       + \sqrt{1-k^2}\sqrt{2-k^2}  \Big)    \\
     &=   i\cdot \Big( - E(k',k')K(k) + \frac{F(k',k')}{K(k')}(E(k')K(k)-\frac{\pi}{2})   +
     \sqrt{1-k^2}\sqrt{2-k^2}K(k)  \Big)    \\
     &=    i\cdot \Big( - E(k',k')K(k) + \frac{F(k',k')}{K(k')}(K(k')(K(k)-E(k)))   +
     \sqrt{1-k^2}\sqrt{2-k^2}K(k)  \Big)   \\
     &=   i\cdot \Big( - (E(k',k')K(k) +  F(k',k')(E(k)-K(k)))   +
     \sqrt{1-k^2}\sqrt{2-k^2}K(k)  \Big)   \\
     &=    i\cdot \Big( - \frac{\pi}{2}\Lambda_0(\arcsin(k'),k)  + \sqrt{1-k^2}\sqrt{2-k^2}K(k) 
     \Big)   \\
     &= i\cdot \frac{1}{2}(\Delta\theta_k-\pi).
   \end{align*}
   In the first equality we used 143.02, in the second we used 
   $\sn(\alpha_k,k')=k'$, in the third 130.02 and  140.01, then 110.10 in the fifth and
    150.03 in the seventh and the eigth equality results from the definition of $\Delta\theta_k$,
    see~\eqref{eq:def_rotation}. This proves (vi). 
    Finally, we have
    \begin{align*}
      (w_1+iw_2)(-K(k)) 
      \stackrel{\eqref{eq:Hermite_Halphen}}{=} 
        e^{-K(k)\zeta(i\alpha_k,k)} 
      =  e^{i(\pi-\Delta\theta_k)/2} 
    \end{align*}  
    because of $(w_1^2+w_2^2)(-K(k))=1$ by Proposition~\ref{prop:w1w2_orbitlike}~(ii) and because 
    $H(-K(k)-i\alpha_k,k)$, $\Theta(-K(k),k)$ are real and positive for all $k\in (0,1)$ by
    \eqref{eq:JacobianTheta_series}.   \qed
    
   \medskip

   Next we prove Proposition~\ref{prop:W1W2_orbitlike}, i.e. that the functions $W_1,W_2$ given by 
   $$
     W_j(s)= \frac{1}{\sqrt{2-k^2}} w_j(-K(k)+\frac{s+s_*}{\sqrt{2-k^2}}) \quad (j=1,2)
   $$
   satisfy the following identities:
   \begin{itemize}
       \item[(i)] $W_1^2+W_2^2 = \frac{\kappa^2-\mu}{\kappa^2}$,
       \item[(ii)] $W_1'= \frac{\mu \kappa'}{\kappa(\kappa^2-\mu)} W_1
      - \frac{\sqrt\mu \kappa^2}{2(\kappa^2-\mu)}W_2 $,
       \item[(iii)] $W_2' = \frac{\sqrt\mu \kappa^2}{2(\kappa^2-\mu)}W_1 + \frac{\mu
       \kappa'}{\kappa(\kappa^2-\mu)} W_2$,
       \item[(iv)] $W_1W_2'-W_2W_1' = \frac{\sqrt\mu}{2}$,
       \item[(v)] $(W_1+iW_2)(s+2l\sqrt{2-k^2}K(k))=(W_1+iW_2)(s)e^{il\Delta\theta_k}$ for $s\in\R,l\in\Z$,
       \item[(vi)] $(W_1+iW_2)(s)= (-W_1+iW_2)(-s-2s_*+2l\sqrt{2-k^2}K(k)) e^{i(l-1)\Delta\theta_k}$
       for $s\in\R,l\in\Z$.
  \end{itemize}
  
  \medskip

   For the proof we set $z:=-K(k)+ \frac{s+s_*}{\sqrt{2-k^2}}$ so that
      \eqref{eq:def_kappa_orbitlike},\eqref{eq:def_Wj} give
     \begin{align} \label{eq:formulas_z_vs_s}
       \begin{aligned}
       w_j(z)= \sqrt{2-k^2}W_j(s),&\quad w_j'(z)= (2-k^2)W_j'(s) \quad (j=1,2)\\
       \kappa(s) &= \frac{2}{\sqrt{2-k^2}}\dn(z+K(k),k),\\
       \kappa'(s) &= -\frac{2k^2}{2-k^2}\sn(z+K(k),k)\cn(z+K(k),k).
       \end{aligned}
     \end{align}
     Using the identities from \eqref{eq:Jac_Additionstheoreme} we get
     \begin{align} \label{eq:formulas_dnsncn}
       \begin{aligned}
       \dn(z,k)
       &= \frac{\sqrt{1-k^2}}{\dn(z+K(k),k)}
       \stackrel{\eqref{eq:formulas_z_vs_s}}{=}  \frac{2\sqrt{1-k^2}}{\sqrt{2-k^2}} \frac{1}{\kappa(s)}, \\
       \sn(z,k)\cn(z,k)
       &= -  \frac{\sqrt{1-k^2} \sn(z+K(k),k)\cn(z+K(k),k)}{\dn^2(z+K(k),k)}
       \stackrel{\eqref{eq:formulas_z_vs_s}}{=} \frac{2\sqrt{1-k^2}}{k^2} \frac{\kappa'(s)}{\kappa(s)^2}.
       \end{aligned}
     \end{align}
     In particular, using that $w_1,w_2$ are linearly independent by Proposition~\ref{prop:W1W2_orbitlike} (i)
     as well as
     \begin{align*}
       W_j''(s)
       = \frac{1}{(2-k^2)^{3/2}}w_j''(z)
       = - \frac{2}{(2-k^2)^{3/2}} \dn^2(z,k) w_j(z)
       \stackrel{\eqref{eq:formulas_dnsncn}}{=} - 2\mu\kappa(s)^{-2} W_j(s)
       \quad\text{for $j=1,2$},
     \end{align*}
     we get that $\{W_1,W_2\}$ is a fundamental system of the given ODE.
     Now we exploit \eqref{eq:formulas_z_vs_s},\eqref{eq:formulas_dnsncn} and
     Proposition~\ref{prop:w1w2_orbitlike} in order to prove (i)-(vi).
     The first item follows from 
     \begin{align*}
       W_1(s)^2+W_2(s)^2
       &\stackrel{\eqref{eq:formulas_z_vs_s}}{=} \frac{w_1(z)^2+w_2(z)^2}{2-k^2}
       \stackrel{\ref{prop:w1w2_orbitlike}.(ii)}{=} 1-\frac{\dn^2(z,k)}{2-k^2}
       \stackrel{\eqref{eq:formulas_dnsncn}}{=}
       \frac{\kappa(s)^2-\mu}{\kappa(s)^2}. \\
     \intertext{The assertion (ii) is true because of}
       W_1'(s)
       &\stackrel{\eqref{eq:formulas_z_vs_s}}{=} \frac{1}{2-k^2} w_1'(z) \\
       &= \frac{1}{2-k^2}
       \frac{w_1(z)(w_1w_1'+w_2w_2')(z)-w_2(z)(w_1w_2'-w_2w_1')(z)}{(w_1^2+w_2^2)(z)} \\
       &\hspace{-1.2cm}\stackrel{\ref{prop:w1w2_orbitlike}.(ii)-(iv)}{=} \frac{1}{(2-k^2)^{3/2}}
         \frac{k^2\dn(z,k)\cn(z,k)\sn(z,k) W_1(s)- \sqrt{1-k^2}\sqrt{2-k^2}W_2(s)}{W_1(s)^2+W_2(s)^2}
         \\
      &\stackrel{\eqref{eq:formulas_dnsncn}}{=}\frac{4(1-k^2)}{(2-k^2)^2} \frac{\kappa'(s)}{\kappa(s)^3}
      \frac{W_1(s)}{W_1(s)^2+W_2(s)^2} - \frac{\sqrt{1-k^2}}{2-k^2} \frac{W_2(s)}{W_1(s)^2+W_2(s)^2}  \\
      &\stackrel{\ref{prop:W1W2_orbitlike}.(i)}{=} \frac{\mu\kappa'(s)}{\kappa(s)(\kappa(s)^2-\mu)} W_1(s)
      - \frac{\sqrt\mu \kappa(s)^2}{2(\kappa(s)^2-\mu)}W_2(s). \\
     \intertext{Similarly, (iii) follows from}
       W_2'(s)
       &= \frac{1}{2-k^2}
       \frac{w_2(z)(w_1w_1'+w_2w_2')(z)+w_1(z)(w_1w_2'-w_2w_1')(z)}{(w_1^2+w_2^2)(z)} \\
       &= \frac{\sqrt\mu \kappa(s)^2}{2(\kappa(s)^2-\mu)}W_1(s) +
       \frac{\mu\kappa'(s)}{\kappa(s)(\kappa(s)^2-\mu)} W_2(s). \\
     \intertext{
     The items (iv),(v) follow directly from
     Proposition~\ref{prop:w1w2_orbitlike}~(iv),(v).
     Finally, we set $t:= -s-2s_*+2l\sqrt{2-k^2}K(k)$ so that (vi) results from}
     W_1(s)+iW_2(s)
     &= (w_1+iw_2) \big( -K(k)+\frac{-(t+s_*)+2l\sqrt{2-k^2}K(k)}{\sqrt{2-k^2}} \big) \\
     &= (w_1+iw_2) \big( K(k)-\frac{t+s_*}{\sqrt{2-k^2}} + (2l-2)K(k)\big) \\
     &\stackrel{\ref{prop:w1w2_orbitlike}.(vi)}{=}  (w_1+iw_2)
     \big(K(k)-\frac{t+s_*}{\sqrt{2-k^2}}\big) e^{i(l-1)\Delta\theta_k}  \\
     &\stackrel{\ref{prop:w1w2_orbitlike}.(i)}{=} (-w_1+iw_2) \big(-K(k)+\frac{t+s_*}{\sqrt{2-k^2}}\big)
     e^{i(l-1)\Delta\theta_k}  \\
     &= (-W_1(t)+iW_2(t))  e^{i(l-1)\Delta\theta_k}.
   \end{align*} 
   \qed

  Next we prove the assertions from Proposition~\ref{prop:properties_vartheta_orbitlike}, namely  
    $$
      \vartheta' = \frac{\sqrt\mu \kappa^2}{2(\kappa^2-\mu)},\qquad
       \matII{W_1}{W_2}{W_1'}{W_2'}
      =  
      \matII{0}{\frac{\sqrt{\kappa^2-\mu}}{\kappa}}{\frac{-\sqrt{\mu}\kappa}{2\sqrt{\kappa^2-\mu}}}{
      \frac{\mu\kappa'}{\kappa^2\sqrt{\kappa^2-\mu}}}
      \matII{\cos(\vartheta)}{\sin(\vartheta)}{-\sin(\vartheta)}{\cos(\vartheta)}
    $$
    as well as 
    \begin{itemize}
	  \item[(i)] $\vartheta(-s_*+l\sqrt{2-k^2}K(k)) = \frac{l-1}{2}\Delta\theta_k$,
      \item[(ii)] $\vartheta(s+2l\sqrt{2-k^2}K(k))-\vartheta(s) = l\Delta\theta_k$,
      \item[(iii)] $\vartheta(-s_*+s)+\vartheta(-s_*-s) = -\Delta\theta_k$.
    \end{itemize}
   To this end we set $c(s):= \frac{\sqrt{\kappa(s)^2-\mu}}{\kappa(s)}$ so that
   the defining equation for $\vartheta$ \eqref{eq:def_theta} reads
   $$
      \vecII{W_1(s)}{W_2(s)} = c(s) \vecII{-\sin(\vartheta(s))}{\cos(\vartheta(s))}.
    $$
    Differentiating both sides and using Proposition~\ref{prop:W1W2_orbitlike}~(ii),(iii) we arrive at
    \begin{align*}
      &-c(s)\vartheta'(s) \vecII{\cos(\vartheta(s))}{\sin(\vartheta(s))}
      + c'(s) \vecII{-\sin(\vartheta(s))}{\cos(\vartheta(s))} \quad
      =\quad \vecII{W_1'(s)}{W_2'(s)} \\
      &= \frac{\mu\kappa'(s)}{\kappa(s)(\kappa(s)^2-\mu)} \vecII{W_1(s)}{W_2(s)}
      + \frac{\sqrt\mu\kappa(s)^2}{2(\kappa(s)^2-\mu)} \vecII{-W_2(s)}{W_1(s)}   \\
      &= \frac{\mu\kappa'(s)c(s)}{\kappa(s)(\kappa(s)^2-\mu)} \vecII{-\sin(\vartheta(s))}{\cos(\vartheta(s))}
      - \frac{\sqrt\mu\kappa(s)^2c(s)}{2(\kappa(s)^2-\mu)} \vecII{\cos(\vartheta(s))}{\sin(\vartheta(s))}.
    \end{align*}
    This identity gives
    $$
      \vartheta'(s) = \frac{\sqrt{\mu}\kappa(s)^2}{2(\kappa(s)^2-\mu)},\qquad
      c'(s)
      = \frac{\mu\kappa'(s)}{\kappa(s)(\kappa(s)^2-\mu)}c(s)
      =  \frac{\mu\kappa'(s)}{\kappa(s)^2\sqrt{\kappa(s)^2-\mu}} 
    $$
    and thus the first claim follows from
    \begin{align*}
      \matII{W_1(s)}{W_2(s)}{W_1'(s)}{W_2'(s)}
      &=
      \matII{0}{c(s)}{-c(s)\vartheta'(s)}{c'(s)}
      \matII{\cos(\vartheta(s))}{\sin(\vartheta(s))}{-\sin(\vartheta(s))}{\cos(\vartheta(s))}.
    \end{align*}
    Now we prove (i),(ii),(iii). For $l\in\Z$ we have
    \begin{align*}
       \vartheta(-s_*+ l\sqrt{2-k^2}K(k))-\vartheta(-s_*)
       &= \int_{-s_*}^{-s_*+l\sqrt{2-k^2}K(k)} \vartheta'(s) \,ds\\
       &= \int_{-s_*}^{-s_*+l\sqrt{2-k^2}K(k)} \frac{\sqrt{\mu}\kappa(s)^2}{2(\kappa(s)^2-\mu)} \,ds \\
       &= \int_0^{lK(k)}
       \frac{\sqrt{2-k^2}\sqrt{\mu}  \kappa(\sqrt{2-k^2}s-s_*)^2}{2(\kappa(\sqrt{2-k^2}s-s_*)^2-\mu)} \,ds
       \\
       &\stackrel{\eqref{eq:def_kappa_orbitlike}}{=}  \int_0^{lK(k)}
       \frac{\sqrt{1-k^2}\sqrt{2-k^2}\dn(s,k)^2}{(2-k^2)\dn(s,k)^2-1+k^2} \,ds   \\
       &=  \sqrt{1-k^2}\sqrt{2-k^2} \int_0^{lK(k)}
       \frac{\dn(s,k)^2}{(2-k^2)\dn(s,k)^2-1+k^2} \,ds   \\
       &= \frac{l}{2}\cdot 2\sqrt{2-k^2}\sqrt{1-k^2} \int_0^{K(k)}
       \frac{\dn(s,k)^2}{1-k^2(2-k^2)\sn(s,k)^2} \,ds.
     \end{align*}
     Here we used $\dn^2=1-k^2\sn^2$ and $\dn(2K(k)-s,k)=\dn(2K(k)+s,k)=\dn(s,k)$ in the last equation, 
     see~\eqref{eq:Jacobi_trigonometric_identities}. The proof of (i) is finished once we have proved that 
     the second factor equals $\Delta\theta_k$. Using the formula 412.04 from \cite{ByrFri_handbook} for
     $\alpha:=k\sqrt{2-k^2}>k$ we get
     \begin{align*}
       &\;2\sqrt{2-k^2}\sqrt{1-k^2} \int_0^{K(k)}
       \frac{ \dn(s,k)^2}{1-k^2(2-k^2)\sn(s,k)^2} \,ds \\
       &= 2\sqrt{2-k^2}\sqrt{1-k^2} \Big(  K(k) +
       \frac{\pi   \sqrt{\alpha^2-k^2}
       (1-\Lambda_0(\arcsin(\sqrt{(1-\alpha^2)(k')^{-2}}),k))}{2\alpha\sqrt{1-\alpha^2}} \Big)  \\
       &= 2\sqrt{2-k^2}\sqrt{1-k^2} \Big(
       K(k) + \frac{\pi(1-\Lambda_0(\arcsin(k'),k)) }{2\sqrt{1-k^2}\sqrt{2-k^2}} \Big)  \\
       &\stackrel{\eqref{eq:def_rotation}}{=} \Delta\theta_k.
     \end{align*}
     So (i) is proved  once we have shown $\vartheta(-s_*)=-\Delta\theta_k/2$. This identity, however, follows
     from $-2\pi<\vartheta(-s_*)\leq 0$ and 
     \begin{align*}
       e^{i\vartheta(-s_*)}
       \stackrel{\eqref{eq:def_theta}}{=} \frac{(W_2-iW_1)(-s_*)}{\sqrt{2-k^2}}
       \stackrel{\eqref{eq:def_Wj}}{=} (w_2-iw_1)(-K(k))
       = -i(w_1+iw_2)(-K(k))
       \stackrel{\ref{prop:w1w2_orbitlike}.(vii)}{=} e^{- i\Delta\theta_k /2}.
     \end{align*}
     Item (ii) holds thanks to the $2\sqrt{2-k^2}K(k)-$periodicity of $\vartheta'$ and (i), namely
     $$
        \vartheta(s+ 2l\sqrt{2-k^2}K(k))-\vartheta(s)
       = \int_{s}^{s+2l\sqrt{2-k^2}K(k)} \vartheta'(z) \,dz\\
       = \int_{-s_*}^{-s_*+2l\sqrt{2-k^2}K(k)}  \vartheta'(z) \,dz 
       = l\Delta\theta_k
     $$
     Finally, $\vartheta'(-s_*+s)=\vartheta'(-s_*-s)$ and $2\vartheta(-s_*)=-\Delta\theta_k$ imply~(iii).  
     \qed

   \subsection{Proof of Proposition \ref{prop:w1w2_wavelike} and
    Proposition~\ref{prop:properties_vartheta_wavelike}}  \label{subsec:proof_propositions_wavelike}
   

    We first show  that for all $k\in (0,1)$ the ordinary differential equation $w''-
    \frac{2(1-k^2)}{\cn^2(s,k)}w=0$ has a fundamental system $\{w_1,w_2\}$ on $(-K(k),K(k))$ such that
     \begin{itemize}
       \item[(i)] $w_1$ is odd about 0, $w_2$ is even about 0, $w_2(0)>0>w_1'(0)$,
       \item[(ii)] $\cn^2(s,k)(w_2(s)^2 - w_1(s)^2) = 1-k^2+(2k^2-1)\cn^2(s,k)$ on $(-K(k),K(k))$,
       \item[(iii)] $\cn^3(s,k)(w_2(s)w_2'(s) - w_1(s)w_1'(s)) = (1-k^2)\sn(s,k)\dn(s,k)$ on $(-K(k),K(k))$,
       \item[(iv)] $w_2(s)w_1'(s)-w_1(s)w_2'(s)= -k\sqrt{(1-k^2)(2k^2-1)}$ on $(-K(k),K(k))$, 
        \item[(v)] The functions $\cn(\cdot,k)w_1,\cn(\cdot,k)w_2$ have unique smooth extensions $\hat
       w_1,\hat w_2:\R\to\R$ such that the functions $w_1:=\frac{\hat
       w_1}{\cn(\cdot,k)},w_2:=\frac{\hat w_2}{\cn(\cdot,k)}$ build a fundamental system on
       each connected component of $\R\sm\{\kappa=0\}$. Moreover, $\hat w_2$ is positive on $\R$, $s\hat
       w_1(s)$ is negative on $\R\sm\{0\}$ and, as $s\to\pm\infty$, we have $\hat w_1(s)\to \mp\infty,\hat
       w_2(s)\to\infty$ exponentially and $\hat w_1(s)/\hat w_2(s)\to \mp 1$. 
     \end{itemize}
   
   \medskip
    
   The proof is similar to the orbitlike case. We define
   $$
    	w_\pm (s) 
    	:= \frac{\Theta_1(s\pm\alpha_k,k)}{\Theta(s,k)} e^{\mp s\zeta(\alpha_k,k)} 
    	\quad\text{where }\dn^2(\alpha_k,k)=k^2, \text{ i.e. }\alpha_k:= F(k'/k,k)
   $$
     for $\Theta,\Theta_1,\zeta$ as in section~\ref{sec:Jacobi}.
     Notice that $k\in (\frac{1}{\sqrt 2},1)$ implies $k'/k\in (0,1)$ and thus $\alpha_k$ is real. 
     From Proposition~6 in \cite{Val_Heun} (for
     parameters $m_1=0,m_2=1,m_3=0,m_0=0,A=2k^2,s=\frac{2k^2-1}{4},\sigma=2k^2-1$) we infer that  
	 \begin{equation}\label{eq:defn_w1w2_wavelike}
	    w_1(s):= \frac{k(w_+(s) - w_-(s))}{(w_+(0)+w_-(0))\cn(s,k)},\qquad 
	    w_2(s):= \frac{k(w_+(s) + w_-(s))}{(w_+(0)+w_-(0))\cn(s,k)}
	 \end{equation}
      is a fundamental system of the ODE on each connected component of $\R\sm \{\cn(\cdot,k)=0\}$, which
      satisfies (i) thanks to $w_\pm(-s)=w_\mp(s)$.
      The same proposition tells us that the product $w_+w_-$ is a real multiple of the function $1-k^2+(2k^2-1)\cn^2(\cdot,k)$ so that our choice for
      $w_1,w_2$ yields (ii). Differentiating this identity gives (iii) and (iv). Item (v) follows
      from the constancy of the Wronskian and $w_2(0)w_1'(0)-w_1(0)w_2'(0)= -k
      \sqrt{(1-k^2)(2k^2-1)}$. Finally, \eqref{eq:defn_w1w2_wavelike} shows
      that property (v) holds for the choice
      
      \begin{align*}
        \hat w_1(s)
        &:=  \frac{k(w_+(s) - w_-(s))}{w_+(0)+w_-(0)}
        =\frac{k(\Theta_1(s+\alpha_k,k)e^{-s\zeta(\alpha_k,k)}
        - \Theta_1(s-\alpha_k,k)e^{s\zeta(\alpha_k,k)}}{ 
        (w_+(0)+w_-(0))\Theta(s,k)}, \\
        \hat w_2(s)
        &:= \frac{k(w_+(s) + w_-(s))}{w_+(0)+w_-(0)}
        = \frac{k(\Theta_1(s+\alpha_k,k)e^{-s\zeta(\alpha_k,k)}
        + \Theta_1(s-\alpha_k,k)e^{s\zeta(\alpha_k,k)}}{ 
        (w_+(0)+w_-(0))\Theta(s,k)},
      \end{align*}
      since $\zeta(\alpha_k,k)>0$ and $\Theta_1(\cdot,k),\Theta(\cdot,k)$ are positive periodic functions.
      Notice that the all complex zeros of $\Theta_1,\Theta$ are known to be non-real, see 1051.07 in
      \cite{ByrFri_handbook}. 
      \qed
    
     \medskip
   
    Now we prove Proposition~\ref{prop:properties_vartheta_wavelike}. For $W_1,W_2,\vartheta$ as in
   \eqref{eq:def_Wj_wavelike},\eqref{eq:def_theta_wavelike} we have to prove that the following identities
   hold on $\R\sm\{\kappa=0\}$:
    $$
      \vartheta'
       = \frac{\sqrt{|\mu|} \kappa^2}{2(\kappa^2-\mu)},\quad
       \matII{W_1}{W_2}{W_1'}{W_2'}
      =  
      \matII{0}{\frac{\sqrt{\kappa^2-\mu}}{\kappa}}{-\frac{\sqrt{|\mu|}\kappa}{2\sqrt{\kappa^2-\mu}}}{
      \frac{\mu\kappa'}{\kappa^2\sqrt{\kappa^2-\mu}}}
      \matII{\cosh(\vartheta)}{-\sinh(\vartheta)}{-\sinh(\vartheta)}{\cosh(\vartheta)}.
    $$
    
    \medskip

    We proceed as in the orbitlike case. With $c(s):=
    \frac{\sqrt{\kappa(s)^2-\mu}}{\kappa(s)}$ we may write 
    $$
      \vecII{W_1(s)}{W_2(s)} = c(s) \vecII{-\sinh(\vartheta(s))}{\cosh(\vartheta(s))}.
    $$
    Differentiating both sides and using Proposition~\ref{prop:W1W2_wavelike}~(ii),(iii) we arrive at
    \begin{align*}
      &-c(s)\vartheta'(s) \vecII{\cosh(\vartheta(s))}{-\sinh(\vartheta(s))}
      + c'(s) \vecII{-\sinh(\vartheta(s))}{\cosh(\vartheta(s))} \;
      =\; \vecII{W_1'(s)}{W_2'(s)} \\
      &= \frac{\mu\kappa'(s)}{\kappa(s)(\kappa(s)^2-\mu)} \vecII{W_1(s)}{W_2(s)}
      - \frac{\sqrt{|\mu|}\kappa(s)^2}{2(\kappa(s)^2-\mu)} \vecII{W_2(s)}{W_1(s)}   \\
      &= \frac{\mu\kappa'(s)c(s)}{\kappa(s)(\kappa(s)^2-\mu)}
      \vecII{-\sinh(\vartheta(s))}{\cosh(\vartheta(s))} 
      - \frac{\sqrt{|\mu|}\kappa(s)^2c(s)}{2(\kappa(s)^2-\mu)}
      \vecII{\cosh(\vartheta(s))}{-\sinh(\vartheta(s))}.
    \end{align*}
    This identity gives
    $$
      \vartheta'(s) = \frac{\sqrt{|\mu|}\kappa(s)^2}{2(\kappa(s)^2-\mu)},\qquad
      c'(s)
      = \frac{\mu\kappa'(s)c(s)}{\kappa(s)(\kappa(s)^2-\mu)} 
      =  \frac{\mu\kappa'(s)}{\kappa(s)^2\sqrt{\kappa(s)^2-\mu}}
    $$
    and the claim follows from
    \begin{align*}
      \matII{W_1}{W_2}{W_1'}{W_2'}
      &= 
       \matII{0}{c}{-c\vartheta'}{c'}\matII{\cosh(\vartheta)}{
       -\sinh(\vartheta)}{-\sinh(\vartheta)}{\cosh(\vartheta)}.
    \end{align*}
    \qed

  \section*{Acknowledgements} 
    The author would like to thank the Scuola Normale Superiore di Pisa for the hospitality during his
    stay there. The work on this project was supported by the Deutsche Forschungsgemeinschaft (DFG, German
    Research Foundation) - grant number {MA~6290/2-1}.

\medskip
  
\bibliographystyle{plain}
\bibliography{doc}

\end{document}